\documentclass[a4paper,11pt]{amsart}

\usepackage{amssymb}
\usepackage{amscd}
\usepackage{amsthm}
\usepackage{amsmath}
\usepackage[all]{xy}
\usepackage{extarrows}
\usepackage{rotating}
\usepackage{tikz-cd}
\usepackage{quiver}
\usepackage{leftidx}
\usepackage{bm}
\usepackage{dynkin-diagrams}
\usepackage{mathrsfs}
\usepackage{rotating}
\usepackage{upgreek}
\usepackage{thmtools}

\declaretheoremstyle[%
spaceabove=-6pt,%
spacebelow=6pt,%
postheadspace=0.5em,%
qed=\qedsymbol,%
headpunct={}
]{mystyle} 
\declaretheorem[name={Proof},style=mystyle,unnumbered,
]{Proof}

\usepackage{hyperref}

\tikzset{
	symbol/.style={
		draw=none,
		every to/.append style={
			edge node={node [sloped, allow upside down, auto=false]{$#1$}}}
	}
}

\DeclareUnicodeCharacter{2212}{-}
\DeclareUnicodeCharacter{0008}{~}
\DeclareUnicodeCharacter{2011}{~}

\renewenvironment{proof}{{ \textbf{Proof}.}}{\qed}

\newtheorem{Thm}{Theorem}[section]
\newtheorem{Lem}[Thm]{Lemma}
\newtheorem{lemma}[Thm]{Lemma}
\newtheorem{Def}[Thm]{Definition}
\newtheorem{Cor}[Thm]{Corollary}
\newtheorem{Prop}[Thm]{Proposition}
\newtheorem{Ex1}[Thm]{Example}
\newtheorem{Rem1}[Thm]{Remark}

\newtheorem{assumption}{Assumption}

\newcommand{\coker}{\mathrm{coker}}
\newcommand{\cone}{\mathrm{Cone}}
\newcommand{\Hom}{\mathrm{Hom}}
\newcommand{\thick}{\mathrm{thick}}
\newcommand{\add}{\mathrm{add}}
\newcommand{\per}{\mathrm{per}}
\newcommand{\Ext}{\mathrm{Ext}}
\newcommand{\End}{\mathrm{End}}
\newcommand{\obj}{\mathrm{obj}}

\newcommand{\bijar}[1][]{%
	\ar[#1]
	\ar@<0.7ex>@{}[#1]|-*[@]{\sim}}

\newcommand{\pvd}{\mathrm{pvd}}

\newcommand{\ind}{\mathrm{ind}}
\newcommand{\copr}{\mathrm{copr}}

\newenvironment{Rem}{\begin{Rem1}\rm}{\end{Rem1}}
\newenvironment{Ex}{\begin{Ex1}\rm}{\end{Ex1}}

 \normalbaselines
\newcommand{\Si}{\Sigma}
\newcommand{\La}{\Lambda}
\newcommand{\si}{\sigma}
\newcommand{\we}{\wedge}
\newcommand{\ten}{\otimes}
\newcommand{\lten}{\overset{\mathbf{L}}{\ten}}

\newcommand{\RHom}{\mathrm{\mathbf{R}Hom}}
\newcommand{\ra}{\rightarrow}

\newcommand{\iso}{\xrightarrow{_\sim}}

\newcommand{\id}{\mathbf{1}}
\renewcommand{\qedsymbol}{$\surd$}
\newcommand{\bmgamma}{\bm\Gamma}
\newcommand{\pr}{\mathrm{pr}}
\newcommand{\Gammabf}{\bm\Gamma}
\newcommand{\relGammabf}{\bm\Gamma}
\newcommand{\overGamma}{\overline{\bmgamma}}
\newcommand{\Gr}{\mathrm{Gr}}
\newcommand{\Mod}{\mathrm{Mod}}
\renewcommand{\mod}{\mathrm{mod}}

\newcommand{\overc}{\overline{\cc}}
\newcommand{\gpr}{\mathrm{gpr}}

\newcommand{\ca}{{\mathcal A}}
\newcommand{\cb}{{\mathcal B}}
\newcommand{\cc}{{\mathcal C}}
\newcommand{\cd}{{\mathcal D}}
\newcommand{\ce}{{\mathcal E}}
\newcommand{\cf}{{\mathcal F}}

\newcommand{\ch}{{\mathcal H}}

\newcommand{\ck}{{\mathcal K}}

\newcommand{\cm}{{\mathcal M}}

\newcommand{\cp}{{\mathcal P}}

\newcommand{\cs}{{\mathcal S}}
\newcommand{\ct}{{\mathcal T}}
\newcommand{\cu}{{\mathcal U}}

\newcommand{\cw}{{\mathcal W}}
\newcommand{\cx}{{\mathcal X}}
\newcommand{\cy}{{\mathcal Y}}
\newcommand{\cz}{{\mathcal Z}}

\newcommand{\T}{\mathbb{T}}
\renewcommand{\P}{\mathbb{P}}
\newcommand{\Q}{\mathbb{Q}}
\newcommand{\Z}{\mathbb{Z}}
\newcommand{\C}{\mathbb{C}}

\renewcommand{\phi}{\varphi}

\newcommand{\ko}{\;\;, }
\newcommand{\ie}{i.e.~}

\newcommand{\ol}[1]{\overline{#1}}
\newcommand{\ul}[1]{\underline{#1}}
\renewcommand{\hat}[1]{\widehat{#1}}
\renewcommand{\tilde}[1]{\widetilde{#1}}

\begin{document}
	
	\title[]{Relative cluster categories and Higgs categories\\ with infinite-dimensional morphism spaces}
	
	

	\author{Bernhard Keller and Yilin WU}
	
	\address{Universit\'e Paris Cit\'e and Sorbonne Université, CNRS, IMJ-PRG, F-75013 Paris, France}
	\email{bernhard.keller@imj-prg.fr}
	\urladdr{https://webusers.imj-prg.fr/~bernhard.keller/}

\address{
	School of Mathematical Sciences\\
	University of Science and Technology of China\\
    Hefei 230026, Anhui\\
	P. R. China}
\email{wuyilinmath@gmail.com}
\urladdr{https://yilinwu-math.github.io/}
	
	
	\dedicatory{}
	
	
	
\begin{abstract}
Cluster algebras {\em with coefficients} are important since they appear in nature as coordinate algebras of varieties like Grassmannians, double Bruhat cells, unipotent cells, \ldots\ . For such `Lie-theoretic' cluster algebras, 
Geiss--Leclerc--Schr\"oer have constructed categorifications using Frobenius exact categories. However,
large classes of cluster algebras with coefficients (for example, principal coefficients) cannot admit such Frobenius
categorifications. In previous work, the second-named author has constructed Higgs categories and relative cluster categories in the relative Jacobi-finite setting. Higgs categories generalize the Frobenius categories used by 
Geiss--Leclerc--Schr\"oer but are no longer exact categories in the sense of Quillen in general. They serve
to categorify cluster algebras with {\em non-invertible coefficents} whereas relative cluster categories
serve to categorify their localizations at the coefficients.
		 
In this article, we construct the Higgs category and the relative cluster category in the 
relative Jacobi-infinite setting under suitable hypotheses. These cover for example the case
of Jensen--King--Su's Grassmannian cluster category and, more generally, the positroid cluster
categories of Pressland and Canakci--King--Pressland.
As in the relative Jacobi-finite case, the Higgs category need not be exact but still extriangulated in the 
sense of Nakaoka--Palu.
We also construct a cluster character refining Plamondon's and show that it allows to lift
cluster variables and clusters to the categorical level.

In the appendix, Chris Fraser and the first-named author categorify quasi-cluster morphisms using 
Frobenius categories (for example suitable Higgs categories). A recent application of this result is due 
to Matthew Pressland, who uses it to prove a conjecture by Muller--Speyer.
\end{abstract}
	
	\dedicatory{{\rm With an appendix by Chris Fraser and Bernhard Keller} \\[0.5cm]
	Dedicated to Professor Henning Krause on the occasion of his 60th birthday}
	
	\maketitle
	
	\setcounter{tocdepth}{1}
	
	\tableofcontents

\section{Introduction}	
Cluster categories were introduced in 2006 by Buan--Marsh--Reineke--Reiten--Todorov
\cite{buanClusterStructures2Calabi2009} in order to categorify acyclic 
cluster algebras  \cite{fominClusterAlgebrasFoundations2002a, fominsergeyTotalPositivityCluster, FominWilliamsZelevinsky16}
without  coefficients. Caldero and Chapoton used the geometry of 
quiver  Grassmannians to define the cluster character~\cite{Caldero-Chapot2006}, i.e. a decategorification map 
which yields a bijection from the set of isomorphism classes of indecomposable objects of the cluster category of a Dynkin quiver to the set of cluster variables in the associated cluster algebra. More generally, for cluster algebras
associated with acyclic quivers, Caldero--Keller \cite{Caldero-Keller2008} showed that the
Caldero--Chapoton map induces a bijection between the set of isomorphism classes of indecomposable rigid objects 
and the set of cluster variables.

In order to generalize the representation-theoretic approach to cluster algebras from acyclic quivers
to quivers with oriented cycles, Derksen--Weyman--Zelevinsky 
\cite{DerksenWeymanZelevinsky08, DerksenWeymanZelevinsky10} extended the
mutation operation from quivers to quivers with potential and their representations. 
In the Jacobi-finite case, Amiot~\cite{amiotClusterCategoriesAlgebras2009}
generalized the construction of the cluster category.
If the potential is also non-degenerate, the cluster character constructed
by Palu in~\cite{paluClusterCharacters2Calabi2008} induces a bijection
from the isomorphism classes of reachable rigid indecomposables of the
generalized cluster category to the cluster variables of the associated
cluster algebra~\cite{CerulliKellerLabardiniPlamondon2013}.
Plamondon~\cite{plamondonCategoriesAmasseesAux2011} extended these
constructions to quivers with potential without the Jacobi-finiteness
assumption by working with a suitable subcategory of the cluster category.

Cluster algebras with coefficients are of great importance in geometric examples. They appear in nature as coordinate algebras of varieties like Grassmannians, double Bruhat cells, unipotent cells, \ldots\ cf.~for example \cite{Berenstein-Fomin-Zelevinsky2005,GLS2010,scottGrassmanniansClusterAlgebras2006}. The work of Geiss--Leclerc--Schr\"oer provides Frobenius exact categories which allow to categorify such cluster algebras in many cases~\cite{GLS2006,GLS2010}. 
In their approach, the Frobenius exact category is a full subcategory of the
category of modules over a certain preprojective algebra. Geiss--Leclerc--Schr\"oer's setting was
axiomatized by Fu--Keller \cite{Fu-Keller2010} using stably $2$-Calabi--Yau Frobenius categories.
They also observed that not all cluster algebras with coefficients admit such a categorification 
(for example, acyclic cluster algebras with principal coefficients do not, cf.~Remark~5.7 of  \cite{Fu-Keller2010}).

In order to extend Geiss--Leclerc--Schr\"oer's approach to larger classes of cluster algebras with coefficients,
the second-named author introduced relative cluster categories and Higgs categories in 
\cite{wuRelativeClusterCategories2021}. Higgs categories generalize the 
Frobenius categories used by Geiss--Leclerc--Schr\"oer. They serve to categorify cluster algebras
with {\em non-invertible} coefficients whereas relative cluster categories serve to categorify
their {\em localizations at the coefficients}. The setting in \cite{wuRelativeClusterCategories2021} was
however restricted to the case of ice quivers with potential $(Q,F,W)$ whose associated relative Jacobian algebra is
finite-dimensional. In this article, our aim is to construct the associated Higgs category $ \ch(Q,F,W) $
and the relative cluster category $ \cc(Q,F,W) $ under much weaker assumptions~(cf.~Assumption~\ref{assumption}
and section~\ref{ss:reduced-Jacobi-finite-case}). 
Under the hypotheses of Theorem~\ref{Thm: cluster bijection}, we construct a canonical cluster character and use it to lift cluster variables and clusters to the categorical level. Our cluster character generalizes Plamondon's 
\cite{plamondonCategoriesAmasseesAux2011} to the relative context. 

Let us state our main results more precisely.
Let $(Q,F,W)$ be an ice quiver with potential and
$\relGammabf=\Gammabf(Q,F,W)$ the associated complete relative Ginzburg
dg algebra (cf.~Section~\ref{ss:Ginzburg algebras}).
Let $e=\sum_{i\in F_0}e_i$ be the idempotent associated with the frozen vertices.
Let $ J(Q,F,W)=H^0(\relGammabf)$ be the corresponding relative Jacobian algebra.
We denote by $ \cp=\add(e\relGammabf) $ the closure under finite direct sums and summands
of $e\relGammabf$ in the perfect derived category $ \per\relGammabf $. The {\em relative cluster
	category $\cc(Q,F,W)$} is defined as the idempotent completion of the quotient of $\per\Gammabf$ by the thick subcategory
generated by the simple $H^0(\Gammabf)$-modules associated with the unfrozen vertices of $Q$.
The {\em Higgs category $\ch(Q,F,W)$} is an extension-closed full
subcategory of $\cc(Q,F,W)$
(cf.~Definition~\ref{Def: Higgs cat} and
Proposition~\ref{Prop: H is extension closed}).

Let $(\overline Q,\overline W)$ be the quiver with potential obtained
by deleting the frozen vertices, and let
$\overGamma=\mathbf{\Gamma}(\overline Q,\overline W)$ be its complete
Ginzburg dg algebra. Vertex deletion induces a morphism of complete dg
algebras $p\colon\relGammabf\ra\overGamma$, and extension of scalars
along $p$ gives a triangle functor
\[
p^*\colon\cc(Q,F,W)\ra\cc(\overline Q,\overline W)
\]
to the generalized cluster category of $(\overline Q,\overline W)$.
When $(\overline Q,\overline W)$ is Jacobi-finite, $p^*$ realizes
$\cc(\overline Q,\overline W)$ as the Verdier quotient of $\cc(Q,F,W)$
by the thick subcategory generated by the image of $e\relGammabf$
(see Proposition~\ref{Prop:silting reduction-fully faithful in cc}).
Let $\cd(Q,F,W)$ be the full subcategory of $\cc(Q,F,W)$ consisting of
those objects $M$ whose image $p^*(M)$ lies in Plamondon's category
$\cd(\overline Q,\overline W)$
(see Subsection~\ref{Subsection: Plamondon's category}). In the theorem
below, we write $\Gammabf=\Gammabf(Q,F,W)$ and $\ch=\ch(Q,F,W)$ for
brevity.

\begin{Thm}(Theorem~\ref{Thm: main results})
Assume that $(\overline Q,\overline W)$ is Jacobi-finite.
	Then Assumption~\ref{assumption} holds by
	Proposition~\ref{prop:reduced-finite-implies-approximations}.
	\begin{itemize}
		\item[1)] The functor $p^*$ induces an equivalence of $k$-linear categories
		\[
		\ch(Q,F,W)/[\cp]\xrightarrow{\sim}\cc(\overline Q,\overline W).
		\]
		\item[2)] The Higgs category $\ch$ equals the full subcategory
		of $\cc$ formed by the objects $X$ such that
		\[
		\Ext^i(X,P)=0=\Ext^i(P,X)
		\qquad\text{for all }i>0\text{ and }P\in\cp.
		\]
		Here $\Ext^i$ is computed in $\cc$. The category $\ch$ is Frobenius
		extriangulated with projective-injective subcategory $\cp$, and
		the equivalence in $1)$ is a triangle equivalence for the
		triangulated structure on $\ch/[\cp]$.
		Moreover, $T=\pi^{\mathrm{rel}}\bmgamma$ is a canonical cluster-tilting
		object of $\ch$, with $\End_{\ch}(T)\cong H^0(\bmgamma)=J(Q,F,W)$.
		\item[3)] Assume in addition that $H^i(\bmgamma)=0$ for $i\ne0$. Put
		$A=H^0(\bmgamma)=J_{rel}$ and $B=eAe$. In $\ch$ put
		$I=\pi^{\mathrm{rel}}(e\bmgamma)$ and $\cm=\add T$.
		Then $\ch$ is a Frobenius exact category. The boundary algebra $B$ is
		$\mathrm{fp}_{\infty}$-Gorenstein of dimension at most $3$ with
		respect to $A$ in the sense of
		Theorem~\ref{Thm: structure thm without Noetherian}.
		Restricted Yoneda gives an exact equivalence
		\[
		\Phi\colon\ch\xrightarrow{\sim}\gpr_{\infty}(B,A),
		\qquad X\longmapsto\Hom_{\ch}(I,X).
		\]
		Here $\gpr_{\infty}(B,A)$ consists of those
		$L\in\mathrm{fp}_{\infty}(B)$ for which
		$\operatorname{Ext}_B^q(L,B)=0$ for all $q>0$ and
		$L\otimes_B eA$ has a bounded resolution by finitely generated
		projective right $A$-modules. In particular, $\Phi(T)\cong Ae$.
		If $A$ is right coherent, then $\Phi$ induces an exact equivalence
		$\ch\simeq\gpr_{\infty}(B)$.
		
		\item[4)] Under the hypotheses of~$3)$, there is a commutative
		diagram of triangulated categories, up to natural isomorphism,
		whose rows are Verdier localization sequences and whose vertical
		arrows are triangle equivalences.
		\[
		\begin{tikzcd}[column sep=large]
			\ck^b_{\ch\text{-ac}}(\cm) \arrow[r,hook] \arrow[d,"\sim"']
			& \ck^b(\cm) \arrow[r] \arrow[d,"\sim"']
			& \cd^b(\ch) \arrow[d,"\sim"] \\
			\pvd_e(\bmgamma) \arrow[r,hook]
			& \per\bmgamma \arrow[r,"\pi^{\mathrm{rel}}"']
			& \cc(Q,F,W).
		\end{tikzcd}
		\]
		In particular, we have a triangle equivalence $\cc(Q,F,W)\simeq\cd^b(\ch)$, where $\cd^b(\ch)$
		is taken with respect to the exact structure in~$3)$.
	\end{itemize}
	
\end{Thm}

In Section~\ref{s:cluster characters}, we work over $k=\mathbb C$.
Suppose that $\cp=\add(e\relGammabf)$ is functorially finite in
$\add(\relGammabf)$ and that $(\overline Q,\overline W)$ is Jacobi-finite.
Number the vertices so that $1,\ldots,r$ are unfrozen and $r+1,\ldots,n$
are frozen. We generalize Fu--Keller's cluster character~\cite{Fu-Keller2010}
to a cluster character $CC=X_{?}$ on the Higgs category $\ch$, which
may have infinite-dimensional morphism spaces. If $(Q,W)$, regarded as
an ordinary quiver with potential on all $n$ vertices, is non-degenerate
in the sense of Derksen--Weyman--Zelevinsky, then $X_M$ belongs to
\[
\cu_{Q,F}^{+}\subseteq
\mathbb{Q}[x_{r+1},\ldots,x_n][x_1^{\pm1},\ldots,x_r^{\pm1}]
\qquad (M\in\ch),
\]
where $\cu_{Q,F}^{+}$ is the upper cluster algebra with
{\em non-invertible} frozen variables. Write $\ca_{Q,F}$ for the corresponding cluster algebra.
Throughout the paper, we use the term \emph{cluster variable} for
both mutable and frozen variables. By a \emph{cluster} we mean the
full set of mutable and frozen variables of a seed.

In every seed obtained by mutation at unfrozen vertices, $X_M$ is
Laurent in the mutable variables of that seed and polynomial in the
frozen variables
(see Remark~\ref{rem:relative-character-frozen-coefficients}).
The following theorem generalizes results of
\cite{Fu-Keller2010,CerulliKellerLabardiniPlamondon2013}.

\begin{Thm}(Theorem \ref{Thm: cluster bijection})
	Let $ (Q,F) $ be an ice quiver and $ W $ a non-degenerate potential on $ Q $. Assume that \((\bar Q,\bar W)\) is Jacobi-finite. Let
	$ \ch $ be the associated Higgs category. Then the cluster character $ X_{?}\colon\obj(\ch)\ra \cu_{Q,F}^{+} $ induces a bijection
	$$ \{\rm\text{reachable rigid indecomposable objects of $\ch$}\}/\text{isom}\iso\{\text{cluster variables}\}\subseteq\cu_{Q,F}^{+} .$$ For each frozen vertex $f$, the summand $T_f$ corresponds to $x_f$.
	Under this bijection, the isomorphism classes of basic cluster-tilting
	objects reachable from $T$ correspond to the extended clusters of
	$\ca_{Q,F}$.
\end{Thm}

Let us point out that Wang--Wei--Zhang \cite{WangWeiZhang23} have recently constructed a canonical cluster character
in the framework of ($\Hom$-finite) $2$-Calabi--Yau Frobenius extriangulated categories,
cf.~also the construction announced by Grabowski--Pressland
in~\cite{Grabowsk-Pressland2023}.

For the extension of $CC$ to the whole relative cluster category,
assume that $k=\mathbb C$, that $\cp$ is functorially finite in
$\add(\relGammabf)$ and that $(\overline Q,\overline W)$ is Jacobi-finite.
By definition, the Higgs category $\ch(Q,F,W)$ is a full subcategory of
$\cc(Q,F,W)$. After inverting the frozen variables, the following theorem
extends $CC$ to a Laurent-polynomial-valued map on $\cc(Q,F,W)$.

\begin{Thm}(Theorem~\ref{Thm: commutative diagram})
	Under these assumptions, consider the following diagram on isomorphism
	classes of objects, where $q$ is
	the stable quotient functor, $\underline{p^*}$ is the equivalence
	induced by $p^*$, and $\overline{CC}$ is Plamondon's cluster
	character~\cite{plamondonCategoriesAmasseesAux2011}.
	\[
	\begin{tikzcd}[column sep=small]
		\ch \arrow[r,hook] \arrow[d,"q"']
		& \cc(Q,F,W) \arrow[r,dashed,"CC_{loc}"] \arrow[d,"p^*"']
		& \mathbb Q[x_1^{\pm1},\ldots,x_n^{\pm1}]
		\arrow[d,"{x_f\mapsto1\ (f>r)}"] \\
		\underline{\ch} \arrow[r,"{\underline{p^*}}"',"\sim"]
		& \cc(\overline Q,\overline W) \arrow[r,"\overline{CC}"']
		& \mathbb Q[x_1^{\pm1},\ldots,x_r^{\pm1}].
	\end{tikzcd}
	\]
	There is a unique map
	\[
	CC_{loc}\colon\cc(Q,F,W)\ra\mathbb Q[x_1^{\pm1},\ldots,x_n^{\pm1}]
	\]
	such that the diagram commutes and the following conditions hold.
	\begin{itemize}
		\item[1)] For each triangle in $ \cc(Q,F,W) $
		$$ P\ra X\ra M\ra\Si P $$ with $ P\in\thick_{\cc}(\cp) $,
		we have 
		$$ CC_{loc}(X)=CC_{loc}(P)\cdot CC_{loc}(M). $$
		\item[2)] The restriction $ CC_{loc}|_{\ch} $ is the cluster character $ X_{?} $ defined 
		in~Section~\ref{s:cluster characters}.
		\item[3)] For each $P\in\thick_{\cc}(\cp)$, let $\widetilde P$ be
		its lift to $\thick_{\per\bmgamma}(e\bmgamma)$ under
		$\pi^{\mathrm{rel}}$. Then $CC_{loc}(P)=x^{[\widetilde P]}$,
		where $[\widetilde P]\in K_0(\per\bmgamma)\simeq\mathbb Z^n$.
	\end{itemize}
\end{Thm}

This article is organized as follows. In Section~\ref{Section: ice quiver},
we first recall the definitions of ice quivers with potential and their
mutations, and then construct the associated (complete) relative Ginzburg
algebras. In Section~\ref{Section: Relative cluster categories and Higgs Categories},
we follow Plamondon's approach to define the Higgs category. Its
relationship with Plamondon's category is explained in
Section~\ref{Section: Relationship with Plamondon's categor}, and in
particular in Corollary~\ref{Cor: equivalence of k-cat}.

Turning to mutation, assume that $W$ is irredundant and that no loop or
oriented $2$-cycle passes through a vertex $v$, and write
$(Q',F',W')=\mu_v(Q,F,W)$ for the mutation at $v$. When $v$ is unfrozen
and is the source of at least one arrow, mutation at $v$ yields an
equivalence between the relative cluster categories of $(Q',F',W')$ and
$(Q,F,W)$ that restricts to an equivalence of their Higgs categories
(Proposition~\ref{Prop: mutation at unfrozen}). When $v$ is a frozen
source or sink, the corresponding triangle equivalence restricts to a
$k$-linear equivalence $\cd(Q',F',W')\simeq\cd(Q,F,W)$
(Proposition~\ref{Rem: mutation at frozen} and its dual).

Section~\ref{Section: Cluster characters} is devoted to constructing the
cluster character $CC=X_{?}$ (with respect to $\relGammabf$) on
$\ch(Q,F,W)$. Using an argument similar to Plamondon's, we establish the
multiplication formula (Section~\ref{Subsection: Multiplication formula})
and then extend $CC$ to a map
\[
CC_{\mathrm{loc}}\colon\cc(Q,F,W)\ra\mathbb{Q}[x_{1}^{\pm1},\ldots,x_{r}^{\pm1},x_{r+1}^{\pm1},\ldots,x_{n}^{\pm1}].
\]
In Section~\ref{Section: Applications to quasi-cluster homomorphisms}, we
consider the seeds determined by the signed adjacency matrices whose
columns are indexed by the unfrozen vertices. With the frozen variables
inverted, the substitution induced by mutation at a frozen source or sink
restricts to a quasi-cluster isomorphism between the corresponding cluster
algebras (see Remarks~\ref{rem:frozen-source-quasi-cluster}
and~\ref{rem:frozen-sink-quasi-cluster} for the precise hypotheses). Finally, in
Section~\ref{s:Postnikov diagrams}, we explain how the ice quivers with
potential arising from connected Postnikov diagrams in the disc fit into
the framework developed in this article.

In Appendix~\ref{appendix}, Chris Fraser and the first-named author use
Frobenius categorifications (which can often be obtained from the main
results of this paper) to construct quasi-cluster isomorphisms.
In~\cite{presslandMuller-Speyer conjecture}, Matthew Pressland used this
criterion to prove the conjecture of
Muller--Speyer~\cite[Rem.~4.7]{muller-speyer2017} that the
source-labelled and target-labelled cluster structures on an open
positroid variety quasi-coincide.

\subsection*{Acknowledgments}
The second-named author thanks Guodong Zhou for his
constant support and encouragement throughout his mathematical career. He
would also like to thank Peigen Cao, Xiaofa Chen, Junyang Liu, and Yu Wang
for helpful discussions during his stay in Paris. Both authors are
grateful to Matthew Pressland for his interest and feedback, and they
thank Fang Li, Sibylle Schroll, and Alireza Nasr-Isfahani for
opportunities to present some of the results of this paper. Finally, the
authors thank the referee for a careful reading of the manuscript and for
suggestions that helped clarify several proofs.

\subsection*{Use of generative AI}

The manuscript of this paper originally submitted to the journal in December 2023
was written without using AI. During the revision in November 2026, the second-named 
author used GPT-6.2 Astra to edit the English and the \LaTeX{} source, organize the manuscript, 
and discuss the presentation of certain proofs. We independently verified the arguments, calculations, and references, made all final decisions about the manuscript, and take full 
responsibility for its content.

\subsection*{Notation and conventions}
Throughout this paper, $k$ denotes an algebraically closed field of
characteristic zero, and $D=\Hom_{k}(-, k)$ denotes the $k$-dual. Unless
stated otherwise, all modules are right modules.

Let $\ct$ be a triangulated category. For any subcategory $\ct'$ of $\ct$,
we write $\ind\ct'$ for the set of isomorphism classes of indecomposable
objects of $\ct$ lying in $\ct'$, and $\add\ct'$ for the full
subcategory of $\ct$ consisting of all direct summands of finite direct
sums of objects in $\ct'$. We say that $\ct'$ is \emph{rigid} if
$\Hom_{\ct}(X,\Si Y)=0$ for all $X,Y\in\ct'$.

Given a subcategory $\cp$ of $\ct$, let $[\cp]$ denote the ideal of
morphisms in $\ct$ factoring through an object of $\cp$, and write
$\ct/[\cp]$ for the corresponding additive quotient category. We denote by
$\mathrm{tri}_{\ct}(\cp)$ the triangulated subcategory of $\ct$ generated
by $\cp$, that is, the smallest triangulated subcategory of $\ct$
containing $\cp$.

For two collections $\cx$ and $\cy$ of objects in $\ct$, we denote by
$\cx*\cy$ the collection of objects $Z\in\ct$ fitting into a triangle
$X\ra Z\ra Y\ra\Si X$ with $X\in\cx$ and $Y\in\cy$.

\section{Preliminaries}\label{Section: ice quiver}

\subsection{Ice quivers with potential}
\begin{Def}\rm
	A \emph{quiver} is a tuple $ Q=(Q_{0}, Q_{1},s,t) $, where $ Q_{0} $ and $ Q_{1} $ are sets, and $ s, t\colon Q_{1}\ra Q_{0} $ are functions. We think of the elements of $ Q_{0} $ as vertices and of those of $ Q_{1} $ as arrows, so that each $ \alpha\in Q_{1} $ is realised as an arrow $ \alpha\colon s(\alpha)\ra t(\alpha) $. We call $ Q $ \emph{finite} if $ Q_{0} $ and $ Q_{1} $ are finite sets.
\end{Def}
\begin{Def}\rm
	Let $ Q $ be a quiver. A quiver $ F=(F_{0}, F_{1}, s', t') $ is a \emph{subquiver} of $ Q $ if it is a quiver such that $ F_{0}\subseteq Q_{0} $, $ F_{1}\subseteq Q_{1} $ and the functions $ s' $ and $ t' $ are the restrictions of $ s $ and $ t $ to $ F_{1} $ . We say $ F $ is a \emph{full subquiver} if $ F_{1}=\{\alpha\in Q_{1}\colon s(\alpha),t(\alpha)\in F_{0} \} $, so that a full subquiver of $ Q $ is completely determined by its set of vertices.	
\end{Def}
\begin{Def}\rm
	An \emph{ice quiver} is a pair $ (Q,F) $, where $ Q $ is a finite quiver and $ F $ is a (not necessarily full) subquiver of $ Q $.
	We call $ F_{0} $, $ F_{1} $ and $ F $ the frozen vertices, arrows and subquiver respectively. We also call $ Q_{0}\setminus F_{0} $ and $ Q_{1}\setminus F_{1} $ the unfrozen vertices and arrows respectively.	
\end{Def}

Let $ k $ be a field. Let $ Q $ be a finite quiver.
\begin{Def}\rm\label{Def: semisimple algebra of vertices}
	Let $ S $ be the semisimple $ k $-algebra $ \prod_{i\in Q_{0}}ke_{i} $. The vector space $ kQ_{1} $ naturally becomes an $ S $-bimodule. Then the $ \emph{complete path algebra} $ of $ Q $ is the completed tensor algebra
	$$ \widehat{kQ}=\widehat{T_{S}}(kQ_{1}) .$$ It has underlying vector space 
	$$ \prod_{d=0}^{\infty}(kQ_{1})^{\otimes_{S}^{d}} $$ and multiplication given by concatenation. The algebra $ \widehat{kQ} $ becomes a graded pseudocompact $ S $-algebra in the sense of~\cite{Van den Bergh2015}.
\end{Def}

\begin{Def}\rm\cite[Definition 2.8]{presslandMutationFrozenJacobian2020}\label{Def:potential}
	A \emph{potential} on $ Q $ is an element $ W $ in $ H\!H_{0}(\widehat{kQ})=\widehat{kQ}/[\widehat{kQ},\widehat{kQ}] $ expressible as a (possibly infinite) linear combination of homogeneous elements of degree at least 2, such that any term involving a loop has degree at least 3. An \emph{ice quiver with potential} is a tuple $ (Q, F, W) $ in which $ (Q, F) $ is a finite ice quiver and $ W $ is a potential on $ Q $. If $ F=\emptyset $ is the empty quiver, then $ (Q,\emptyset,W)=(Q,W) $ is simply called a \emph{quiver with potential}. We say that $ W $ is \emph{irredundant} if each term of $ W $ includes at least one unfrozen arrow.
\end{Def}

A potential can be thought of as an infinite formal linear combination of cyclic paths in $ Q $ (of length at least 2), considered up to the equivalence relation on such cycles induced by
$$ \alpha_{n}\cdots\alpha_{1}\sim\alpha_{n-1}\cdots\alpha_{1}\alpha_{n} .$$

\begin{Def}\rm
	Let $ p=\alpha_{n}\cdots\alpha_{1} $ be a cyclic path, with each $ \alpha_{i}\in Q_{1} $, and let $ \alpha\in Q_{1} $ be any arrow. Then the \emph{cyclic derivative} of $ p $ with respect to $ \alpha $ is
	\[
	\partial_{\alpha}p
	=\sum_{\substack{1\leq i\leq n\\ \alpha_i=\alpha}}
	\alpha_{i-1}\cdots\alpha_1\,\alpha_n\cdots\alpha_{i+1},
	\]
	where an empty segment is omitted. We extend $\partial_\alpha$ by
	linearity and continuity. It determines a map
	$H\!H_0(\widehat{kQ})\to\widehat{kQ}$. For an ice quiver with
	potential $(Q,F,W)$, we define the \emph{relative Jacobian algebra}
	$J(Q,F,W)$ (or $J_{rel}$) as
	$$ \widehat{kQ}/\overline{\langle\partial_{\alpha}W\colon\alpha\in Q_{1}\setminus F_{1}\rangle} .$$
	If $ F=\emptyset $, we call $ J(Q,W)=J(Q,\emptyset,W) $ the \emph{Jacobian algebra} of the quiver with potential $ (Q, W) $. An ice quiver with potential $ (Q,F,W) $ is called \emph{Jacobi-finite} if the relative Jacobian algebra $ J(Q,F,W) $ is finite-dimensional. Otherwise, we say it is \emph{Jacobi-infinite}.
\end{Def}
\begin{Def}\rm
	Let $ Q $ be a quiver. An ideal of $ \widehat{kQ} $ is called \emph{admissible} if it is contained in the square of the closed ideal generated by the arrows of $ Q $. We call an ice quiver with potential $ (Q, F, W) $ \emph{reduced} if $ W $ is irredundant and the Jacobian ideal of $ \widehat{kQ} $ determined by $ F $ and $ W $ is admissible. An ice quiver with potential $ (Q, F, W) $ is \emph{trivial} if its relative Jacobian algebra $ J(Q, F, W) $ is a product of copies of the base field $ k $.
\end{Def}

\subsection{Mutation of ice quivers with potentials}\label{Subsection:mutations}

Two ice quivers with potential $ (Q,F,W) $ and $ (Q',F',W') $ are \emph{right equivalent} if $ Q_{0}=Q'_{0} $, $ F_{0}=F'_{0} $ and there exists an $ S $-algebra isomorphism $ \varphi\colon\widehat{kQ}\ra\widehat{kQ'} $ such that
\begin{itemize}
	\item[(1)] $ \varphi|_{S}=\id_{S} $,
	\item[(2)] $ \varphi(\widehat{kF})=\widehat{kF'} $, where $ \widehat{kF} $ and $ \widehat{kF'} $ are treated in the natural way as subalgebras of $ \widehat{kQ} $ and $ \widehat{kQ'} $ respectively, and
	\item[(3)] $ \varphi(W) $ equals $ W' $ in $ H\!H_{0}(\widehat{kQ'}) $.
\end{itemize}
In that case, the relative Jacobian algebras of the two ice quivers with potential are isomorphic (see~\cite[Proposition 3.10]{presslandMutationFrozenJacobian2020}).

By \cite[Theorem~3.6]{presslandMutationFrozenJacobian2020}, every ice
quiver with potential $(Q,F,W)$ admits a reduced ice quiver with
potential $(Q_{\mathrm{red}},F_{\mathrm{red}},W_{\mathrm{red}})$
with isomorphic relative Jacobian algebra.

In statements concerning reduction or mutation, we assume from now on
that the input potential is irredundant in the sense of
Definition~\ref{Def:potential}. For such an input, the construction
of \cite[Theorem~3.6]{presslandMutationFrozenJacobian2020} determines
$(Q_{\mathrm{red}},F_{\mathrm{red}},W_{\mathrm{red}})$ up to right
equivalence. We call it the \emph{reduction} of $(Q,F,W)$, following
\cite[Proposition~3.15 and Definition~3.16]{presslandMutationFrozenJacobian2020}.
This convention will be used when comparing Ginzburg functors under
mutation. 
\subsubsection{Mutation at unfrozen vertices}
Let $ v $ be an unfrozen vertex of $ Q $ such that no loops or 2-cycles of $ Q $ are incident with $ v $. The mutation at
the vertex $ v $ is the new ice quiver with potential $ \mu_{v}(Q,F,W) $ obtained from $ (Q,F,W) $ (see~\cite{presslandMutationFrozenJacobian2020}) in the following way:
\begin{itemize}
	\item[(1)] For each pair of arrows $ \alpha\colon u\to v $ and $ \beta\colon v\to w $, add an unfrozen ‘composite' arrow $ [\beta\alpha]\colon u\to w $ to $ Q $.
	\item[(2)] Reverse each arrow incident with $v$, writing $\alpha^*$
	for the reverse of $\alpha$. This gives a new ice quiver $(Q',F)$.
	
	\item[(3)] Choose a representative $W_0\in\widehat{kQ}$ of $W$
	in which no cyclic path is based at $v$. Such a choice is possible
	by cyclic rotation, since there is no loop at $v$. Simultaneously
	replace each subpath $\beta\alpha$ through $v$ by $[\beta\alpha]$
	and add the term $[\beta\alpha]\alpha^*\beta^*$ once for each
	pair $(\alpha,\beta)$ from~(1). This gives a new potential $W'$.
\end{itemize}
The mutation $ \mu_{v}(Q,F,W) $ of $ (Q,F,W) $ at $ v $ is then defined to be the reduction of $ (Q',F,W') $.
\subsubsection{Mutation at frozen vertices}
Let $ (Q,F,W) $ be an ice quiver with potential. Let $ v $ be a frozen vertex.
\begin{Def}\rm
	We say that $ v $ is a \emph{frozen source} of $ Q $ if $ v $ is a source vertex of $ F $ and there are no unfrozen arrows with source $ v $. Similarly, we say that $ v $ is a \emph{frozen sink} of $ Q $ if $ v $ is a sink vertex of $ F $ and there are no unfrozen arrows with target $ v $. For two vertices $ i $ and $ j $, we say that they have the \emph{same state} if they are both in $ F_{0} $ or $ Q_{0}\setminus F_{0} $. Otherwise, we say that they have \emph{different states}. Similarly, for two arrows in $ Q $, we say that they have the \emph{same state} if they are both in $ F_{1} $ or $ Q_{1}\setminus F_{1} $. Otherwise, we say that they have \emph{different states}.
\end{Def}
Let $(Q,F,W)$ be an ice quiver with irredundant potential, so that
every cyclic term of $W$ contains at least one arrow outside $F_1$.
Let $v$ be a frozen source or a frozen sink, and assume that no loop
or oriented $2$-cycle of $Q$ passes through $v$.
Following \cite[Section~6.2]{wuCategorificationIceQuiver2021}, we
define the premutation
\[
\widetilde\mu_v(Q,F,W)=(\widetilde Q,\widetilde F,\widetilde W)
\]
as follows.
\begin{enumerate}
	\item For each pair of arrows $\alpha\colon u\to v$ and
	$\beta\colon v\to w$, add an unfrozen arrow
	$[\beta\alpha]\colon u\to w$.
	
	\item Reverse each arrow incident with $v$, writing
	$\alpha^*\colon v\to u$ and $\beta^*\colon w\to v$ for the reverses
	of $\alpha$ and $\beta$. A reversed arrow is frozen exactly when the
	original arrow is frozen. Retain the original vertex set, frozen
	vertices, and all arrows not incident with $v$, with their original
	frozen or unfrozen status. This defines $(\widetilde Q,\widetilde F)$.
	
	\item Choose a representative $W_0\in\widehat{kQ}$ of the cyclic
	class of $W$ in which each cyclic path is based at a vertex other
	than $v$. Such a representative exists by cyclic rotation, since
	there is no loop at $v$. Let $[W_0]$ be obtained by simultaneously
	replacing every subpath $\beta\alpha$ through $v$ by the arrow
	$[\beta\alpha]$. Define $\widetilde W$ to be the cyclic class of
	\[
	[W_0]+
	\sum_{\substack{\alpha,\beta\in Q_1\\ t(\alpha)=v=s(\beta)}}
	[\beta\alpha]\,\alpha^*\beta^* .
	\]
\end{enumerate}
The potential $\widetilde W$ is irredundant. Each modified cyclic
term and each new cubic term contains an unfrozen composite arrow,
and every unmodified term already contains an unfrozen arrow.
The \emph{mutation} at $v$ is the reduction of this premutation,
\[
\mu_v(Q,F,W)
=(\widetilde Q,\widetilde F,\widetilde W)_{\mathrm{red}},
\]
taken up to right equivalence in the sense of
\cite[Proposition~3.15 and Definition~3.16]
{presslandMutationFrozenJacobian2020}.

\bigskip	
We use $\widetilde\mu_v(Q,F,W)$ and $\mu_v(Q,F,W)$ for premutation
and mutation at either kind of vertex. The same argument as above
shows that premutation at an unfrozen vertex preserves irredundancy.

\subsection{Complete relative Ginzburg dg algebras}
\label{ss:Ginzburg algebras}
\begin{Def}\rm
	Let $ (Q,F,W) $ be a finite ice quiver with potential. Let $ \widetilde{Q} $ be the graded quiver with the same vertices as $ Q $ and whose arrows are
	\begin{itemize}
		\item the arrows of $Q$, in degree $0$,
		\item an arrow $a^*\colon j\to i$ of degree $-1$ for each
		arrow $a\colon i\to j$ in $Q_1\setminus F_1$,
		\item a loop $ t_{i}\colon i\to i $ of degree $ -2 $ for each vertex $ i $ of $ Q $ not belonging to $ F $.
	\end{itemize}
	
	Define the \emph{completed relative Ginzburg dg algebra} $ \bm{\Gamma}(Q,F,W) $ as the dg algebra whose underlying graded space is the completed graded path algebra $ \widehat{k\widetilde{Q}} $. Its differential is the unique $ k $-linear continuous endomorphism of degree 1 which satisfies the Leibniz rule
	\[
	d(uv)=d(u)v+(-1)^p u\,d(v)
	\]
	for all homogeneous $ u $ of degree $ p $ and all $ v $, and takes the following values on the arrows of $ \widetilde{Q} $:
	\begin{itemize}
		\item $ d(a)=0 $ for each arrow $ a $ of $ Q $,
		\item $d(a^{*})=\partial_{a}W$ for each arrow $ a $ of $ Q $ not belonging to $ F $,
		\item $d(t_i)=e_i\bigl(\sum_{a\in Q_1\setminus F_1}
		[a,a^*]\bigr)e_i$ for each $i\in Q_0\setminus F_0$,
		where $e_i$ is the stationary path at $i$.
	\end{itemize}
	
\end{Def}

\begin{Def}\rm
	Let $ F $ be any finite quiver. Let $  \widetilde{F} $ be the graded quiver with the same vertices as $ F $ and whose arrows are
	\begin{itemize}
		\item the arrows of $F$, in degree $0$,
		\item an arrow $\tilde a\colon j\to i$ of degree $0$ for each
		arrow $a\colon i\to j$ of $F$,
		\item a loop $ r_{i}:i\to i $ of degree $ -1 $ for each vertex $ i $ of $ F $.
	\end{itemize}
	Define \emph{complete derived preprojective algebra} $ \bm\Pi_{2}(F) $ as the dg algebra whose underlying graded space is the completed graded path algebra $ \widehat{k\widetilde{F}} $. Its differential is the unique $ k $-linear continuous endomorphism of degree 1 which satisfies the Leibniz rule
	\[
	d(uv)=d(u)v+(-1)^p u\,d(v)
	\]
	for all homogeneous $ u $ of degree $ p $ and all $ v $, and takes the following values on the arrows of $ \widetilde{F} $:
	\begin{itemize}
		\item $ d(a)=0 $ for each arrow $ a $ of $ F $,
		\item $d(\tilde{a})=0 $ for each arrow $ a $ in $ F $,
		\item $ d(r_{i})=e_{i}(\sum_{a\in F_{1}}[a,\tilde{a}])e_{i} $ for each vertex $ i $ of $ F $, where $ e_{i} $ is the lazy path corresponding to the vertex $ i $.
	\end{itemize}
\end{Def}

\begin{Lem}\rm
	With the above notations, $ J(Q,F,W) $ is isomorphic to $ \mathrm{H}^{0}(\bm{\Gamma}(Q,F,W)) $.
\end{Lem}

Let $(Q,F,W)$ be a finite ice quiver with potential. For completed
path algebras we use continuous Hochschild and cyclic homology.
Viewing $W$ as a class in $HC_0(\widehat{kQ})$, put $c=B(W)$, where
\[
B\colon HC_0(\widehat{kQ})\longrightarrow HH_1(\widehat{kQ})
\]
is Connes' connecting map
\cite[Section~6.1]{kellerDeformedCalabiYau2011a}.
Let $G\colon\widehat{kF}\hookrightarrow\widehat{kQ}$ be the inclusion
of vertex dg categories, and let $\xi$ be the image of $c$ under
the canonical map
\[
HH_1(\widehat{kQ})\longrightarrow HH_1(G).
\]
Using the representative with zero frozen component, the deformed
relative $3$-Calabi--Yau completion of $G$ with parameter $\xi$ gives
the following dg functor
\cite[Sections~2--4]{wuCategorificationIceQuiver2021}
$$ \bm{G}_{rel}\colon\bm{\Pi}_{2}(F)\ra\bm{\Gamma}(Q,F,W) .$$

This is the \emph{Ginzburg functor}. On vertices and arrows it is
given by the following formulas \cite[Section~4]{wuCategorificationIceQuiver2021}.
\begin{itemize}
	\item $ \bm G_{rel}(i)=i $ for each frozen vertex $ i\in F_{0} $,
	\item $ \bm G_{rel}(a)=a $ for each arrow $ a\in F_{1} $,
	\item $ \bm G_{rel}(\tilde{a})=-\partial_{a}W $ for each arrow $ a\in F_{1} $,
	\item $ \bm G_{rel}(r_{i})=e_{i}(\sum_{a\in Q_{1}\setminus F_{1}}[a,a^{*}])e_{i} $ for each frozen vertex $ i\in F_{0} $.
\end{itemize}

Let $e=\sum_{i\in F_0}e_i$ and write $\pvd_e(\relGammabf)$ for
the full subcategory of $\pvd(\relGammabf)$ formed by right dg modules
$M$ for which $Me$ is acyclic. Equivalently,
\[
\pvd_e(\relGammabf)=
\thick_{\cd(\relGammabf)}\langle S_i\mid i\in Q_0\setminus F_0\rangle,
\]
the thick subcategory generated by the simple modules associated with
the unfrozen vertices.
\begin{Prop}\label{Prop:relative 3-CY}
	Write $\mathcal D=\cd(\relGammabf)$ and $D=\Hom_k(-,k)$.
	For $M\in\pvd_e(\relGammabf)$ and $N\in\mathcal D$, there is a
	bifunctorial isomorphism
	\[
	D\Hom_{\mathcal D}(M,N)
	\cong \Hom_{\mathcal D}(N,\Si^3 M).
	\]
	If $\Hom_{\mathcal D}(M,N)$ is finite-dimensional, then there is
	also a canonical isomorphism
	\[
	\Hom_{\mathcal D}(M,N)
	\cong D\Hom_{\mathcal D}(N,\Si^3 M).
	\]
	This finiteness condition holds whenever $N\in\per(\relGammabf)$.
\end{Prop}

\begin{proof}
	Fix $M\in\pvd_e(\relGammabf)$, and assume first that $N$ is perfect.
	We apply the pseudocompact duality of
	\cite[Lemma~A.16]{kellerYangDerived equivalences} to the completed
	relative Calabi--Yau bimodule triangle from
	\cite[Sections~3--4]{wuCategorificationIceQuiver2021}. Since
	restriction along $\bm G_{\mathrm{rel}}$ evaluates $M$ at the frozen
	vertices and $Me$ is acyclic, the boundary term vanishes by
	adjunction. Just as in the proof of
	\cite[Proposition~3.12 and Corollary~3.13]{wuRelativeClusterCategories2021},
	this yields the trace pairing between $M$ and $N$. Here we identify
	perfect and perfectly valued pseudocompact modules with their ordinary
	counterparts via
	\cite[Propositions~2.14 and~2.17]{wuCategorificationIceQuiver2021}; in
	particular, $M$ is perfect.
	
	Now let $N$ be an arbitrary object of the ordinary derived category
	$\mathcal D$. Composing with the trace of $M$ induces a natural map
	\[
	\Hom_{\mathcal D}(N,\Si^3 M)\longrightarrow
	D\Hom_{\mathcal D}(M,N).
	\]
	Since $M$ is compact, both sides send coproducts in $N$ to products.
	The full subcategory of objects $N\in\mathcal D$ for which this map is
	an isomorphism in every shift is therefore localizing and contains
	$\relGammabf$, so it coincides with all of $\mathcal D$. Whenever
	$\Hom_{\mathcal D}(M,N)$ is finite-dimensional, dualizing and applying
	canonical biduality yields the second isomorphism.
	
	Finally, for every integer $p$ we have
	$\Hom_{\mathcal D}(\Si^p\relGammabf,\Si^3 M)\cong H^{3-p}(M)$,
	which is finite-dimensional. Passing to cones and direct summands, we
	conclude that $\Hom_{\mathcal D}(N,\Si^3 M)$ is finite-dimensional for
	every $N\in\per(\relGammabf)$, and the first isomorphism then implies
	that $\Hom_{\mathcal D}(M,N)$ is finite-dimensional as well.
\end{proof}

\begin{Prop}\label{Homtopy Cofiber sequence of (Q,F,W)}
	Let $(Q,F,W)$ be a finite ice quiver with potential and let
	$e=\sum_{f\in F_0}e_f$. Write
	\[
	\Gamma=\bmgamma(Q,F,W),\qquad
	\bar\Gamma=\bmgamma(\bar Q,\bar W),
	\]
	where $(\bar Q,\bar W)$ is obtained by deleting the frozen vertices.
	Deleting the vertices corresponding to $e$ induces an isomorphism of
	complete dg algebras
	\[
	\Gamma/\overline{\Gamma e\Gamma}\xrightarrow{\sim}\bar\Gamma,
	\]
	where $\overline{\Gamma e\Gamma}$ is the degreewise closure of
	$\Gamma e\Gamma$ in the path topology.
	
	Let $D$ be the ordinary derived idempotent quotient of $\Gamma$ by $e$.
	There is a canonical morphism $r:D\to\bar\Gamma$. If $(\bar Q,\bar W)$ is Jacobi-finite, then $r$ is a
	quasi-isomorphism. Under this hypothesis,
	\[
	\bm{\Pi}_2(F)\xrightarrow{\bm{G}_{\mathrm{rel}}}
	\Gamma\xrightarrow{p}\bar\Gamma
	\]
	is a cofiber sequence in the Morita homotopy theory of dg categories.
\end{Prop}

\begin{proof}
	Let $p\colon\Gamma\to\bar\Gamma$ be the vertex-deletion morphism. By
	the formulas for the cyclic derivatives and the vertex loops, $p$
	commutes with the differential. A homogeneous series lies in $\ker p$
	if and only if every path in its support passes through a frozen
	vertex, in which case each of its finite truncations belongs to
	$\Gamma e\Gamma$. Conversely, $\ker p$ is closed and contains
	$\Gamma e\Gamma$. Hence $\ker p=\overline{\Gamma e\Gamma}$, which
	proves the first assertion.
	
	In the ordinary Drinfeld dg quotient of $\Gamma$ by $e\Gamma e$, we
	adjoin to each frozen vertex $f$ an endomorphism $s_f$ of degree $-1$
	with $ds_f=e_f$. Let $D$ be the endomorphism dg algebra of the sum of
	the vertex objects in the additive hull of this quotient; sending each
	$s_f$ to zero extends $p$ to a morphism $r\colon D\to\bar\Gamma$. Set
	$A=H^0\Gamma$ and $\bar A=H^0\bar\Gamma$. Since $\Gamma$ is
	concentrated in nonpositive degrees, any word of degree $-1$ involving
	one of the $s_f$ has the form $u s_f v$ with $u,v$ of degree zero.
	Because morphisms in the ordinary Drinfeld dg quotient are formed
	using ordinary tensor products, the new boundaries in degree zero form
	precisely the algebraic ideal $AeA$, and therefore $H^0D=A/AeA$.
	
	We next check that $\ker(A\to\bar A)=\overline{AeA}$. If a class in
	this kernel is represented by $z\in\Gamma^0$, then $p(z)=d(\bar y)$
	for some $\bar y\in\bar\Gamma^{-1}$. Lifting $\bar y$ along the
	surjection $\Gamma^{-1}\to\bar\Gamma^{-1}$ to an element
	$y\in\Gamma^{-1}$, we obtain a cohomologous representative
	$z-dy\in\ker p$, which is a limit of finite sums of paths passing
	through frozen vertices; hence $\ker(A\to\bar A)\subseteq\overline{AeA}$.
	Conversely, the image of the differential $\bar\Gamma^{-1}\to\bar\Gamma^0$
	is closed, being the continuous linear image of a linearly compact
	space, so $\bar A$ is Hausdorff. Since $A\to\bar A$ is continuous and
	vanishes on $AeA$, its kernel is closed and therefore contains
	$\overline{AeA}$.
	
	Assume now that $\bar A$ is finite-dimensional. Set
	$K=\ker(A\to\bar A)$ and let $\mathfrak m$ be the arrow radical of
	$A$. Since the arrow radical of $\bar A$ is nilpotent, we have
	$\mathfrak m^l\subseteq K$ for some $l\ge1$. Grouping each series by
	its leftmost segment of length $l$ shows that the right ideal
	$\mathfrak m^l$ is generated by the images of the finitely many paths
	of length $l$. As $A/\mathfrak m^l$ is finite-dimensional, it follows
	that $K$ is a finitely generated right ideal of $A$. Because $AeA$ is
	dense in $K$ and $K\mathfrak m\supseteq\mathfrak m^{l+1}$ is open, we
	obtain $K=AeA+K\mathfrak m$. Moreover, the geometric series formula
	implies $\mathfrak m\subseteq\operatorname{rad}A$, so applying
	Nakayama's lemma to $K/AeA$ yields $K=AeA$. In particular,
	$H^0(r)\colon H^0D\to\bar A$ is an isomorphism.
	
	For each unfrozen vertex $i$, let $\bar S_i$ be the corresponding
	simple right $\bar\Gamma$-module and put $S_i=p_*\bar S_i$. We claim
	that the counit
	\[
	S_i\otimes_\Gamma^{\mathbf L}\bar\Gamma\longrightarrow\bar S_i
	\]
	is an isomorphism. To see this, consider the cone of the inclusion
	$\ker(e_i\Gamma\to S_i)\hookrightarrow e_i\Gamma$. As a graded right
	$\Gamma$-module, this kernel is a finite direct sum of degree-shifted
	copies of the projective modules $e_j\Gamma$, indexed by the leftmost
	arrow of each nontrivial path ending at $i$; this decomposition
	extends to completed paths because only finitely many such arrows
	exist. Since the differential of each generating arrow involves only
	generators of strictly higher degree, the cone provides a finite
	semi-free resolution of $S_i$
	\cite[Section~2.14]{kellerYangDerived equivalences}. Tensoring with
	$\bar\Gamma$ annihilates the summands corresponding to frozen vertices
	and projects the differential onto the remaining ones, yielding the
	analogous resolution of $\bar S_i$ over $\bar\Gamma$ together with its
	augmentation. This proves the claim.
	
	Now regard $\bar S_i$ as a right $D$-module via $r$. Since restriction
	from $D$ to $\Gamma$ is fully faithful on derived categories
	\cite[Proposition~4.6(ii)]{drinfeldDGQuotientsDG2004}, the canonical
	morphism $S_i\otimes_\Gamma^{\mathbf L}D\to\bar S_i$ is an isomorphism
	of right $D$-modules. By associativity of the derived tensor product
	and the claim above, $\bar S_i\otimes_D^{\mathbf L}\bar\Gamma\to\bar S_i$
	is also an isomorphism. Consequently, setting $N=\operatorname{Cone}(r)$
	and viewing it as a left $D$-module, we have
	$\bar S_i\otimes_D^{\mathbf L}N=0$ for every unfrozen vertex $i$.
	
	Note that $N$ is bounded above. Suppose for contradiction that $N$ has
	nonzero cohomology, let $q$ be the largest integer with $H^qN\ne0$,
	and put $S=\bar A/\operatorname{rad}\bar A$. Since $D$ is concentrated
	in nonpositive degrees, the top truncation triangle of $N$ yields
	\[
	H^q(S\otimes_D^{\mathbf L}N)
	\cong S\otimes_{\bar A}H^qN
	\cong H^qN/(\operatorname{rad}\bar A)H^qN.
	\]
	Here the left-hand side vanishes because $S$ is a finite direct sum of
	the simple modules $\bar S_i$ at the unfrozen vertices, whereas the
	right-hand side is nonzero because $\operatorname{rad}\bar A$ is
	nilpotent and $H^qN\ne0$, a contradiction. Thus $N$ is acyclic, and $r$
	is a quasi-isomorphism. Finally, on perfect modules, the thick
	subcategory generated by the image of $\bm{G}_{\mathrm{rel}}$ is
	generated by the projective modules $e_f\Gamma$ associated with the
	frozen vertices. By the universal property of the ordinary Drinfeld dg
	quotient therefore, we obtain the Morita cofiber sequence.
\end{proof}

\bigskip

For simplicity of notation, we write $ \bmgamma $ instead of $ \bmgamma(Q,F,W) $ and use $ \overGamma $ for $ \bm{\Gamma}(\overline{Q},\overline{W}) $. Denote by $ p:\bmgamma\ra\overGamma $ the vertex-deletion morphism of dg algebras. Then $ p $ induces the usual triple of adjoint functors $ (p^{*},p_{*},p^{!}) $ between $ \cd(\overGamma) $ and $ \cd(\bmgamma) $, where $ p_{*} $ is the restriction functor, $ p^{*}=?\,\lten_{\bmgamma}\overGamma $ and $ p^{!}=\RHom_{\bmgamma}(\overGamma,?) $.

\begin{Prop}\label{prop:perfect-frozen-quotient}
	Under the notation of
	Proposition~\ref{Homtopy Cofiber sequence of (Q,F,W)}, let
	$q:\Gamma\to D$ be the ordinary derived idempotent quotient by $e$,
	and write $p=rq:\Gamma\to\bar\Gamma$ for vertex deletion.
	The functor
	\[
	i_!=- \otimes_{e\Gamma e}^{\mathbf L} e\Gamma:
	\per(e\Gamma e)\longrightarrow\per\Gamma
	\]
	is fully faithful with essential image $\thick(e\Gamma)$, and
	extension of scalars along $q$ induces a triangle equivalence
	\[
	\per\Gamma/\thick(e\Gamma)\xrightarrow{\sim}\per D.
	\]
	If $(\bar Q,\bar W)$ is Jacobi-finite, then $p^*$ induces a triangle
	equivalence
	\[
	\per\Gamma/\thick(e\Gamma)\xrightarrow{\sim}\per\bar\Gamma.
	\]
	In particular, under this hypothesis there is an exact sequence of
	triangulated categories
	\[
	\per(e\Gamma e)\xrightarrow{i_!}\per\Gamma
	\xrightarrow{p^*}\per\bar\Gamma .
	\]
\end{Prop}

\begin{proof}
	On derived categories, the right adjoint of $i_!$ is given by
	$Y\mapsto Ye$. Since the unit of the adjunction is an isomorphism on
	the free module $e\Gamma e$, it is an isomorphism on all of
	$\per(e\Gamma e)$, and hence $i_!$ is fully faithful on perfect
	modules. Moreover, $i_!$ sends the generator $e\Gamma e$ to $e\Gamma$;
	because $\per(e\Gamma e)$ is idempotent complete, the essential image
	of $i_!$ is precisely $\thick(e\Gamma)$.
	
	By dg localization at the full subcategory of frozen vertices
	\cite{drinfeldDGQuotientsDG2004}, the induced functor
	\[
	\Phi\colon\per\Gamma/\thick(e\Gamma)\longrightarrow\per D
	\]
	is fully faithful and an equivalence up to direct summands. We now
	show that $\Phi$ is already essentially surjective.
	
	Set $A=H^0\Gamma$ and $R=H^0D=A/AeA$. Let $\mathfrak m$ be the image
	in $A$ of the arrow ideal of the completed path algebra, and let
	$\mathfrak r$ be the image of $\mathfrak m$ in $R$. 
	
	Since the completed path algebra is complete with respect to the arrow ideal, $1-x$ is invertible (with inverse $\sum_{n\ge 0} x^n$) for every element $x$ in the arrow ideal; thus the arrow ideal lies in its Jacobson radical, and hence $\mathfrak r\subseteq\operatorname{rad}R$.
	
	On the other hand, since the relations have positive path length, the
	quotient
	\[
	R/\mathfrak r\cong\prod_{i\in Q_0\setminus F_0} k e_i
	\]
	is semisimple, and hence $\mathfrak r=\operatorname{rad}R$.
	
	Let $P$ be a finitely generated projective right $R$-module. For each
	$i\in Q_0\setminus F_0$, set $m_i=\dim_k(P/P\mathfrak r)e_i$; lifting
	a $k$-basis of $(P/P\mathfrak r)e_i$ to $Pe_i$ yields a morphism
	\[
	\bigoplus_{i\in Q_0\setminus F_0}(e_iR)^{m_i}\longrightarrow P
	\]
	that becomes an isomorphism modulo $\mathfrak r$. By Nakayama's lemma,
	this morphism is surjective; since $P$ is projective, it splits. Its
	kernel is therefore finitely generated and vanishes modulo
	$\mathfrak r$, so a second application of Nakayama's lemma shows that
	the morphism is an isomorphism.
	
	Since $D$ is concentrated in nonpositive degrees,
	$\operatorname{add}(D)$ is a silting subcategory of $\per D$, and the
	associated bounded co-$t$-structure implies that
	$\operatorname{add}(D)$ generates $\per D$ as a triangulated category
	\cite[Proposition~2.8]{iyamaSiltingReductionCalabi2018}. The functor
	$H^0$ identifies $\operatorname{add}(D)$ with the category of finitely
	generated projective right $R$-modules. By the preceding paragraph,
	every object of $\operatorname{add}(D)$ is therefore a finite direct
	sum of copies of $e_iD$ ($i\in Q_0\setminus F_0$), each of which is the
	image of $e_i\Gamma$ under $\Phi$. Since $\Phi$ is fully faithful and
	exact, its essential image is a triangulated subcategory of $\per D$
	containing $\operatorname{add}(D)$, and thus coincides with all of
	$\per D$.
	
	Finally, if $(\bar Q,\bar W)$ is Jacobi-finite, then
	Proposition~\ref{Homtopy Cofiber sequence of (Q,F,W)} yields a
	quasi-isomorphism $r\colon D\to\bar\Gamma$, and extension of scalars
	along $r$ induces an equivalence
	$r^*\colon\per D\xrightarrow{\sim}\per\bar\Gamma$. Since
	$p^*=r^*q^*$, both the claimed equivalence and the exact sequence
	follow.
\end{proof}

%

\begin{Lem}\label{lem:perfect-drinfeld-quotient}
	Let $\ca$ be a small dg category over a field $k$, let $\cb\subseteq\ca$
	be a full dg subcategory, and let $q:\ca\to\mathcal Q=\ca/\cb$ be
	the Drinfeld dg quotient. For a dg category $\mathcal E$, write
	$\add H^0(\mathcal E)$ for the full subcategory of $\per(\mathcal E)$
	formed by the direct summands of finite direct sums of representable
	right dg modules.
	
	Identify $\per(\cb)$ with its image in $\per(\ca)$ under extension
	of scalars. The induced triangle functor
	\[
	F:\per(\ca)/\per(\cb)\longrightarrow\per(\mathcal Q)
	\]
	is fully faithful and induces a triangle equivalence
	\[
	\operatorname{Kar}\bigl(\per(\ca)/\per(\cb)\bigr)
	\xrightarrow{\sim}\per(\mathcal Q),
	\]
	where $\operatorname{Kar}$ denotes idempotent completion.
	
	Suppose, in addition, that $\mathcal Q$ is connective, i.e.\
	$H^j\mathcal Q(x,y)=0$ for all $x,y$ and all $j>0$, and that
	\[
	q^*:\add H^0(\ca)\longrightarrow\add H^0(\mathcal Q)
	\]
	is essentially surjective. Then $F$ itself is a triangle equivalence.
	The connectivity hypothesis on $\mathcal Q$ holds whenever $\ca$
	is connective.
\end{Lem}

\begin{proof}
	For a dg category $\mathcal E$, let $\mathcal E^{\mathrm{tr}}$ denote
	the homotopy category of its pretriangulated hull, so that
	$\per(\mathcal E)\simeq\operatorname{Kar}(\mathcal E^{\mathrm{tr}})$
	canonically. Since $k$ is a field, Drinfeld's quotient theorem
	\cite[Theorem~3.4]{drinfeldDGQuotientsDG2004} yields a triangle
	equivalence
	\[
	\ca^{\mathrm{tr}}/\cb^{\mathrm{tr}}
	\xrightarrow{\sim}\mathcal Q^{\mathrm{tr}}.
	\]
	The idempotent completion of $\ca^{\mathrm{tr}}/\cb^{\mathrm{tr}}$ is
	canonically equivalent to
	$\operatorname{Kar}\bigl(\per(\ca)/\per(\cb)\bigr)$: indeed, exact
	functors from either category into an idempotent-complete triangulated
	category correspond to exact functors out of $\ca^{\mathrm{tr}}$ that
	annihilate $\cb^{\mathrm{tr}}$. Passing to idempotent completions
	therefore gives the asserted equivalence. Since the canonical functor
	from a triangulated category to its idempotent completion is fully
	faithful, it follows that $F$ is fully faithful.
	
	Assume now that $\mathcal Q$ is connective. Set
	$\mathcal M=\add H^0(\mathcal Q)$ and write $h_x=\mathcal Q(-,x)$
	for $x\in\mathcal Q$. By the dg Yoneda lemma, for all
	$x,y\in\mathcal Q$ and $j>0$ we have
	\[
	\Hom_{\per(\mathcal Q)}(h_x,\Sigma^j h_y)
	\cong H^j\mathcal Q(x,y)=0,
	\]
	and this vanishing is preserved under finite direct sums and direct
	summands. Since the representable dg modules thickly generate
	$\per(\mathcal Q)$, it follows that $\mathcal M$ is a silting
	subcategory of $\per(\mathcal Q)$. By
	\cite[Proposition~2.8]{iyamaSiltingReductionCalabi2018}, every object
	of $\per(\mathcal Q)$ can therefore be obtained as a finite iterated
	extension of shifts of objects in $\mathcal M$.
	
	Note that $\mathcal M$ lies in the essential image of $F$. Because $F$
	is fully faithful and exact, its essential image is closed under
	shifts and cones: any morphism between two objects in the essential
	image lifts to $\per(\ca)/\per(\cb)$, where its cone can be formed.
	Consequently, the essential image of $F$ contains every such iterated
	extension, showing that $F$ is essentially surjective.
	
	Finally, if $\ca$ is connective, taking the nonpositive truncations of
	its morphism complexes yields a quasi-equivalent dg category
	$\tau_{\leq 0}\ca$ containing $\tau_{\leq 0}\cb$ as a full dg
	subcategory, and the corresponding dg quotient is quasi-equivalent to
	$\mathcal Q$ \cite[Section~3.6(ii)]{drinfeldDGQuotientsDG2004}. Notice that all morphisms of $\tau_{\leq 0}\ca$ have nonpositive
	degree, while the contracting endomorphisms adjoined at objects of
	$\tau_{\leq 0}\cb$ have degree $-1$. Hence the quotient has no
	morphisms of positive degree, which proves that $\mathcal Q$ is
	connective.
\end{proof}

\section{Relative cluster categories and Higgs categories}\label{Section: Relative cluster categories and Higgs Categories}
Let $ (Q,F,W) $ be an ice quiver with potential. 
Denote by $ \bm\Gamma $ the associated complete relative Ginzburg dg algebra $\bmgamma(Q,F,W)$. Let $e=\sum_{i\in F_0}e_i$ be the idempotent associated with the set of frozen vertices and $ \cp $ the additive subcategory $ \add(e\bmgamma) $ of $ \per\bmgamma $.
\begin{Def}\rm
	The \emph{relative cluster category} $ \cc(Q,F,W) $ (or $ \cc $) of $ (Q,F,W) $ is defined as the idempotent completion of the Verdier quotient of triangulated categories $$ \mathrm{per}(\bm\Gamma)/\mathrm{pvd}_{e}(\bm\Gamma) .$$
	If $ F=\emptyset $, the \emph{cluster category} associated with $ (Q,W) $ is defined as $ \cc(Q,\emptyset,W) $ and we denote it by $ \cc(Q,W) $.
\end{Def}

\begin{Rem}
	If $(Q,F,W)$ is Jacobi-finite, then $(\overline Q,\overline W)$ is
	also Jacobi-finite, since $J(\overline Q,\overline W)$ is a quotient
	of $J(Q,F,W)$. Whenever $(\overline Q,\overline W)$ is Jacobi-finite,
	both Verdier quotients $\per\bmgamma/\pvd_e(\bmgamma)$ and
	$\per\overGamma/\pvd(\overGamma)$ are idempotent complete
	by Proposition~\ref{Prop: stable finite to dense} below and its
	special case $F=\emptyset$, respectively.
\end{Rem}

We denote by $ \overline{\cd} $ the unbounded derived category of $ \bmgamma(\overline{Q},\overline{W}) $ and by $ \overc $ the cluster category $ \cc(\overline{Q},\overline{W}) $ associated with $ (\overline{Q},\overline{W}) $.

\begin{Prop}
	Assume that $(\overline{Q},\overline{W})$ is Jacobi-finite.
	The following diagram commutes up to natural isomorphism
	\[
	\begin{tikzcd}
		& \thick_{\per\bmgamma}(\cp)
		\arrow[r,"\pi^{rel}","\sim"'] \arrow[d,hook]
		& \thick_{\cc}(\pi^{rel}\cp) \arrow[d,hook] \\
		\pvd_{e}(\bmgamma)
		\arrow[r,hook] \arrow[d,"p^{*}","\sim"']
		& \per\bmgamma
		\arrow[r,"\pi^{rel}"] \arrow[d,two heads,"p^{*}"]
		& \cc(Q,F,W) \arrow[d,two heads,"p^{*}"] \\
		\pvd(\overGamma) \arrow[r,hook]
		& \per\overGamma \arrow[r,"\pi"]
		& \cc(\overline{Q},\overline{W}).
	\end{tikzcd}
	\]
	The middle and right columns are Verdier localization sequences.
	In each of the two lower rows, the horizontal projection is the
	Verdier quotient functor followed by the canonical embedding into
	its idempotent completion.
\end{Prop}
\begin{proof}
	By Proposition~\ref{Homtopy Cofiber sequence of (Q,F,W)},
	vertex deletion realizes the derived idempotent quotient by $e$.
	Thus restriction of scalars $p_*$ is fully faithful. It sends the simple $\overGamma$-modules to the corresponding simple
	$\bmgamma$-modules at the unfrozen vertices. Since these simple
	modules thickly generate the respective categories in the left column,
	$p^*$ and $p_*$ restrict to mutually inverse equivalences between them.
	
	For $N\in\pvd_e(\bmgamma)$ and $j\in\mathbb Z$, we have
	$\Hom_{\per\bmgamma}(e\bmgamma,\Si^jN)=H^j(Ne)=0$.
	Hence $\pi^{rel}$ is fully faithful on
	$\thick_{\per\bmgamma}(\cp)$. This subcategory is idempotent
	complete, so its image in $\cc$ is exactly
	$\thick_{\cc}(\pi^{rel}\cp)$.
	The middle column is Proposition~\ref{prop:perfect-frozen-quotient}.
	The right column follows from the induced equivalence
	$$\cc/\thick_{\cc}(\pi^{rel}\cp)\simeq\overc$$ proved in
	Proposition~\ref{Prop:silting reduction-fully faithful in cc} below.
	The description of the rows and the commutativity follow from
	the construction of the induced functors.
\end{proof}

\begin{Prop}\cite[Lemma 2.17]{kellerYangDerived equivalences}
	The perfect derived categories $ \per\relGammabf $ and $ \per\overGamma $ are Krull–Schmidt categories.
\end{Prop}

Let $ \cd^{\leqslant0} $ (and $ \cd^{\geqslant0} $ respectively) be the full subcategory of $ \cd(\bmgamma) $ whose objects are those $ X $ whose homology is concentrated in non-positive (and non-negative, respectively) degrees. Since $ \bmgamma $ is connective, the pair $ (\cd^{\leqslant0},\cd^{\geqslant0}) $ is a canonical $ t $-structure on $ \cd(\bmgamma) $. Similarly, we define $ \overline{\cd}^{\leqslant0} $ and $ \overline{\cd}^{\geqslant0} $. And the pair $ (\overline{\cd}^{\leqslant0},\overline{\cd}^{\geqslant0}) $ is a canonical $ t $-structure on $ \overline{\cd} $.

For the construction of the relative $t$-structure through
Corollary~\ref{Cor: relatove truncation functor}, assume that
$(\overline Q,\overline W)$ is Jacobi-finite.
By Proposition~\ref{Homtopy Cofiber sequence of (Q,F,W)}, the
ordinary derived idempotent quotient of $\bmgamma$ by $e$ is
quasi-isomorphic to $\overGamma$. Thus the idempotent recollement
of $\cd(\bmgamma)$ has outer terms $\overline{\cd}$ and
$\cd(e\bmgamma e)$, with $p_*$ identifying $\overline{\cd}$
with the kernel of restriction $X\mapsto Xe$.

On $\overline{\cd}$, we take the canonical $t$-structure with heart
$\mathrm{Mod}(J(\overline Q,\overline W))$, and on
$\cd(e\bmgamma e)$ we take the trivial $t$-structure whose left
aisle is $\cd(e\bmgamma e)$. We denote by $(i^{*},i_{*},i^{!})$
the triple of adjoint functors associated with the full dg
subcategory on the frozen vertices, so $i_{*}(X)=Xe$.

\begin{Prop}\label{Relative t-structure}\cite[Section~4]{wuRelativeClusterCategories2021}
	There is a t-structure in $ \mathcal{D}(\bmgamma) $ obtained by gluing the canonical t-structure on $ \mathcal{D}(\overGamma) $ with the trivial t-structure on $ \mathcal{D}(e\bmgamma e) $ through a recollement diagram. 
	
	We denote by $ (\mathcal{D}_{rel}^{\leqslant n},	\mathcal{D}_{rel}^{\geqslant n}) $ the glued t-structure on $ \cd(\bmgamma) $. For each $ n\in\mathbb{Z} $, 
	\begin{align*}
		\mathcal{D}_{rel}^{\leqslant n}
		&=\{X\in\mathcal{D}(\bmgamma)\mid
		H^{l}(p^{*}X)=0\ \text{for all }l>n\},\\
		\mathcal{D}_{rel}^{\geqslant n}
		&=\{X\in\mathcal{D}(\bmgamma)\mid i_{*}(X)\simeq0,\
		H^{l}(X)=0\ \text{for all }l<n\}.
	\end{align*}
	The heart $\heartsuit^{rel}$ is equivalent to
	$\mathrm{Mod}_{e}H^{0}(\relGammabf)$, the category of right
	$H^{0}(\relGammabf)$-modules $M$ satisfying $Me=0$. We will call $ (\mathcal{D}_{rel}^{\leqslant n},	\mathcal{D}_{rel}^{\geqslant n}) $ the \emph{relative t-structure} on $ \mathcal{D}(\bmgamma) $. We illustrate this glued t-structure in the following picture
	
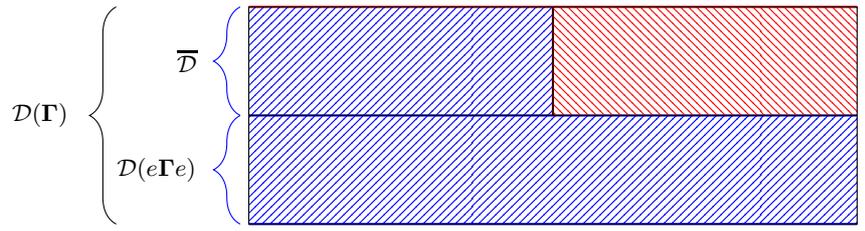
\begin{align*}
		\begin{tikzpicture}[scale=1.6]
			\draw [thick,red] (0,1.8)--(5,1.8);
			\draw [thick,blue] (0,0.9)--(5,0.9);
			\draw [thick,blue] (0,0)--(5,0);
			\draw[thick,red](2.5,1.8)--(2.5,0.9);
			\draw[decorate,decoration={brace,amplitude=10pt},xshift=-31pt,yshift=0pt](0,0)--(0,1.8)node [black,midway,xshift=-1cm] {\footnotesize $\cd(\bmgamma)$};
			\draw[decorate,blue,decoration={brace,amplitude=10pt},xshift=-2pt,yshift=0pt](0,0)--(0,0.9)node [black,midway,xshift=-1.1cm] {\footnotesize $\cd(e\bmgamma e)$};
			\draw[decorate,blue,decoration={brace,amplitude=10pt},xshift=-2pt,yshift=0pt](0,0.9)--(0,1.8)node [black,midway,xshift=-0.7cm] {\footnotesize $\overline{\cd}$};
			\draw[pattern=north east lines, pattern color=blue] (0,0) rectangle (5,0.9);
			\draw[pattern=north east lines, pattern color=blue] (0,0.9) rectangle (2.5,1.8);
			\draw[pattern=north west lines, pattern color=red] (2.5,0.9) rectangle (5,1.8);
		\end{tikzpicture}\,,
	\end{align*} 
	where the blue region represents the subcategory $ \cd_{rel}^{\leqslant0} $ and the red region represents the subcategory $ \cd_{rel}^{\geqslant0} $.
	
\end{Prop}

\begin{Def}\rm\label{Relative truncation functor}
	We define the \emph{relative truncation functor} $ \tau^{rel}_{>n} $ to be the following composition
	$$
	\xymatrix{
		\tau^{rel}_{> n}\colon\mathcal{D}(\relGammabf)\ar[r]^-{p^{*}}&\mathcal{D}(\overGamma)\ar[r]^{\tau_{> n}}&\mathcal{D}(\overGamma)\ar[r]^{p_{*}}&\mathcal{D}(\relGammabf).
	}$$
\end{Def}
Thus, for any $ X\in\mathcal{D}(\relGammabf) $, we have a canonical triangle in $ \cd(\relGammabf) $
$$ \tau_{\leqslant n}^{rel}X\to X\to \tau_{>n}^{rel}X\to\Si\tau_{\leqslant n}^{rel}X $$
such that $ \tau_{\leqslant n}^{rel}X $ belongs to $ \mathcal{D}_{rel}^{\leqslant n} $ and $ \tau^{rel}_{>n}(X)=p_{*}(\tau_{>n}(p^{*}X)) $ belongs to $ \mathcal{D}_{rel}^{\geqslant n+1} $.

\begin{Cor}\label{Cor: relatove truncation functor}
	If $ (\overline{Q},\overline{W}) $ is Jacobi-finite, the relative $ t $-structure on $ \cd(\bmgamma) $ restricts to the perfect derived category $ \per\bmgamma $. 
\end{Cor}
\begin{proof}
	Let $ X $ be an object in $ \per\bmgamma $. Consider the canonical triangle with respect to the relative $ t $-structure on $ \cd(\Gamma) $
	$$ \tau_{\leqslant 0}^{rel}X\ra X\ra\tau_{>0}^{rel}X\ra\Si\tau^{rel}_{\leqslant 0}X .$$
	
	The object $p^*X$ is perfect over the connective dg algebra
	$\overGamma$, hence bounded above. Its cohomology groups are
	finite-dimensional since $\per\overGamma$ is Hom-finite in the
	Jacobi-finite case
	\cite[Proposition~2.5]{kalckRelativeSingularityCategories2016}.
	It follows that $\tau_{>0}p^*X$ has finite-dimensional total
	cohomology, supported in degrees $0<l\leqslant N$ for some $N\geqslant 0$.
	Restriction of scalars preserves the underlying complex and
	$e$ acts by zero. Hence
	$\tau_{>0}^{rel}X=p_*(\tau_{>0}p^*X)$ belongs to
	$\pvd_e(\bmgamma)\subseteq\per\bmgamma$.
	The triangle then gives $\tau_{\leqslant0}^{rel}X\in\per\bmgamma$.
	Shifting gives the assertion in every degree.
\end{proof}

\subsection{Silting reduction}
Recall that a full subcategory $ \cp $ of a triangulated category $ \ct $ is \emph{presilting} if $ \Hom_{\ct}(\cp,\Si^{i}\cp)=0 $ for any $ i > 0 $. It is \emph{silting} if in addition $ \ct=\thick\cp $. It is clear that $ \cp=\add(e\bmgamma) $ is a presilting subcategory of $ \per\bmgamma $.

Let $ \cz $ be the following subcategory of $ \per\bmgamma $
$$ \cz={}^{\perp}(\Si^{>0}\cp)\cap(\Si^{<0}\cp)^{\perp}. $$

\begin{Prop}\label{Prop:silting reduction-fully faithful}
	Write $\Gamma=\bmgamma$, let $\ct=\per\Gamma$, $\cp=\add(e\Gamma)$, and
	\[
	\cz={}^{\perp}(\Sigma^{>0}\cp)
	\cap(\Sigma^{<0}\cp)^{\perp}.
	\]
	Let $q:\Gamma\to D$ be the ordinary derived idempotent quotient by $e$,
	as in Proposition~\ref{prop:perfect-frozen-quotient}.
	The Verdier quotient functor induces a fully faithful $k$-linear functor
	\[
	\cz/[\cp]\lhook\joinrel\longrightarrow
	\ct/\thick(\cp)\xrightarrow{\sim}\per D.
	\]
	The composite is induced by $q^*=-\otimes_\Gamma^{\mathbf L}D$.
	If $(\bar Q,\bar W)$ is Jacobi-finite, vertex deletion
	$p:\Gamma\to\bar\Gamma$ therefore induces a fully faithful functor
	\[
	p^*:\cz/[\cp]\lhook\joinrel\longrightarrow\per\bar\Gamma.
	\]
\end{Prop}

\begin{proof}
	Put $\mathcal S=\thick_{\ct}(\cp)$ and let
	$\pi:\ct\to\ct/\mathcal S$ be the quotient. Since $\Gamma$ is
	nonpositive, $\cp$ is presilting in $\ct$ and silting in $\mathcal S$.
	We verify the Hom-space assertion of
	\cite[Lemma~3.4]{iyamaSiltingReductionCalabi2018} in this notation.
	
	For an interval $I\subseteq\mathbb Z$, write $\mathcal S_I$ for the
	closure under finite extensions and direct summands of the objects
	$\Sigma^jP$ with $j\in I$ and $P\in\cp$. Use subscripts such as
	$\le 0$ and $>0$ for the corresponding intervals.
	The bounded co-$t$-structure on $\mathcal S$ associated with $\cp$
	gives, for each $C\in\mathcal S$, a triangle
	\[
	U\longrightarrow C\longrightarrow V\longrightarrow\Sigma U,
	\qquad U\in\mathcal S_{\le0},\quad V\in\mathcal S_{>0}.
	\]
	Moreover, $U$ admits a triangle
	$U_-\to U\to P_0\to\Sigma U_-$ with
	$U_-\in\mathcal S_{<0}$ and $P_0\in\cp$
	\cite[Proposition~2.8]{iyamaSiltingReductionCalabi2018}.
	
	Fix $X,Y\in\cz$. By orthogonality,
	$\Hom_{\ct}(X,\mathcal S_{>0})=0$ and
	$\Hom_{\ct}(\mathcal S_{<0},Y)=0$. Represent a morphism
	$\pi X\to\pi Y$ by a roof $X\xleftarrow{s}Z\xrightarrow{g}Y$,
	and complete $s$ to a triangle
	\[
	Z\xrightarrow{s}X\xrightarrow{a}C\longrightarrow\Sigma Z,
	\qquad C\in\mathcal S.
	\]
	Since $\Hom_{\ct}(X,V)=0$, the map $a$ factors through $U\to C$.
	Completing the resulting map $X\to U$ to a triangle gives a
	refinement $Z'\to Z$ of the roof with denominator $s':Z'\to X$
	and cone $U$. Now $\Sigma^{-1}U\in\mathcal S_{<0}$, so
	$\Hom_{\ct}(X,Y)\to\Hom_{\ct}(Z',Y)$ is surjective.
	Thus the refined roof is represented by a morphism $X\to Y$ in
	$\ct$. This proves fullness.
	
	A morphism $f:X\to Y$ becomes zero in $\ct/\mathcal S$ precisely
	when it factors through some $C\in\mathcal S$. In such a
	factorization, the map $X\to C$ factors through $U\to C$ as above.
	The resulting map $U\to Y$ factors through $U\to P_0$ because
	$\Hom_{\ct}(U_-,Y)=0$. Hence $f$ factors through $\cp$.
	Conversely, $\pi$ kills $\cp$. We have therefore proved
	\[
	\Hom_{\ct}(X,Y)/[\cp](X,Y)
	\xrightarrow{\sim}\Hom_{\ct/\mathcal S}(\pi X,\pi Y).
	\]
	
	Proposition~\ref{prop:perfect-frozen-quotient} identifies
	$\ct/\mathcal S$ with $\per D$ via $q^*$.
	When $(\bar Q,\bar W)$ is Jacobi-finite, the same proposition
	identifies $\ct/\mathcal S$ with $\per\bar\Gamma$ via $p^*$.
	This gives both assertions.
\end{proof}
%

\bigskip
Similarly, the category $ \cp=\add (e\bmgamma)\subset\cc=\cc(Q,F,W) $ is a presilting subcategory of $ \cc $. Let $ \cy $ be the following subcategory of $ \cc $
$$ \cy=^{\perp}\!(\Si^{>0}\cp)\cap(\Si^{<0}\cp)^{\perp}. $$

\begin{Prop}\label{Prop:silting reduction-fully faithful in cc}
	Let $\pi^{\mathrm{rel}}:\per\Gamma\to\cc$ be the canonical functor.
	Put $P=\pi^{\mathrm{rel}}(e\Gamma)$ and $\cp=\add(P)$, and take
	the following orthogonals in $\cc$
	\[
	\cy={}^{\perp}(\Sigma^{>0}\cp)\cap(\Sigma^{<0}\cp)^{\perp}.
	\]
	The Verdier quotient induces a fully faithful $k$-linear functor
	\[
	\rho:\cy/[\cp]\lhook\joinrel\longrightarrow
	\cc/\thick_{\cc}(\cp).
	\]
	Then the vertex deletion map $p:\Gamma\to\bar\Gamma$ induces a triangle functor
	\[
	\widehat p:\cc/\thick_{\cc}(\cp)\longrightarrow\overc.
	\]
	If $(\bar Q,\bar W)$ is Jacobi-finite, then $\widehat p$ is a
	triangle equivalence. In particular, under this hypothesis
	$\widehat p\circ\rho:\cy/[\cp]\to\overc$ is fully faithful.
\end{Prop}

\begin{proof}
	Write $\ct=\per\Gamma$ and $\mathcal N=\pvd_e(\Gamma)$, so that
	$\cc=\operatorname{Kar}(\ct/\mathcal N)$, where
	$\operatorname{Kar}$ denotes idempotent completion.
	For $N\in\mathcal N$ and $j\in\mathbb Z$ we have
	\[
	\Hom_{\ct}(e\Gamma,\Sigma^jN)=H^j(Ne)=0.
	\]
	Thus $\Hom_{\ct}(e\Gamma,-)$ inverts every morphism with cone in
	$\mathcal N$. Computing morphisms out of $e\Gamma$ by Verdier
	fractions, and using the fully faithful inclusion
	$\ct/\mathcal N\hookrightarrow\cc$, gives
	\[
	\Hom_{\cc}(P,\Sigma^jP)\cong H^j(e\Gamma e).
	\]
	This vanishes for $j>0$ since $\Gamma$ is nonpositive. Hence $\cp$
	is presilting in $\cc$. Applying the Hom-space calculation of
	\cite[Lemma~3.4]{iyamaSiltingReductionCalabi2018} to $(\cc,\cp)$,
	as in the proof of Proposition~\ref{Prop:silting reduction-fully faithful},
	gives, for $X,Y\in\cy$,
	\[
	\Hom_{\cc}(X,Y)/[\cp](X,Y)
	\xrightarrow{\sim}
	\Hom_{\cc/\thick_{\cc}(\cp)}(\rho X,\rho Y).
	\]
	Here $[\cp](X,Y)$ consists of the maps factoring through an object
	of $\cp$. This proves the first assertion.
	
	Put $\overline{\ct}=\per\bar\Gamma$ and
	$\overline{\mathcal N}=\pvd(\bar\Gamma)$.
	For each unfrozen vertex $i$, let $S_i$ and $\bar S_i$ denote the
	corresponding simple modules. The finite semi-free resolution used
	in the proof of Proposition~\ref{Homtopy Cofiber sequence of (Q,F,W)}
	gives $p^*S_i\simeq\bar S_i$: tensoring it with $\bar\Gamma$
	deletes the frozen projective summands and yields the same
	resolution of $\bar S_i$. This computation applies to every
	$(\bar Q,\bar W)$. By truncation and finite-length filtrations,
	the $S_i$ generate $\mathcal N$ as a thick subcategory, and the
	$\bar S_i$ generate $\overline{\mathcal N}$. It follows that
	\[
	\thick_{\overline{\ct}}(p^*\mathcal N)=\overline{\mathcal N}.
	\]
	Thus $p^*$ descends to the Verdier quotients and extends to their
	idempotent completions. Since $p^*(e\Gamma)=0$, the resulting
	functor $\cc\to\overc$ factors through the Verdier quotient
	$\cc\to\cc/\thick_{\cc}(\cp)$, giving $\widehat p$.
	
	Now suppose that $(\bar Q,\bar W)$ is Jacobi-finite. By
	Proposition~\ref{prop:perfect-frozen-quotient}, $p^*$ identifies
	$\ct/\thick_{\ct}(e\Gamma)$ with $\overline{\ct}$. By the
	preceding equality and transitivity of Verdier localization,
	\[
	\ct/\thick_{\ct}(\mathcal N\cup\{e\Gamma\})
	\xrightarrow{\sim}\overline{\ct}/\overline{\mathcal N}.
	\]
	Consequently $\widehat p$ induces an equivalence
	\[
	\operatorname{Kar}\bigl(\cc/\thick_{\cc}(\cp)\bigr)
	\xrightarrow{\sim}\overc .
	\]
	Indeed, exact functors from either side to an idempotent-complete
	triangulated category correspond to exact functors from $\ct$
	annihilating both $\mathcal N$ and $e\Gamma$. In particular,
	$\widehat p$ is fully faithful.
	
	It remains to check essential surjectivity before idempotent
	completion. Since $(\bar Q,\bar W)$ is Jacobi-finite,
	$\overline{\ct}/\overline{\mathcal N}$ is already idempotent
	complete (see~\cite{amiotClusterCategoriesAlgebras2009}).
	Hence every object of $\overc$ is represented by an object of
	$\overline{\ct}$. By Proposition~\ref{prop:perfect-frozen-quotient},
	the latter is isomorphic to $p^*X$ for some $X\in\ct$.
	This proves essential surjectivity of $\widehat p$.
\end{proof}

\begin{Rem}
	Condition (P1) of \cite[Section~3.1]{iyamaSiltingReductionCalabi2018}
	need not hold for $\cp$. Thus the density of the fully faithful
	functors on $\cz/[\cp]$ and $\cy/[\cp]$ in
	Propositions~\ref{Prop:silting reduction-fully faithful}
	and~\ref{Prop:silting reduction-fully faithful in cc}
	requires a separate argument.
\end{Rem}

\subsection{SMC reduction}\label{Section: SMC}
Let $ \cs $ be the subcategory of $ \per\relGammabf $ formed by the modules $ S_{i} $ associated with unfrozen vertices $ i\in Q_{0}\setminus F_{0} $. Then $ \cs $ is a simple-minded collection (\cite[Definition 2.4]{Jin2019}) of $ \pvd_{e}(\bm{\Gamma}) $ and a pre-simple-minded collection of $ \per(\bm{\Gamma}) $.

Consider the following subcategory of $ \per\relGammabf $
$$ \cw=(\Si^{\geqslant0}\cs)^{\perp}\cap{ }^{\perp}(\Si^{\leqslant0}\cs).$$

\begin{Prop}\label{Prop: SMC fully faithful}
	The composition $\cw\subseteq\per(\bm{\Gamma})
	\xrightarrow{\pi^{rel}}\cc(Q,F,W)$ induces a fully faithful embedding
	\[
	\pi^{rel}_{\cw}\colon\cw\hookrightarrow\cc(Q,F,W).
	\]
\end{Prop}
\begin{proof}
	Put $\mathcal N=\pvd_e(\bmgamma)$. Its objects have bounded
	finite-dimensional cohomology supported on the unfrozen vertices.
	Thus their standard truncations belong to $\mathcal N$, and each
	cohomology module is a finite extension of objects of $\cs$.
	We check full faithfulness in the Verdier quotient by $\mathcal N$,
	which embeds fully faithfully in its idempotent completion.
	
	Let $U,V\in\cw$. A map $U\to V$ which vanishes in the quotient
	factors through an object $N\in\mathcal N$. We have
	$\Hom(\tau_{\leq0}N,V)=0$ and $\Hom(U,\tau_{\geq1}N)=0$
	by the definition of $\cw$ and the cohomology filtration.
	The truncation triangle of $N$ then shows that the composite
	$U\to N\to V$ vanishes.
	
	For fullness, represent a morphism in the quotient by a roof
	$U\xrightarrow{f}T\xleftarrow{s}V$ and complete $s$ to a triangle
	$V\to T\to N\xrightarrow{h}\Si V$ with $N\in\mathcal N$.
	Since $\Hom(\tau_{\leq-1}N,\Si V)=0$, the map $h$ factors through
	$B=\tau_{\geq0}N$. Completing the resulting map $B\to\Si V$
	to a triangle gives a morphism of triangles
	\[
	\begin{tikzcd}
		V\arrow[r,"s"]\arrow[d,equal]&T\arrow[r]\arrow[d,"g"]&
		N\arrow[r,"h"]\arrow[d]&\Si V\arrow[d,equal]\\
		V\arrow[r,"t"']&T'\arrow[r]&B\arrow[r]&\Si V .
	\end{tikzcd}
	\]
	The two roofs with numerators $f$ and $gf$ define the same
	morphism in the quotient. As $\Hom(U,B)=0$, we have $gf=ta$
	for some $a\colon U\to V$. Hence the original roof is the
	image of $a$.
\end{proof}

\begin{Rem}
	The denseness of $\pi^{rel}_{\cw}$ requires an additional argument since (R1) of
	\cite[Theorem~3.1]{Jin2019} need not hold in general.
\end{Rem}

\begin{Lem}\label{Lem: stable finite to functorially finite simples}
	If $(\overline Q,\overline W)$ is Jacobi-finite, then
	$\operatorname{Filt}(\cs)$ is functorially finite in $\per\bmgamma$.
\end{Lem}

\begin{proof}
	Let $X\in\per\bmgamma$. For each unfrozen vertex $i$,
	$\bigoplus_j\Hom(X,\Si^j S_i)$ is finite-dimensional.
	Indeed, this holds for $X=\bmgamma$ and the property is preserved
	by shifts, cones and direct summands. By relative $3$-Calabi--Yau
	duality (Proposition~\ref{Prop:relative 3-CY}), it implies the same
	finiteness for $\bigoplus_j\Hom(S_i,\Si^j X)$. Taking bases of
	$\Hom(X,S_i)$ and $\Hom(S_i,X)$ for all $i$ gives a left and a
	right $\add(\cs)$-approximation of $X$, respectively.
	
	Put $B=H^0(\overGamma)$ and choose $d\geq1$ with
	$(\operatorname{rad}B)^d=0$. The category
	$\operatorname{Filt}(\cs)$ consists of the finite-dimensional
	modules supported on the unfrozen vertices, viewed in degree zero.
	These are $B$-modules by the degree-zero identification in
	Proposition~\ref{Homtopy Cofiber sequence of (Q,F,W)}.
	The radical filtration of each such module has at most $d$
	semisimple layers. It follows that
	\[
	\operatorname{Filt}(\cs)
	=\underbrace{\add(\cs)*\cdots *\add(\cs)}_{d\ \mathrm{factors}}.
	\]
	By \cite[Theorem~1.3]{Chen2009Extensions} and its dual,
	the right-hand side
	is both covariantly and contravariantly finite in $\per\bmgamma$.
\end{proof}

\begin{Prop}\label{Prop: stable finite to dense}
	If $(\overline Q,\overline W)$ is Jacobi-finite, the fully faithful
	functor $\pi^{rel}_{\cw}$ is dense. Thus we obtain a $k$-linear
	equivalence
	\[
	\pi^{rel}_{\cw}\colon\cw\iso\cc(Q,F,W).
	\]
\end{Prop}
\begin{proof}
	The preceding lemma gives condition (R1) of
	\cite[Theorem~3.1]{Jin2019}. For a fixed perfect object $X$, the
	vanishing of $\Hom(X,\Si^j S_i)$ and $\Hom(S_i,\Si^j X)$ for
	$|j|\gg0$ is uniform in $i$, since there are only finitely many
	unfrozen vertices. The same vanishing holds for all objects of
	$\operatorname{Filt}(\cs)$ by extension, with the same bound.
	This gives (R2). Since $\per\bmgamma$ is Krull--Schmidt and
	$\pvd_e(\bmgamma)=\thick(\cs)$, Jin's theorem identifies $\cw$
	with the Verdier quotient $\per\bmgamma/\pvd_e(\bmgamma)$.
	The orthogonality conditions defining $\cw$ are preserved by
	direct summands in $\per\bmgamma$. Hence $\cw$ is idempotent
	complete, so this quotient is already equivalent to $\cc(Q,F,W)$.
\end{proof}

\subsection{Higgs categories}
Let $ \ct $ be any triangulated category. Let $ \ct' $ be a full subcategory of $ \ct $. We denote by $ \pr_{\ct}\ct' $ the full subcategory of $ \ct $ whose objects are cones of morphisms in $ \add\ct' $. Similarly, we denote by $ \copr_{\ct}\ct' $ the full subcategory of $ \ct $ whose objects are those $ X $ such that $ \Si X $ is in $ \pr_{\ct}\ct' $. If $ \ct'=\add T $ for some object $ T\in\ct $, the categories $ \pr_{\ct}\ct' $ and $ \copr_{\ct}\ct' $ will be simply denoted by $ \pr_{\ct}T $ and $ \copr_{\ct}T $ respectively.
\begin{Lem}\cite[Lemma 2.11]{plamondonCategoriesAmasseesAux2011}\label{Lemma: pr is extension closed}
	We have 
	$$ \pr_{\cd}\bmgamma=\cd^{\leqslant0}\cap{}^{\perp}\cd^{\leqslant-2}\cap\per\bmgamma. $$
	Thus, the category $ \pr_{\cd}\bmgamma $ is an extension closed subcategory of $ \per\bmgamma $.
	
\end{Lem}
\begin{proof}
	The minimal perfect model argument of
	\cite[Lemmas~2.11 and~2.14]{plamondonCategoriesAmasseesAux2011}
	applies to the complete relative Ginzburg algebra as well. Its
	differential preserves the arrow ideal, and invertible entries
	of the attaching maps split off contractible summands.
	In a minimal model for $X$, the multiplicity of $\Si^p\bmgamma_i$
	is $\dim_k\Hom_{\cd}(X,\Si^p S_i)$, where $i$ ranges over all
	vertices. If $X\in\cd^{\leqslant0}\cap{}^{\perp}\cd^{\leqslant-2}$,
	these spaces vanish for $p<0$ and for $p\geq2$. Thus a perfect
	$X$ in this intersection has an $\add\bmgamma$-presentation.
	The reverse inclusion follows by applying the two orthogonality
	conditions to a presentation triangle.
\end{proof}

\begin{Prop}\label{Prop:fully faithful of pr}
	The quotient functor $ \pi^{rel}\colon\per\bmgamma\ra\cc $ restricts to an equivalence of $ k $-linear categories $ \pr_{\cd}\bmgamma\iso\pr_{\cc}\bmgamma $.
\end{Prop}
\begin{proof}
	First, we show that $\pi^{\mathrm{rel}}\colon\pr_{\cd}\bmgamma\ra\pr_{\cc}\bmgamma$
	is fully faithful. Let $X,Y\in\pr_{\cd}\bmgamma$; in particular, both
	$X$ and $Y$ lie in $\cd^{\leqslant0}(\bmgamma)$.
	
	To prove faithfulness, suppose that a morphism $f\colon X\ra Y$ maps
	to zero in $\cc$, so that it factors as
	\[
	X\xrightarrow{g}M\xrightarrow{h}Y
	\]
	for some $M\in\pvd_{e}(\bmgamma)$. Since $X\in\cd^{\leqslant0}(\bmgamma)$,
	we have $X\simeq\tau_{\leqslant1}X$, so $g$ factors through
	$\tau_{\leqslant1}M$, which still lies in $\pvd_{e}(\bmgamma)$. By
	Proposition~\ref{Prop:relative 3-CY},
	\[
	D\Hom_{\cd(\bmgamma)}(\tau_{\leqslant1}M,Y)
	\cong\Hom_{\cd(\bmgamma)}(Y,\Si^{3}\tau_{\leqslant1}M)=0,
	\]
	because $Y\in{}^{\perp}\cd^{\leqslant-2}$ and
	$\Si^{3}\tau_{\leqslant1}M\in\cd^{\leqslant-2}$. Hence $f=0$, showing
	that $\pi^{\mathrm{rel}}|_{\pr_{\cd}\bmgamma}$ is faithful.
	
	To prove fullness, let $f'\colon\pi^{\mathrm{rel}}X\ra\pi^{\mathrm{rel}}Y$
	be a morphism in $\cc$, represented by a fraction
	\[
	X\xrightarrow{f}Y'\xleftarrow{s}Y
	\]
	whose cone $N$ of $s$ lies in $\pvd_{e}(\bmgamma)$. Then both
	$\tau_{\leqslant0}N$ and $\tau_{\geqslant1}N$ belong to
	$\pvd_{e}(\bmgamma)$. By Proposition~\ref{Prop:relative 3-CY} and the
	fact that $Y\in{}^{\perp}\cd^{\leqslant-2}$, we have
	\[
	D\Hom_{\cd(\bmgamma)}(\tau_{\leqslant0}N,\Si Y)
	\cong\Hom_{\cd(\bmgamma)}(Y,\Si^{2}\tau_{\leqslant0}N)=0,
	\]
	so $\Hom_{\cd(\bmgamma)}(\tau_{\leqslant0}N,\Si Y)=0$. Applying
	$\Hom_{\cd(\bmgamma)}(-,\Si Y)$ to the truncation triangle
	$\tau_{\leqslant0}N\ra N\ra\tau_{\geqslant1}N\ra\Si\tau_{\leqslant0}N$,
	we obtain a morphism $h\colon\tau_{\geqslant1}N\ra\Si Y$ through which
	the connecting morphism $N\ra\Si Y$ factors. Completing $h$ to a
	distinguished triangle
	$Y\xrightarrow{t}Y''\ra\tau_{\geqslant1}N\xrightarrow{h}\Si Y$
	and choosing a morphism $g\colon Y'\ra Y''$ completing the morphism of
	triangles yields the commutative diagram
	\[
	\begin{tikzcd}
		&Y\arrow[r,equal]\arrow[d,"s"]&Y\arrow[d,"t"]\\
		X\arrow[r,"f"]&Y'\arrow[r,"g"]\arrow[d]&Y''\arrow[d]\\
		\tau_{\leqslant0}N\arrow[r]&N\arrow[r]\arrow[d]&\tau_{\geqslant1}N\arrow[d,"h"]\\
		&\Si Y\arrow[r,equal]&\Si Y
	\end{tikzcd}
	\]
	Since $t=gs$ and the cone $\tau_{\geqslant1}N$ of $t$ lies in
	$\pvd_{e}(\bmgamma)$, the fraction $X\xrightarrow{gf}Y''\xleftarrow{t}Y$
	represents the same morphism $f'$ in $\cc$ as
	$X\xrightarrow{f}Y'\xleftarrow{s}Y$. Moreover,
	$\Hom_{\cd(\bmgamma)}(X,\tau_{\geqslant1}N)=0$ because
	$X\in\cd^{\leqslant0}(\bmgamma)$ and
	$\tau_{\geqslant1}N\in\cd^{\geqslant1}(\bmgamma)$, so $gf$ factors
	through $t$ as $gf=tl$ for some morphism $l\colon X\ra Y$. The
	fraction $X\xrightarrow{gf}Y''\xleftarrow{t}Y$ therefore coincides with
	$\pi^{\mathrm{rel}}(l)$, proving that
	$\pi^{\mathrm{rel}}|_{\pr_{\cd}\bmgamma}$ is full.
	
	It remains to show that $\pi^{\mathrm{rel}}|_{\pr_{\cd}\bmgamma}$ is
	dense. Let $Z\in\pr_{\cc}\bmgamma$, and choose an
	$\add\bmgamma$-presentation
	\[
	T_{1}\xrightarrow{f'}T_{0}\ra Z\ra\Si T_{1}
	\]
	in $\cc$. By full faithfulness on $\pr_{\cd}\bmgamma$, the
	idempotents defining $T_{1}$ and $T_{0}$ as direct summands of finite
	direct sums of $\pi^{\mathrm{rel}}\bmgamma$ lift to idempotents on the
	corresponding sums of $\bmgamma$. Since $\per\bmgamma$ is idempotent
	complete, these idempotents split and give lifts
	$P_{1},P_{0}\in\add\bmgamma$ of $T_{1}$ and $T_{0}$. By fullness, $f'$
	lifts to a morphism $f\colon P_{1}\ra P_{0}$, whose cone lies in
	$\pr_{\cd}\bmgamma$ and maps to $Z$ under $\pi^{\mathrm{rel}}$.
\end{proof}

\bigskip
An object $ X $ is in $ \copr_{\cd}\bmgamma $ if and only if $ \Si X $ is in $ \pr_{\cd}\bmgamma $. Therefore we have the following dual statement of Proposition~\ref{Prop:fully faithful of pr}.
\begin{Cor}\label{Cor:fully faithful of copr}
	The quotient functor $ \pi^{rel}\colon\per\bmgamma\ra\cc $ restricts to an equivalence of $ k $-linear categories $ \copr_{\cd}\bmgamma\iso\copr_{\cc}\bmgamma $.
\end{Cor}

Let $ \pr^{F}_{\cd}\bmgamma $ be the following subcategory of $ \pr_{\cd}\bmgamma $
$$ \{\cone(X_{1}\xrightarrow{f} X_{0})\,|\,X_{i}\in \add(\bmgamma)\,\text{and $ \Hom_{\cd}(f, I) $ is surjective for any object $ I\in\cp $}\}. $$
Clearly, we have $ \pr^{F}_{\cd}\bmgamma=\pr_{\cd}\bmgamma\cap\cz=\pr_{\cd}\bmgamma\cap^{\perp}\!(\Si^{>0}\cp)\cap(\Si^{<0}\cp)^{\perp} $, where $ \cp=\add(e\relGammabf) $.

Dually, we define $ \copr_{\cd}^{F}\bmgamma $ as the following subcategory of $ \copr_{\cd}\bmgamma $
$$ \{\Si^{-1}\cone(X_{0}\xrightarrow{f} X_{1})\,|\,X_{i}\in \add(\bmgamma)\,\text{and $ \Hom_{\cd}(P,f) $ is surjective for any object $ P\in\cp $}\}. $$
And we have $ \copr^{F}_{\cd}\bmgamma=\copr_{\cd}\bmgamma\cap\cz $.

Similarly, we define subcategories $$ \pr_{\cc}^{F}\bmgamma=\pr_{\cc}\bmgamma\cap\cy $$ and 
$$ \copr_{\cc}^{F}\bmgamma=\copr_{\cc}\bmgamma\cap\cy $$ of $ \cc $, where $$ \cy=^{\perp}\!(\Si^{>0}\cp)\cap(\Si^{<0}\cp)^{\perp}\subseteq\cc. $$
\begin{Rem}
	We have $\pr_{\cd}\bmgamma\subseteq\Si^{-1}\cw$, where
	$\cw$ is the subcategory defined in Section~\ref{Section: SMC}. Indeed, the defining
	orthogonality conditions for $\cw$ hold for $\Si\bm{\Gamma}$ and
	$\Si^2\bm{\Gamma}$ by Proposition~\ref{Prop:relative 3-CY} and the
	fact that each $S_i$ is concentrated in degree zero. They are
	preserved by extensions and direct summands.
\end{Rem}
\begin{Prop}\label{Prop: equivalence between pr-copr}
	The quotient functor $ \pi^{rel}\colon\per\bmgamma\ra\cc $ induces equivalences of $ k $-linear categories
	$ \pr^{F}_{\cd}\bmgamma\iso\pr_{\cc}^{F}\bmgamma $ and $ \copr^{F}_{\cd}\bmgamma\iso\copr_{\cc}^{F}\bmgamma $. 
	
\end{Prop}
\begin{proof}
	For any $N\in\pvd_e(\bmgamma)$, $P\in\cp$, and $j\in\mathbb Z$, we
	have $\Hom_{\cd}(P,\Si^j N)=0$ since $H^*(N)e=0$, and relative
	$3$-Calabi--Yau duality (Proposition~\ref{Prop:relative 3-CY}) then
	gives $\Hom_{\cd}(N,\Si^j P)=0$ as well. Consequently,
	$\pi^{\mathrm{rel}}$ induces isomorphisms on morphisms both into and
	out of shifts of objects in $\cp$. In particular, an object $X'$ lies
	in $\cz$ if and only if $\pi^{\mathrm{rel}}X'\in\cy$. The two
	asserted equivalences now follow by restricting those of
	Proposition~\ref{Prop:fully faithful of pr} and
	Corollary~\ref{Cor:fully faithful of copr}.
\end{proof}

\begin{Def}\rm\cite{wuRelativeClusterCategories2021}\label{Def: Higgs cat}
	We define the \emph{Higgs category} $ \ch(Q,F,W) $ (or $ \ch $) as the full subcategory of $ \pr^{F}_{\cc}\bm{\Gamma}\cap\copr^{F}_{\cc}\bmgamma $ whose objects are those $ X $ such that $ \Hom_{\cc}(\Si^{-1}\bmgamma,X) $ is finite-dimensional.
\end{Def}

By Proposition~\ref{Prop:silting reduction-fully faithful in cc},
the vertex-deletion morphism
$p\colon\relGammabf(Q,F,W)\longrightarrow\mathbf{\Gamma}(\overline{Q},\overline{W})$
induces a triangle functor
\[
p^{*}\colon\cc(Q,F,W)\longrightarrow\cc(\overline{Q},\overline{W}).
\]

\begin{Def}\rm\label{Defn: categoey D(Q,F,W)}
	We define the category $ \cd(Q,F,W)\subseteq\cc(Q,F,W) $ as the full subcategory of $ \cc(Q,F,W) $ whose objects are those objects $ M $ of $ \cc(Q,F,W) $ such that $ p^{*}(M) $ lies in Plamondon's category $ \cd(\overline{Q},\overline{W}) $ (see Subsection~\ref{Subsection: Plamondon's category}). 
\end{Def}
By Proposition~\ref{Prop: plamondon cat is extriangulated}, we see that $ \cd(Q,F,W) $ is an extension closed subcategory of $ \cc(Q,F,W) $. Hence it has an extriangulated structure in the sense of Nakaoka-Palu~\cite{nakaokaExtriangulatedCategoriesHovey2019},
cf.~also \cite{Palu2023}.
\begin{Lem}\label{Lem: Higgs is idempotent complete}
	The Higgs category $ \ch(Q,F,W) $ is idempotent complete.
\end{Lem}
\begin{proof}
	The Higgs category is a full subcategory of $ \pr^{F}_{\cc}\relGammabf\cap\copr^{F}_{\cc}\bmgamma=\pr_{\cc}\relGammabf\cap\copr_{\cc}\bmgamma\cap\cy $, where $ \cy=^{\perp}\!(\Si^{>0}\cp)\cap(\Si^{<0}\cp)^{\perp}\subseteq\cc $. By definition, the relative cluster category $ \cc $ is idempotent complete. Then $ \cy $ is also idempotent complete. 
	
	By Lemma~\ref{Lemma: pr is extension closed} and Proposition~\ref{Prop:fully faithful of pr}, $ \pr_{\cc}\relGammabf $ is idempotent complete. Hence so is $ \copr_{\cc}\relGammabf $. This shows that $ \pr^{F}_{\cc}\relGammabf\cap\copr^{F}_{\cc}\bmgamma $ is idempotent complete. Since $\Hom_{\cc}(\Si^{-1}\bmgamma,-)$ preserves finite direct sums,
	its finite-dimensionality is preserved by direct summands. Thus the
	Higgs category $\ch(Q,F,W)$ is idempotent complete.
\end{proof}

\subsection{Modules}\label{Subsction: Modules}
Consider the functors $$ R=\Hom_{\cc}(\Si^{-1}\bmgamma,?)\colon\cc\ra\Mod J_{rel} $$ and $$ G=\Hom_{\cc}(?,\Si\bmgamma)\colon(\cc)^{op}\ra\Mod J_{rel}^{op} ,$$
where $ \Mod J_{rel} $ is the category of right $ J_{rel} $-modules. 

\begin{Prop}\cite[Lemma 3.2]{plamondonCategoriesAmasseesAux2011}\label{Prop: modules}
	Let $ X $ and $ Y $ be objects in $ \cc $.
	\begin{itemize}
		\item[1)] If $ X $ lies in $ \copr_{\cc}\bmgamma $, then $ R $ induces an isomorphism
		$$ \Hom_{\cc}(X,Y)/[\bmgamma] \ra\Hom_{J_{rel}}(RX,RY). $$
		\item[2)] If $ Y $ lies in $ \pr_{\cc}\bmgamma $, then $ G $ induces an isomorphism
		$$ \Hom_{\cc}(X,Y)/[\bmgamma] \ra\Hom_{J_{rel}^{op}}(GY,GX). $$
		\item[3)] $ R $ induces an equivalence of categories
		$$ \copr_{\cc}\bmgamma/[\bmgamma]\ra\mathrm{mod} J_{rel} ,$$
		where $ \mathrm{mod} J_{rel} $ denotes the category of finitely presented $ J_{rel} $-modules.
		\item[4)] Let $e=\sum_{i\in F_0}e_i$. A finite-dimensional right
		$J_{rel}$-module $M$ is isomorphic to $RX$ for some
		$X\in\pr_{\cc}\bmgamma\cap\copr_{\cc}\bmgamma$ if and only if
		$Me=0$. Moreover, every short exact sequence
		\[
		0\longrightarrow L\xrightarrow{u}M\xrightarrow{v}N
		\longrightarrow 0
		\]
		of finite-dimensional right $J_{rel}$-modules with $Le=Ne=0$
		lifts to a triangle
		\[
		X_L\xrightarrow{i}X_M\xrightarrow{p}X_N\longrightarrow\Si X_L
		\]
		in $\cc$, with all three terms in
		$\pr_{\cc}\bmgamma\cap\copr_{\cc}\bmgamma$, such that
		\[
		RX_L\xrightarrow{Ri}RX_M\xrightarrow{Rp}RX_N
		\]
		is isomorphic to $L\xrightarrow{u}M\xrightarrow{v}N$.
	\end{itemize}
\end{Prop}
\begin{proof}
	By Proposition~\ref{Prop:fully faithful of pr}, applied to
	$\bmgamma$ and $\Si\bmgamma$ in $\pr_{\cd}\bmgamma$, we have
	$\End_{\cc}(\bmgamma)\cong J_{rel}$ and
	$\Hom_{\cc}(\bmgamma,\Si\bmgamma)=0$.
	Parts (1)--(3) now follow as in
	\cite[Lemma~3.2]{plamondonCategoriesAmasseesAux2011}.
	We prove part (4). Write $\pi=\pi^{rel}$,
	$\Gamma_i=e_i\bmgamma$, and $e=\sum_{f\in F_0}e_f$.
	
	We first show that the support condition is necessary.
	Let $X\in\pr_{\cc}\bmgamma$. By
	Proposition~\ref{Prop:fully faithful of pr}, we may choose
	$\widetilde X\in\pr_{\cd}\bmgamma$ with $\pi\widetilde X\simeq X$.
	For each frozen vertex $f$, the object $\Gamma_f$ is left orthogonal
	to $\pvd_e(\bmgamma)$. Hence the quotient functor induces an
	isomorphism on morphisms out of $\Gamma_f$, and we obtain
	\[
	(RX)e_f
	\cong \Hom_{\cc}(\pi\Gamma_f,\Si X)
	\cong \Hom_{\cd}(\Gamma_f,\Si\widetilde X)
	\cong H^1(\widetilde X)e_f
	=0.
	\]
	The last equality holds because every object of $\pr_{\cd}\bmgamma$
	has cohomology concentrated in nonpositive degrees. Thus $(RX)e=0$.
	
	We next construct a lift of the simple module $S_i$ at each
	unfrozen vertex $i$. In the standard semi-free resolution of $S_i$
	from \cite[Section~4.6]{wuCategorificationIceQuiver2021}, the vertex
	term and the terms corresponding to the original arrows ending at
	$i$ form the dg submodule
	\[
	V_i=\operatorname{Cone}\!\left(
	\bigoplus_{a:j\to i}\Gamma_j
	\xrightarrow{(a)}\Gamma_i
	\right).
	\]
	The quotient of the resolution by $V_i$ is of the form $\Si^2U_i$
	with $U_i\in\pr_{\cd}\bmgamma$. Its two layers correspond to the
	opposite arrows ending at $i$ and the degree $-2$ loop at $i$.
	We therefore have a triangle
	\[
	V_i\longrightarrow S_i\longrightarrow\Si^2U_i
	\longrightarrow\Si V_i .
	\]
	Since $i$ is unfrozen, $S_i$ belongs to $\pvd_e(\bmgamma)$.
	Applying $\pi$ gives $\pi V_i\simeq\Si\pi U_i$. Consequently,
	\[
	X_i:=\Si^{-1}\pi V_i\simeq\pi U_i
	\quad\text{lies in}\quad
	\pr_{\cc}\bmgamma\cap\copr_{\cc}\bmgamma .
	\]
	Moreover, full faithfulness of $\pi$ on $\pr_{\cd}\bmgamma$ gives
	\[
	RX_i
	\cong \Hom_{\cc}(\pi\bmgamma,\pi V_i)
	\cong \Hom_{\cd}(\bmgamma,V_i)
	\cong H^0(V_i)
	\cong S_i .
	\]
	The identification $H^0(V_i)\cong S_i$ follows from the triangle above since
	$\Si^2 U_i\in\cd^{\leqslant-2}(\bmgamma)$.
	
	We now show how to lift a prescribed short exact sequence
	\[
	0\longrightarrow L\xrightarrow{u}M\xrightarrow{v}N
	\longrightarrow 0
	\]
	once its end terms have lifts $X_L$ and $X_N$ in
	$\pr_{\cc}\bmgamma\cap\copr_{\cc}\bmgamma$.
	Fix identifications $RX_L\simeq L$ and $RX_N\simeq N$.
	Put $T=\Si^{-1}\pi\bmgamma$, so that $R=\Hom_{\cc}(T,-)$.
	For $A=L,N$, choose a triangle
	\[
	T_1^A\xrightarrow{d_A}T_0^A\longrightarrow X_A
	\longrightarrow\Si T_1^A,
	\qquad T_0^A,T_1^A\in\add T .
	\]
	Set $P_j^A=RT_j^A$. Since $T$ is rigid, applying $R$ yields
	a projective presentation
	\[
	P_1^A\xrightarrow{Rd_A}P_0^A\xrightarrow{q_A}A
	\longrightarrow 0 .
	\]
	
	By projectivity of $P_0^N$, choose $h:P_0^N\to M$ with $vh=q_N$.
	Then $hRd_N$ takes values in $\ker v=\operatorname{im}u$.
	Since $q_L$ is surjective and $P_1^N$ is projective, there is a
	map $b:P_1^N\to P_0^L$ such that
	\[
	u q_L b=-hRd_N .
	\]
	These choices give a projective presentation of $M$,
	\[
	P_1^L\oplus P_1^N
	\xrightarrow{\left(\begin{smallmatrix}
			Rd_L & b\\ 0 & Rd_N
		\end{smallmatrix}\right)}
	P_0^L\oplus P_0^N
	\xrightarrow{(u q_L,\,h)}
	M\longrightarrow 0 .
	\]
	Indeed, if $u q_L(x)+h(y)=0$, then $q_N(y)=0$, so
	$y=Rd_N(z)$ for some $z\in P_1^N$. It follows that
	$q_L(x-bz)=0$, whence $x-bz\in\operatorname{im}(Rd_L)$.
	This proves exactness at $P_0^L\oplus P_0^N$.
	The inclusion of the first summand and the projection onto the
	second induce $u$ and $v$, respectively, on cokernels.
	
	The restriction of $R$ to $\add T$ is fully faithful. Lift $b$
	to a morphism $\widetilde b:T_1^N\to T_0^L$ and set
	\[
	X_M=\operatorname{Cone}\!\left(
	T_1^L\oplus T_1^N
	\xrightarrow{\left(\begin{smallmatrix}
			d_L & \widetilde b\\ 0 & d_N
		\end{smallmatrix}\right)}
	T_0^L\oplus T_0^N
	\right).
	\]
	The inclusion and projection of the presentations yield, by the
	octahedral axiom, a triangle
	\[
	X_L\xrightarrow{i}X_M\xrightarrow{p}X_N
	\xrightarrow{\delta}\Si X_L .
	\]
	The preceding cokernel computation identifies $RX_M$ with $M$
	and identifies $Ri$ and $Rp$ with $u$ and $v$, respectively.
	
	By construction, $X_M\in\copr_{\cc}\bmgamma$. To see that
	$X_M$ also belongs to $\pr_{\cc}\bmgamma$, observe that
	injectivity of $Ri$ implies $R(\Si^{-1}\delta)=0$.
	Since $X_N\in\pr_{\cc}\bmgamma$, its shift $\Si^{-1}X_N$
	belongs to $\copr_{\cc}\bmgamma$. Part (1) therefore shows that
	$\Si^{-1}\delta$ factors through $\add\bmgamma$.
	We may now apply Lemma~\ref{Lem: X,Y,Z lie in pr} to the triangle
	\[
	\Si^{-1}X_L\longrightarrow\Si^{-1}X_M\longrightarrow
	\Si^{-1}X_N\xrightarrow{\Si^{-1}\delta}X_L .
	\]
	Its first and third objects belong to $\copr_{\cc}\bmgamma$.
	Hence $\Si^{-1}X_M\in\copr_{\cc}\bmgamma$, which is equivalent
	to $X_M\in\pr_{\cc}\bmgamma$.
	
	Finally, every finite-dimensional $J_{rel}$-module annihilated
	by $e$ has a composition series whose factors are simple modules
	at unfrozen vertices. Starting with the lifts $X_i$ and applying
	the extension construction along this series gives a lift of
	every such module. Given a short exact sequence with $Le=Ne=0$,
	we can therefore choose lifts of $L$ and $N$ and apply the same
	construction to that particular sequence. This proves both
	assertions of part (4).
\end{proof}

\begin{Lem}\cite[Lemma 3.4]{plamondonCategoriesAmasseesAux2011}\label{Lem: X,Y,Z lie in pr}
	Let $ X\ra Y\ra Z\xrightarrow{\epsilon}\Si X $ be a triangle in $ \cc $ such that $ \epsilon $ lies in $ [\relGammabf] $. If two of $ X $, $ Y $ and $ Z $ lie in $ \copr_{\cc}\relGammabf $, so does the third one.
\end{Lem}
\begin{proof}
	Apply the proof of~\cite[Lemma~3.4(1)]{plamondonCategoriesAmasseesAux2011}
	to the rigid object $T=\Si^{-1}\bmgamma$. Then
	$\pr_{\cc}T=\copr_{\cc}\bmgamma$ and $[\Si T]=[\bmgamma]$. 
\end{proof}

\begin{Prop}\label{Higgs is extriangulated}\label{Prop: H is extension closed}
	The Higgs category $\ch(Q,F,W)$ is extension-closed in $\cc$.
	Consequently, it inherits an extriangulated structure in the sense of Nakaoka-Palu~\cite{nakaokaExtriangulatedCategoriesHovey2019} whose
	conflations are the first two maps of triangles in $\cc$ with
	all three terms in $\ch(Q,F,W)$.
\end{Prop}

\begin{proof}
	Set $\Gamma=\bmgamma$, let $\pi\colon\per\Gamma\to\cc$ be the
	quotient functor, and let $e$ be the sum of the frozen vertex
	idempotents. For $A=\Gamma$ or $\Gamma^{\mathrm{op}}$, write
	$P^\vee=\operatorname{RHom}_A(P,A)$ for the perfect dual of
	$P\in\per A$, and put $D=\Hom_k(-,k)$. We first note that if $M$ is a
	finite-dimensional $H^0(A)$-module supported on the unfrozen vertices,
	then $M$ is perfect and satisfies
	\[
	M^\vee\simeq\Si^{-3}DM.
	\]
	Indeed, $M$ is a finite iterated extension of simple $H^0(A)$-modules
	at the unfrozen vertices, each of which is perfect. For $A=\Gamma$,
	Proposition~\ref{Prop:relative 3-CY} identifies $H^3(M^\vee)$ with
	$DM$ as an $H^0(\Gamma^{\mathrm{op}})$-module and shows that
	$H^j(M^\vee)=0$ for all $j\ne3$; the corresponding statement for
	$A=\Gamma^{\mathrm{op}}$ then follows by perfect biduality.
	
	Next, we record a truncation argument. Suppose that $P\in\per A$ is
	such that both $P$ and $P^\vee$ have cohomology concentrated in
	degrees at most $1$, and that $M=H^1(P)$ is finite-dimensional and
	supported on the unfrozen vertices. Setting $V=\tau_{\le0}P$, we have
	the truncation triangle and its perfect dual:
	\[
	V\longrightarrow P\longrightarrow\Si^{-1}M\longrightarrow\Si V,
	\qquad
	\Si^{-2}DM\longrightarrow P^\vee\longrightarrow V^\vee
	\longrightarrow\Si^{-1}DM.
	\]
	In particular, $V$ is perfect and $V^\vee$ has cohomology in degrees
	at most $1$. For any $L\in\cd^{\le-2}(A)$, since $A$ is concentrated
	in nonpositive degrees, the complex
	\[
	\operatorname{RHom}_A(V,L)\simeq L\otimes_A^{\mathbf L}V^\vee
	\]
	has cohomology concentrated in degrees at most $-1$. Thus
	$V\in\cd^{\le0}(A)\cap{}^\perp\cd^{\le-2}(A)\cap\per A$, and
	Lemma~\ref{Lemma: pr is extension closed} yields
	$V\in\pr_{\cd(A)}A$. Moreover, whenever
	$P^\vee\in\cd^{\le0}(A^{\mathrm{op}})$, the dual triangle gives
	$H^1(V^\vee)\simeq DM$.
	
	Now let $X\to Y\to Z\xrightarrow{\delta}\Si X$ be a distinguished
	triangle in $\cc$ with $X,Z\in\ch$. For $T\in\{X,Z\}$, choose an
	object $B_T\in\copr_{\cd(\Gamma)}\Gamma$ such that $\pi B_T\simeq T$.
	Since both $\Gamma$ and $\Si B_T$ belong to $\pr_{\cd(\Gamma)}\Gamma$,
	Proposition~\ref{Prop:fully faithful of pr} induces an isomorphism of
	right $H^0(\Gamma)$-modules
	\[
	H^1(B_T)\simeq\Hom_{\cc}(\pi\Gamma,\Si T)=RT.
	\]
	By definition of $\ch$, the module $RT$ is finite-dimensional;
	moreover, $(RT)e=0$ because $\Hom_{\cc}(\Si^{-1}\cp,T)=0$. Since
	$B_T\in\copr_{\cd(\Gamma)}\Gamma$ is an extension of an object in
	$\Si^{-1}\add\Gamma$ by one in $\add\Gamma$, we have
	$B_T\in\cd^{\le1}(\Gamma)$ and $B_T^\vee\in\cd^{\le0}(\Gamma^{\mathrm{op}})$.
	Applying the truncation argument above to $B_Z$ and setting
	$A_Z=\tau_{\le0}B_Z$, we obtain
	\[
	A_Z\in\pr_{\cd(\Gamma)}\Gamma,
	\qquad \pi A_Z\simeq Z,
	\qquad H^1(A_Z^\vee)\simeq D(RZ),
	\]
	where $\pi A_Z\simeq Z$ holds because the cone of $A_Z\to B_Z$ is
	$\Si^{-1}RZ\in\pvd_e(\Gamma)$.
	
	Because both $A_Z$ and $\Si B_X$ lie in $\pr_{\cd(\Gamma)}\Gamma$,
	the full faithfulness of $\pi|_{\pr_{\cd(\Gamma)}\Gamma}$ lifts
	$\delta$ to a morphism $A_Z\to\Si B_X$. Completing this lift to a
	triangle
	\[
	B_X\longrightarrow U\longrightarrow A_Z\longrightarrow\Si B_X
	\]
	in $\per\Gamma$ and applying $\pi$, we see that $\pi U\simeq Y$. From
	this triangle and its perfect dual, both $U$ and $U^\vee$ have
	cohomology in degrees at most $1$; furthermore, since $H^1(A_Z)=0$
	and $H^1(B_X^\vee)=0$, the associated long exact cohomology sequences
	yield surjections
	\[
	RX\twoheadrightarrow H^1(U),
	\qquad
	D(RZ)\twoheadrightarrow H^1(U^\vee).
	\]
	Consequently, both $H^1(U)$ and $H^1(U^\vee)$ are finite-dimensional
	and supported on the unfrozen vertices.
	
	Applying the truncation argument to $U$ shows that
	$\tau_{\le0}U\in\pr_{\cd(\Gamma)}\Gamma$ and
	$\pi(\tau_{\le0}U)\simeq Y$, whence $Y\in\pr_{\cc}\Gamma$. Similarly,
	applying it to $U^\vee$ shows that $V'=\tau_{\le0}(U^\vee)$ lies in
	$\pr_{\cd(\Gamma^{\mathrm{op}})}\Gamma^{\mathrm{op}}$, so
	$(V')^\vee\in\copr_{\cd(\Gamma)}\Gamma$. Since the cone of
	$V'\to U^\vee$ is $\Si^{-1}H^1(U^\vee)$, the duality formula from the
	first paragraph implies that its perfect dual lies in $\pvd_e(\Gamma)$.
	Thus the canonical morphism $U\simeq U^{\vee\vee}\to(V')^\vee$
	becomes an isomorphism under $\pi$, showing that $Y\in\copr_{\cc}\Gamma$.
	
	Finally, the two orthogonal subcategories defining
	$\cy={}^{\perp}(\Si^{>0}\cp)\cap(\Si^{<0}\cp)^\perp$ are closed under
	extensions, so $Y\in\cy$. In addition, exactness of the sequence
	$RX\to RY\to RZ$ at $RY$ gives $\dim_k RY\le\dim_k RX+\dim_k RZ$.
	Together, these verify the remaining conditions for $Y$ to lie in
	$\ch$. That $\ch$ inherits an extriangulated structure now follows
	from \cite[Remark~2.18]{nakaokaExtriangulatedCategoriesHovey2019}
	(see also \cite{Palu2023}).
\end{proof}

\subsection{Mutations induce equivalences}\label{Subsection: Mutations induce equivalences}
Let $ (Q,F,W) $ be an ice quiver with potential. Let $ \bmgamma' $ be the complete relative Ginzburg dg algebra of $ \mu_{v}(Q,F,W)=(Q',F',W') $. For any vertex $ i $ of $ Q $, let $ \bmgamma_{i}=e_{i}\bmgamma $ and $ \bmgamma'_{i}=e_{i}\bmgamma' $. 

\subsubsection{Mutation at unfrozen vertices}\

Let $v$ be an unfrozen vertex of $Q$ lying on no loop or oriented
$2$-cycle. As in Subsection~\ref{Subsection:mutations}, one can mutate
$(Q,F,W)$ at $v$. We assume that $v$ is the source of at least one arrow.

\begin{Thm}
	Assume that $W$ is irredundant, that is, every cyclic term of $W$
	contains an unfrozen arrow. Let $v$ be an unfrozen vertex which is
	the source of at least one arrow and which lies on no loop or
	oriented $2$-cycle. Write
	\[
	(\widetilde Q,F,\widetilde W)=\widetilde\mu_v(Q,F,W),
	\qquad
	(Q',F',W')=\mu_v(Q,F,W)
	\]
	for the premutation at $v$ and its reduction, respectively.
	Denote their relative Ginzburg dg algebras by
	$\widetilde{\bmgamma}$ and $\bmgamma'$, and choose a reduction
	quasi-isomorphism
	\[
	r\colon \widetilde{\bmgamma}\longrightarrow\bmgamma'
	\]
	preserving the vertex idempotents.
	
	\begin{itemize}
		\item[1)]
		There is a triangle equivalence
		\[
		\Phi_+\colon \cd(\bmgamma')\longrightarrow\cd(\bmgamma)
		\]
		such that $\Phi_+(\bmgamma'_i)\simeq\bmgamma_i$ for $i\ne v$
		and $\Phi_+(\bmgamma'_v)\simeq\bmgamma_v^*$, where
		\[
		\bmgamma_v^*
		=
		\operatorname{Cone}\!\left(
		\bmgamma_v
		\xrightarrow{(\alpha)_\alpha}
		\bigoplus_{\substack{\alpha\in Q_1\\s(\alpha)=v}}
		\bmgamma_{t(\alpha)}
		\right).
		\]
		The component indexed by $\alpha$ is left multiplication
		by $\alpha$. The functor $\Phi_+$ restricts to triangle
		equivalences
		\[
		\per\bmgamma'\xrightarrow{\sim}\per\bmgamma,
		\qquad
		\pvd(\bmgamma')\xrightarrow{\sim}\pvd(\bmgamma).
		\]
		
		\item[2)]
		Let
		$\widetilde{\bm G}_{\mathrm{rel}}\colon
		\bm\Pi_2(F)\longrightarrow\widetilde{\bmgamma}$
		be the Ginzburg dg functor of the premutation. Then the diagram
		\[
		\begin{tikzcd}[column sep=large, row sep=large]
			\cd(\bm\Pi_2(F))
			\arrow[r,"{(r\circ\widetilde{\bm G}_{\mathrm{rel}})^*}"]
			\arrow[dr,"{(\bm G_{\mathrm{rel}})^*}"']
			& \cd(\bmgamma') \arrow[d,"\Phi_+"] \\
			& \cd(\bmgamma)
		\end{tikzcd}
		\]
		commutes up to natural isomorphism. Here $(-)^*$ denotes
		derived extension of scalars. The top arrow uses the Ginzburg
		functor of the premutation transported along $r$ and has
		boundary source $\cd(\bm\Pi_2(F))$.
	\end{itemize}
\end{Thm}

\begin{proof}
	Let
	$\widetilde\Phi_+\colon
	\cd(\widetilde{\bmgamma})\longrightarrow\cd(\bmgamma)$
	be the premutation equivalence of
	\cite[Theorem~5.4(a)]{wuCategorificationIceQuiver2021}.
	By \cite[Lemma~4.28]{wuCategorificationIceQuiver2021}, the
	reduction map $r$ is a quasi-isomorphism. Hence restriction of
	scalars $r_*$ is quasi-inverse to $r^*$. Set
	\[
	\Phi_+=\widetilde\Phi_+\circ r_*.
	\]
	Since $r$ preserves the vertex idempotents,
	$r_*(\bmgamma'_i)\simeq e_i\widetilde{\bmgamma}$ for every $i$.
	The description of $\widetilde\Phi_+$ on vertex projectives
	therefore gives the asserted formulas for $\Phi_+$. The
	equivalence $r_*$ identifies the perfect subcategories, since they are generated by the indecomposable projective modules associated with the vertices. It also identifies
	the $\pvd$ subcategories, since restriction of scalars preserves
	the underlying complexes. The corresponding restrictions of
	$\widetilde\Phi_+$ now give~\textup{1)}.
	
	For~\textup{2)}, the unit isomorphism
	$\operatorname{Id}\simeq r_*r^*$ and the premutation diagram of
	\cite[Theorem~5.4(c)]{wuCategorificationIceQuiver2021} give
	natural isomorphisms
	\[
	\Phi_+\circ(r\circ\widetilde{\bm G}_{\mathrm{rel}})^*
	\simeq
	\widetilde\Phi_+\circ r_*r^*
	\circ(\widetilde{\bm G}_{\mathrm{rel}})^*
	\simeq
	\widetilde\Phi_+\circ(\widetilde{\bm G}_{\mathrm{rel}})^*
	\simeq
	(\bm G_{\mathrm{rel}})^*.
	\]
\end{proof}

\begin{Rem}\cite[Remark~5.5]{wuCategorificationIceQuiver2021}
	Let $t_{S_v}$ be the twist at the $3$-spherical simple $S_v$ and put
	$\Phi_- = t_{S_v}\circ\Phi_+$. If $v$ is the target of at least one
	arrow, this is the negative mutation equivalence. It sends
	$\bmgamma'_j$ to $\bmgamma_j$ for $j\ne v$ and sends $\bmgamma'_v$
	to the shifted cone
	\[
	\Si^{-1}\operatorname{Cone}\!\left(
	\bigoplus_{\substack{\beta\in Q_1\\t(\beta)=v}}\bmgamma_{s(\beta)}
	\xrightarrow{(\beta)_\beta}\bmgamma_v\right).
	\]
	The component indexed by $\beta$ is left multiplication by $\beta$.
	For each $X\in\cd(\bmgamma)$ there is a functorial triangle
	\[
	\RHom(S_v,X)\ten_k S_v\longrightarrow X\longrightarrow t_{S_v}(X)
	\longrightarrow\Si\bigl(\RHom(S_v,X)\ten_k S_v\bigr).
	\]
\end{Rem}

\bigskip


\begin{Prop}\label{Prop: mutation at unfrozen}
	Under the assumptions of the previous theorem, let $e'$ be the
	sum of the frozen vertex idempotents of $\bmgamma'$. The functors
	\[
	\begin{tikzcd}
		\Phi_{\pm}\colon\per\bmgamma'\arrow[r,shift left=0.9,"\sim"]\arrow[r,shift right=0.9,"\sim",swap]&\per\bmgamma
	\end{tikzcd}
	\]
	induce equivalences $ \Phi_{\pm}\colon \pvd_{e'}(\bmgamma')\iso\pvd_{e}(\bmgamma) $ and 
	\[
	\begin{tikzcd}
		\Phi_{\pm}\colon(\cc)'\coloneqq\cc(Q',F',W')\arrow[r,shift left=0.9,"\sim"]\arrow[r,shift right=0.9,"\sim",swap]&\cc.
	\end{tikzcd}
	\]
	The induced functors on relative cluster categories are naturally
	isomorphic. Write $\Phi$ for their common value up to isomorphism.
	Then $\Phi$ restricts to an equivalence
	$\ch'=\ch(Q',F',W')\xrightarrow{\sim}\ch(Q,F,W)$, and the following
	diagram commutes up to natural isomorphism
	\[
	\begin{tikzcd}
		\per\bmgamma'\arrow[r]\arrow[d,shift left=-2,"\Phi_{+}",swap]\arrow[d,shift right=-1,"\Phi_{-}"]&(\cc)'\arrow[d,"\Phi"]&\ch'\arrow[l,symbol=\subset]\arrow[d,"\Phi"]\\
		\per\bmgamma\arrow[r]&\cc&\ch\arrow[l,symbol=\subset]
	\end{tikzcd}
	\]
\end{Prop}
\begin{proof}
	The equivalence $\Phi_+$ preserves $\pvd$ and sends each frozen
	projective $\bmgamma'_f$ to $\bmgamma_f$. A dg module $N$ belongs to
	$\pvd_e(\bmgamma)$ precisely when it belongs to $\pvd(\bmgamma)$
	and $\Hom_{\cd(\bmgamma)}(\bmgamma_f,\Si^mN)=0$ for every frozen
	$f$ and every $m\in\mathbb Z$. The same criterion applies on the
	primed side. Thus $\Phi_+$ identifies $\pvd_{e'}(\bmgamma')$ with
	$\pvd_e(\bmgamma)$ and induces equivalences on the Verdier quotients
	and their idempotent completions.
	
	For any $X\in\per\bmgamma$, the complex $\RHom(X,S_v)$ has finite-dimensional total cohomology, since this holds for $X=\bmgamma$ and is preserved under shifts, cones, and direct summands. Proposition~\ref{Prop:relative 3-CY} gives the
	same finiteness for $\RHom(S_v,X)$. Thus the cone of the natural
	map $X\to t_{S_v}(X)$ lies in
	$\thick(S_v)\subseteq\pvd_e(\bmgamma)$. Since the twist takes
	$S_v$ to $\Si^{-2}S_v$, it preserves $\thick(S_v)$ in both
	directions. 
	
	The triangle therefore shows that it restricts to
	autoequivalences of $\per\bmgamma$ and $\pvd_e(\bmgamma)$.
	Its natural map from the identity becomes an isomorphism on the Verdier
	quotient, hence also on its idempotent completion. This gives
	$\Phi_+\simeq\Phi_-$ naturally on the relative cluster categories.
	
	Write $T=\bmgamma$ and $T^*=\Phi(\bmgamma')$ in $\cc$, and let
	$T_v^*=\Phi(\bmgamma'_v)$. The two exchange triangles are
	\[
	T_v\longrightarrow E\longrightarrow T_v^*\longrightarrow\Si T_v,
	\qquad
	T_v^*\longrightarrow E'\longrightarrow T_v\longrightarrow\Si T_v^*,
	\]
	where $E$ is the sum over arrows out of $v$ and $E'$ the sum
	over arrows into $v$. The first triangle is the defining cone.
	For the second, the two layers of the resolution of $S_v$ used in
	Proposition~\ref{Prop: modules}\,\textup{4)} are
	$\operatorname{Cone}(E'\to T_v)$ and
	$\Si^2\operatorname{Cone}(T_v\to E)$ before passage to $\cc$.
	Since $S_v$ becomes zero, they give
	$\operatorname{Cone}(E'\to T_v)\simeq\Si T_v^*$ in $\cc$.

	Both $\add T$ and $\add T^*$ are rigid Krull--Schmidt categories,
	by Proposition~\ref{Prop:fully faithful of pr} on the respective
	sides and the semiperfectness of the complete relative Jacobian
	algebras. Applying $\Hom_{\cc}(T_v,-)$ to the second triangle
	identifies $\Hom_{\cc}(T_v,\Si T_v^*)$ with the cokernel of
	$\Hom_{\cc}(T_v,E')\to\End_{\cc}(T_v)$. The image of this
	map is the radical, since in the
	complete path algebra one can group all nonconstant paths ending
	at $v$ by their last arrow. Dually, the first triangle and outgoing arrows give
	$\Hom_{\cc}(T_v^*,\Si T_v)\cong k$. In both cokernel
	computations the connecting morphism is the image of $1_{T_v}$.
	Its class modulo the radical is nonzero, so both triangles are
	non-split. By
	\cite[Proposition~2.7]{plamondonCategoriesAmasseesAux2011},
	$\pr_{\cc}T=\pr_{\cc}T^*$, and shifting yields
	$\copr_{\cc}T=\copr_{\cc}T^*$.
	
	Since $\Phi$ identifies the frozen projectives, it takes
	$\cy'={}^{\perp}(\Si^{>0}\cp')\cap(\Si^{<0}\cp')^{\perp}$ to
	$\cy$. It remains to check the finite-dimensionality in
	Definition~\ref{Def: Higgs cat}. Let $X\in\pr_{\cc}T$ and choose a
	presentation $T_1\to T_0\to X\to\Si T_1$. 
	By rigidity, one obtains a
	surjection $\Hom_{\cc}(T_v,T_0)\to\Hom_{\cc}(T_v,X)$. Hence the
	cokernel of $\Hom_{\cc}(E,X)\to\Hom_{\cc}(T_v,X)$ is a quotient
	of the corresponding cokernel for $T_0$, which embeds in
	$\Hom_{\cc}(T_v^*,\Si T_0)$. This last space is finite-dimensional
	since the common summands are rigid with $T_v^*$ and
	$\Hom_{\cc}(T_v^*,\Si T_v)\cong k$. Applying $\Hom_{\cc}(-,\Si X)$
	to the first exchange triangle now shows that
	$\Hom_{\cc}(T_v^*,\Si X)$ is finite-dimensional whenever
	$\Hom_{\cc}(T,\Si X)$ is. The second exchange triangle gives the
	converse using a $T^*$-presentation. Thus $\Phi$ restricts to
	the asserted equivalence $\ch'\xrightarrow{\sim}\ch$.
	
\end{proof}

\subsubsection{Mutation at frozen vertices}\ 

Let $W$ be irredundant, and let $v$ be a frozen source which
lies on no loop or oriented $2$-cycle of $Q$. Write
\[
(\widetilde Q,\widetilde F,\widetilde W)
=\widetilde\mu_v(Q,F,W),
\qquad
(Q',F',W')=\mu_v(Q,F,W)
\]
for the premutation at $v$ and its reduction, respectively.
Denote the relative Ginzburg dg algebras of the original,
premutated and reduced ice quivers with potential by
$\bmgamma$, $\widetilde{\bmgamma}$ and $\bmgamma'$, respectively.
Choose a reduction quasi-isomorphism
\[
r\colon\widetilde{\bmgamma}\longrightarrow\bmgamma'
\]
preserving the vertex idempotents, as in
\cite[Lemma~4.28]{wuCategorificationIceQuiver2021}. Let
\[
G\colon\bm\Pi_2(F)\longrightarrow\bmgamma,
\qquad
\widetilde G\colon
\bm\Pi_2(\widetilde F)\longrightarrow\widetilde{\bmgamma}
\]
be the Ginzburg dg functors of the original ice quiver with
potential and its premutation.

\begin{Thm}\label{Thm: mutation at frozen source}
	There is a triangle equivalence
	\[
	\Psi_+\colon\cd(\bmgamma')\longrightarrow\cd(\bmgamma).
	\]
	For $i\ne v$, we have
	$\Psi_+(\bmgamma'_i)\simeq\bmgamma_i$. At $v$, we have
	\[
	\Psi_+(\bmgamma'_v)\simeq
	\cone\!\left(
	\bmgamma_v
	\xrightarrow{(\alpha)_\alpha}
	\bigoplus_{\substack{\alpha\in F_1\\s(\alpha)=v}}
	\bmgamma_{t(\alpha)}
	\right),
	\]
	where the component indexed by $\alpha$ is left multiplication
	by $\alpha$. The functor $\Psi_+$ restricts to triangle
	equivalences
	\[
	\per(\bmgamma')\xrightarrow{\sim}\per(\bmgamma),
	\qquad
	\pvd(\bmgamma')\xrightarrow{\sim}\pvd(\bmgamma).
	\]
	
	Moreover, the following square commutes up to natural isomorphism:
	\begin{equation}\label{Diagram:left ice mutation}
		\begin{tikzcd}[column sep=large, row sep=large]
			\cd(\bm\Pi_2(\widetilde F))
			\arrow[r,"{(r\circ\widetilde G)^*}"]
			\arrow[d,"{t^{-1}_{S_v}\circ\mathrm{can}_v}"']
			& \cd(\bmgamma') \arrow[d,"\Psi_+"] \\
			\cd(\bm\Pi_2(F))
			\arrow[r,"{G^*}"']
			& \cd(\bmgamma).
		\end{tikzcd}
	\end{equation}
	Here $(-)^*$ denotes derived extension of scalars. The top
	horizontal functor is induced by the transported boundary map
	$r\circ\widetilde G$ from the premutated frozen quiver.
	The equivalence $\mathrm{can}_v$ is induced by the standard
	identification of $\bm\Pi_2(\widetilde F)$ with $\bm\Pi_2(F)$
	associated with reversing the frozen arrows at $v$.
	The functor $t^{-1}_{S_v}$ is the inverse spherical twist at
	the simple $\bm\Pi_2(F)$-module $S_v$.
	
	More explicitly, put $A=\bm\Pi_2(F)$ and let $k_v$ denote
	$k$ as an $A$-bimodule via the augmentation
	$\varepsilon_v\colon A\to k$ at $v$. For every $X\in\cd(A)$
	there is a functorial triangle
	\[
	t^{-1}_{S_v}(X)\longrightarrow X
	\xrightarrow{\,X\otimes_A^{\mathbf L}\varepsilon_v\,}
	X\otimes_A^{\mathbf L}k_v
	\longrightarrow \Si t^{-1}_{S_v}(X).
	\]
	For $X\in\per(A)$, its third term identifies naturally with
	$\Hom_k(\RHom_A(X,S_v),S_v)$.
\end{Thm}

\begin{proof}
	Let $\widetilde\Psi_+\colon
	\cd(\widetilde{\bmgamma})\to\cd(\bmgamma)$ be the premutation
	equivalence of \cite[Theorem~6.8]{wuCategorificationIceQuiver2021}.
	Since $r$ is a quasi-isomorphism, restriction of scalars $r_*$
	is quasi-inverse to $r^*$. Set
	\[
	\Psi_+=\widetilde\Psi_+\circ r_*.
	\]
	The map $r$ induces a quasi-isomorphism
	$e_i\widetilde{\bmgamma}\to r_*(\bmgamma'_i)$ for each vertex
	$i$. The formulas for the vertex projectives therefore follow
	from the corresponding formulas for $\widetilde\Psi_+$.
	They also show that $r_*$ identifies the perfect subcategories.
	Since $r_*$ preserves the underlying complexes, it identifies
	the $\pvd$ subcategories as well. The asserted restrictions
	now follow from those of $\widetilde\Psi_+$.
	
	For the square, the unit isomorphism
	$\operatorname{Id}\simeq r_*r^*$ and the premutation square of
	\cite[Theorem~6.8]{wuCategorificationIceQuiver2021} give
	\[
	\Psi_+\circ(r\circ\widetilde G)^*
	\simeq
	\widetilde\Psi_+\circ\widetilde G^*
	\simeq
	G^*\circ t^{-1}_{S_v}\circ\mathrm{can}_v.
	\]
	Finally, the bimodule representing the inverse twist is
	$\Si^{-1}\cone(A\xrightarrow{\varepsilon_v}k_v)$.
	Applying $X\otimes_A^{\mathbf L}-$ to its defining triangle
	gives the stated triangle. For perfect $X$, the identification
	of the third term follows from tensor--Hom duality.
\end{proof}

\begin{Prop}\label{Rem: mutation at frozen}
	Keep the hypotheses and notation of
	Theorem~\ref{Thm: mutation at frozen source}. Let $e$ and $e'$
	be the sums of the frozen vertex idempotents on the two sides.
	The equivalence $\Psi_+$ restricts to a triangle equivalence
	\[
	\pvd_{e'}(\bmgamma')\xrightarrow{\sim}\pvd_e(\bmgamma).
	\]
	It induces a triangle equivalence $\Psi_+^{\mathrm{rel}}$ making
	the following diagram commute up to natural isomorphism.
	\[
	\begin{tikzcd}[column sep=large, row sep=large]
		\per\bm\Pi_2(\widetilde F)
		\arrow[r,"{(r\circ\widetilde G)^*}"]
		\arrow[d,"{t_{S_v}^{-1}\circ\mathrm{can}_v}"']
		& \per\bmgamma' \arrow[r,"\pi'"] \arrow[d,"\Psi_+"]
		& \cc(Q',F',W') \arrow[d,"\Psi_+^{\mathrm{rel}}"] \\
		\per\bm\Pi_2(F) \arrow[r,"G^*"']
		& \per\bmgamma \arrow[r,"\pi"']
		& \cc(Q,F,W).
	\end{tikzcd}
	\]
	We have $\Psi_+^{\mathrm{rel}}(\pi'\bmgamma'_v)\notin\ch$,
	although $\pi'\bmgamma'_v\in\ch'$. Moreover,
	$\Psi_+^{\mathrm{rel}}$ restricts to a $k$-linear equivalence
	\[
	\cd(Q',F',W')\xrightarrow{\sim}\cd(Q,F,W).
	\]
\end{Prop}

\begin{proof}
	Put $\mathcal T_F=\thick(e\bmgamma)$ in $\per\bmgamma$, and
	define $\mathcal T_{F'}$ similarly.
	Theorem~\ref{Thm: mutation at frozen source} gives a triangle
	\[
	\bmgamma_v\xrightarrow{a} E_v
	\longrightarrow \Psi_+(\bmgamma'_v)
	\longrightarrow \Si\bmgamma_v,
	\qquad
	E_v=\bigoplus_{\substack{\alpha\in F_1\\s(\alpha)=v}}
	\bmgamma_{t(\alpha)} .
	\]
	Every summand of $E_v$ is a frozen projective at a vertex different
	from $v$. Since $\Psi_+$ fixes the other frozen projectives up to
	isomorphism, the triangle gives
	$\Psi_+(\mathcal T_{F'})=\mathcal T_F$. For $N\in\pvd(\bmgamma)$
	and $f\in F_0$, we have
	\[
	\Hom_{\cd(\bmgamma)}(\bmgamma_f,\Si^n N)=H^n(N)e_f.
	\]
	Thus $\pvd_e(\bmgamma)=\pvd(\bmgamma)\cap\mathcal T_F^\perp$,
	and similarly on the primed side. Since $\Psi_+$ identifies the
	$\pvd$ subcategories, it identifies these orthogonals inside them.
	Passing to the Verdier quotients and then to their idempotent
	completions gives $\Psi_+^{\mathrm{rel}}$ and the right square.
	The left square is the restriction of
	\eqref{Diagram:left ice mutation} to perfect objects.
	
	For the assertion about $\ch$, write $\cc=\cc(Q,F,W)$.
	Since $e\bmgamma$ is left orthogonal to $\pvd_e(\bmgamma)$,
	the quotient functor $\pi$ is fully faithful on
	$\add(e\bmgamma)$. Let $L_v$ be the simple right
	$\bmgamma$-module at $v$. We have
	\[
	\Hom_{\cd(\bmgamma)}(E_v,L_v)=0,
	\qquad
	\Hom_{\cd(\bmgamma)}(\bmgamma_v,L_v)=k.
	\]
	Hence $a$ has no left inverse, and neither does $\pi(a)$.
	The connecting morphism
	\[
	\Psi_+^{\mathrm{rel}}(\pi'\bmgamma'_v)
	\longrightarrow \Si\pi\bmgamma_v
	\]
	is therefore nonzero. As $\pi\bmgamma_v\in\cp$ and
	$\ch\subseteq{}^\perp(\Si\cp)$, the source of this morphism lies
	outside $\ch$. On the primed side, $\pi'\bmgamma'_v\in\cp'\subseteq\ch'$.
	
	To compare deletion on the two sides, write
	$p\colon\bmgamma\to B=\bar\Gamma$ and
	$p'\colon\bmgamma'\to B'=\bar\Gamma'$ for the deletion maps.
	Let $\cd_0(B)=\operatorname{Loc}_{\cd(B)}(\pvd(B))$, where
	$\operatorname{Loc}$ denotes closure under triangles, direct
	summands and arbitrary coproducts.
	
	For every unfrozen vertex $i$, restriction takes $\bar S_i$ to
	$S_i$, and the counit $p^*p_*\bar S_i\to\bar S_i$ is an
	isomorphism by the vertex-simple resolution in the proof of
	Proposition~\ref{Prop:silting reduction-fully faithful in cc}.
	The $\bar S_i$ thickly generate $\pvd(B)$, and the $S_i$ thickly
	generate $\pvd_e(\bmgamma)$. Since $p^*$ and $p_*$ both preserve
	coproducts, the counit is an isomorphism on $\cd_0(B)$. Hence
	restriction induces equivalences
	\[
	\pvd(B)\simeq\pvd_e(\bmgamma),
	\qquad
	\cd_0(B)\simeq
	\operatorname{Loc}_{\cd(\bmgamma)}(\pvd_e(\bmgamma)).
	\]
	The same holds on the primed side. Consequently $\Psi_+$ induces
	an equivalence $K_0\colon\cd_0(B')\to\cd_0(B)$ satisfying
	$p_*K_0\simeq\Psi_+p'_*$ and $K_0(\pvd(B'))=\pvd(B)$
	up to isomorphism.
	
	We use the Nakayama embedding
	$\nu_B\colon\per B\hookrightarrow\cd_0(B)$ of
	\cite[Section~2.20 and Proposition~2.21(a)]{kellerYangDerived equivalences}.
	For $P\in\per B$ and $L\in\cd_0(B)$ it satisfies
	\[
	\Hom(L,\nu_B P)\simeq D\Hom(P,L),
	\qquad D=\Hom_k(-,k).
	\]
	Its image consists of the objects $Y$ such that
	$\bigoplus_n\Hom(M,\Si^nY)$ is finite-dimensional for every
	$M\in\pvd(B)$. Hence $K_0$ identifies the images of $\nu_{B'}$
	and $\nu_B$, and induces a triangle equivalence
	$K\colon\per B'\to\per B$. For $X\in\per\bmgamma'$ and
	$L\in\cd_0(B)$, adjunction gives natural isomorphisms
	\[
	\begin{aligned}
		\Hom(L,K_0\nu_{B'}p'^*X)
		&\simeq D\Hom(p'^*X,K_0^{-1}L)\\
		&\simeq D\Hom(\Psi_+X,p_*L)\\
		&\simeq \Hom(L,\nu_B p^*\Psi_+X).
	\end{aligned}
	\]
	By Yoneda and full faithfulness of $\nu_B$, this gives
	$Kp'^*\simeq p^*\Psi_+$ on perfect objects.
	
	Since $p'^*$ identifies $\pvd_{e'}(\bmgamma')$ with $\pvd(B')$,
	and $p^*$ identifies $\pvd_e(\bmgamma)$ with $\pvd(B)$, this
	isomorphism gives $K(\pvd(B'))=\pvd(B)$ up to isomorphism.
	Thus $K$ descends to an equivalence $K_{\cc}$ of the absolute
	cluster categories satisfying
	\[
	K_{\cc}\,p'^*\simeq p^*\Psi_+^{\mathrm{rel}}.
	\]
	For every unfrozen $i$, evaluating the comparison on $\bmgamma'_i$
	gives $K(e_iB')\simeq e_iB$. Hence $K_{\cc}$ identifies
	the initial objects of the two absolute cluster categories.
	It therefore preserves the presentation and copresentation
	conditions defining Plamondon's subcategory, as well as the
	finite-dimensionality of $\Hom(B,\Si-)$, in both directions.
	Taking inverse images under $p'^*$ and $p^*$ gives the asserted
	equivalence of the relative $\cd$ subcategories.
\end{proof}

\bigskip
\begin{Rem}
	Assume that $W$ is irredundant and that $v$ is a frozen sink
	lying on no loop or oriented $2$-cycle of $Q$. Write
	\[
	(\widetilde Q,\widetilde F,\widetilde W)
	=\widetilde\mu_v(Q,F,W),
	\qquad
	(Q',F',W')=\mu_v(Q,F,W)
	\]
	for the premutation and its reduction. As in
	Theorem~\ref{Thm: mutation at frozen source}, choose a reduction
	quasi-isomorphism $r\colon\widetilde{\bmgamma}\to\bmgamma'$
	preserving the vertex idempotents. Let $\widetilde G$ be the
	Ginzburg dg functor of the premutation and let
	$\widetilde\Psi_-$ be the equivalence of
	\cite[Theorem~6.9]{wuCategorificationIceQuiver2021}.
	Put $\Psi_-=\widetilde\Psi_-\circ r_*$, where $r_*$ is
	restriction of scalars.
	
 	For $i\ne v$ we have $\Psi_-(\bmgamma'_i)\simeq\bmgamma_i$.
	At $v$ there is a triangle
	\[
	\Psi_-(\bmgamma'_v)\longrightarrow E_v
	\xrightarrow{b}\bmgamma_v
	\longrightarrow\Si\Psi_-(\bmgamma'_v),
	\qquad
	E_v=\bigoplus_{\substack{\alpha\in F_1\\t(\alpha)=v}}
	\bmgamma_{s(\alpha)},
	\]
	where the component of $b$ indexed by $\alpha$ is left
	multiplication by $\alpha$. All summands of $E_v$ are frozen
	projectives at vertices different from $v$. The triangle
	therefore gives
	$\Psi_-(\thick(e'\bmgamma'))=\thick(e\bmgamma)$.
	As in Proposition~\ref{Rem: mutation at frozen}, $\Psi_-$
	identifies $\pvd_{e'}(\bmgamma')$ with $\pvd_e(\bmgamma)$
	and induces a triangle equivalence $\Psi_-^{\mathrm{rel}}$
	on the relative cluster categories. The following diagram
	commutes up to natural isomorphism.
	\[
	\begin{tikzcd}[column sep=large, row sep=large]
		\per\bm\Pi_2(\widetilde F)
		\arrow[r,"{(r\circ\widetilde G)^*}"]
		\arrow[d,"{t_{S_v}\circ\mathrm{can}_v}"']
		& \per\bmgamma' \arrow[r,"\pi'"] \arrow[d,"\Psi_-"]
		& \cc(Q',F',W') \arrow[d,"\Psi_-^{\mathrm{rel}}"] \\
		\per\bm\Pi_2(F) \arrow[r,"G^*"']
		& \per\bmgamma \arrow[r,"\pi"']
		& \cc(Q,F,W).
	\end{tikzcd}
	\]
	Here $\mathrm{can}_v$ is induced by reversing the frozen arrows
	at $v$, and $t_{S_v}$ is the spherical twist at the simple
	$\bm\Pi_2(F)$-module $S_v$. The left square follows from the
	premutation square and $r_*r^*\simeq\operatorname{Id}$.
	
	To see the effect on $\ch'$, let $L_v$ be the simple right
	$\bmgamma$-module at $v$. The map induced by $b$ is
	\[
	\Hom_{\cd(\bmgamma)}(\bmgamma_v,L_v)
	\longrightarrow \Hom_{\cd(\bmgamma)}(E_v,L_v),
	\qquad k\longrightarrow 0.
	\]
	Thus $b$ has no right inverse. Since $\pi$ is fully faithful
	on $\add(e\bmgamma)$, $\pi(b)$ has no right inverse either.
	The connecting morphism
	\[
	\pi\bmgamma_v\longrightarrow
	\Si\Psi_-^{\mathrm{rel}}(\pi'\bmgamma'_v)
	\]
	is consequently nonzero. As $\pi\bmgamma_v\in\cp$ and
	$\ch\subseteq(\Si^{-1}\cp)^\perp$, we obtain
	$\Psi_-^{\mathrm{rel}}(\pi'\bmgamma'_v)\notin\ch$,
	although $\pi'\bmgamma'_v\in\cp'\subseteq\ch'$.
	
	Finally, $\Psi_-$ identifies the corresponding $\pvd_e$-subcategories and sends the indecomposable projective module associated with each unfrozen vertex to its counterpart.
	The comparison with vertex deletion in the proof of
	Proposition~\ref{Rem: mutation at frozen} therefore applies
	to $\Psi_-$ and yields a $k$-linear equivalence
	\[
	\Psi_-^{\mathrm{rel}}\colon
	\cd(Q',F',W')\xrightarrow{\sim}\cd(Q,F,W).
	\]
\end{Rem}

\subsection{Iyama--Kalck--Wemyss--Yang's theorem without the Noetherian hypothesis}
\label{ss:IKW-theorem}
In this subsection, we establish a version of the
recognition theorem of Iyama--Kalck--Wemyss--Yang
\cite[Theorem~2.7]{Kalck-Iyama-Wemyss-Yang2014} without the Noetherian hypothesis.

Let $ \ce $ be a Frobenius category and $ \cp\subseteq\ce $ the subcategory of projective-injective objects. For each $ P\in\cp $, we put $ P^{\we}=\cp(?,P)\colon\cp^{op}\ra\Mod k $. We denote by $ \Mod\cp $ the category of right $ \cp $-modules. Then we have a functor
$$ H\colon\ce\ra\Mod\cp $$ which takes $ X $ to $ \ce(?,X)|_{\cp} $. Recall that a module is {\em pseudo-coherent}
if it admits a (possibly infinite) resolution by finitely generated projective modules.
\begin{Lem}\label{Lemma: P is fp Gro}
	\begin{itemize}
		\item[1)] For any object $ X $ in $ \ce $, the right $ \cp $-module $ H(X) $ is pseudo-coherent.
		\item[2)] For any objects $ X\in\ce $ and $ P\in\cp $, we have  $ \Ext_{\cp}^{i}(H(X),P^{\we})=0 $, $ \forall i>0 $.
	\end{itemize}
\end{Lem}
\begin{proof}
	1) Choose conflations
	$$ \Omega^{i} X\rightarrowtail P_{i-1}\twoheadrightarrow \Omega^{i-1} X\quad \text{for $ i\geqslant1 $ with proj-inj $ P_{i-1} $}.$$
	Their images under $ \Hom_{\ce}(P,?) $ are exact, for all $P$ in $\cp$. So the following exact sequence
	$$ \cdots\ra H(P_{i})\ra\cdots\ra H(P_{1})\ra H(P_{0})\ra H(X)\ra0 $$ is a resolution for $ H(X)\in\Mod\cp $.
	
	2) Choose a resolution $ \ldots\ra P_{2}\ra P_{1}\ra P_{0}\ra X\ra0 $ as in $ 1) $. Its image under $ H $ is a resolution of $ H(X) $ in $ \Mod\cp $. Since $ P $ is also injective, the homologies of
	\[
	\begin{tikzcd}
		\Hom(HP_{0},HP)\arrow[r]\arrow[d,"\simeq"]&\Hom(HP_{1},HP)\arrow[r]\arrow[d,"\simeq"]&\Hom(HP_{2},HP)\arrow[r]\arrow[d,"\simeq"]&\cdots\\
		\Hom(P_{0},P)\arrow[r]&\Hom(P_{1},P)\arrow[r]&\Hom(P_{2},P)\arrow[r]&\cdots
	\end{tikzcd}
	\]
	vanish in all degrees $ i>0 $.
\end{proof}	

\bigskip
\bigskip
For an additive category $\mathcal A$, let
$\mathrm{fp}_{\infty}(\mathcal A)$ denote its pseudo-coherent right
modules. Write $\operatorname{mod}\mathcal A$ for the category of finitely
presented right $\mathcal A$-modules. We say $\mathcal A$ is
\emph{right coherent} if $\operatorname{mod}\mathcal A$ is closed
under kernels in $\operatorname{Mod}\mathcal A$. Equivalently,
every morphism in $\mathcal A$ admits a weak kernel. In this case,
successive syzygies of finitely presented modules are finitely
presented, and hence
$\operatorname{mod}\mathcal A=\mathrm{fp}_{\infty}(\mathcal A)$.

Let $\cm$ be a full additive subcategory of $\ce$, closed under direct
summands and containing $\cp$. Write $j\colon\cp\hookrightarrow\cm$
for the inclusion and $\widetilde L=j_!L=L\otimes_{\cp}\cm$ for
induction of right modules. We say $L$ is \emph{$\cm$-pseudo-coherent} if
$L\in\mathrm{fp}_{\infty}(\cp)$ and $\widetilde L$ has a bounded
resolution by finitely generated projective $\cm$-modules. We put
\[
\mathrm{gpr}_{\infty}(\cp)=
\bigl\{L\in\mathrm{fp}_{\infty}(\cp)\mid
\operatorname{Ext}^{q}_{\cp}(L,P^{\wedge})=0
\text{ for all }q>0\text{ and }P\in\cp\bigr\}.
\]
Let $\mathrm{gpr}_{\infty}(\cp,\cm)$ be its full subcategory of
$\cm$-pseudo-coherent modules. We say that $\cp$ is
\emph{$\mathrm{fp}_{\infty}$-Gorenstein of dimension at most $n$
	with respect to $\cm$} if
$\operatorname{Ext}^{q}_{\cp}(L,P^{\wedge})=0$ for every
$\cm$-pseudo-coherent $L$, every $P\in\cp$ and every $q>n$.

\begin{Thm}\label{Thm: structure thm without Noetherian}
	Let $\ce$ be an idempotent complete Frobenius exact category with
	$\cp=\operatorname{proj}\ce$, and let $\cm\subseteq\ce$ be as above.
	Assume that every $V\in\mathrm{fp}_{\infty}(\cm)$ has projective
	dimension at most $n$. Suppose that $\cm$ is contravariantly finite
	in $\ce$ or that $\cm$ is right coherent.
	Then $\cp$ is $\mathrm{fp}_{\infty}$-Gorenstein of dimension at most
	$n$ with respect to $\cm$, and restricted Yoneda induces an exact
	equivalence
	\[
	H\colon\ce\xrightarrow{\sim}\mathrm{gpr}_{\infty}(\cp,\cm),
	\qquad H(X)=\ce(-,X)|_{\cp}.
	\]
	Here the target carries the exact structure inherited from
	$\operatorname{Mod}\cp$. If $\cm$ is right coherent, then
	$\mathrm{gpr}_{\infty}(\cp,\cm)=\mathrm{gpr}_{\infty}(\cp)$.
\end{Thm}

\begin{proof}
	Write $F(Y)=\ce(-,Y)|_{\cm}$ and $M^{\wedge}_{\cm}=\cm(-,M)$.
	Induction is right exact and sends $P^{\wedge}$ to
	$P^{\wedge}_{\cm}$. Since $j$ is fully faithful, the unit
	$L\to j^*j_!L$ is an isomorphism. Both assertions can be checked
	on projective presentations.
	
	First we show that $j_!H(X)$ is perfect for every $X\in\ce$.
	Suppose that $\cm$ is contravariantly finite. For each $Y\in\ce$
	choose a right $\cm$-approximation $M\to Y$. Adding a projective
	deflation $P\twoheadrightarrow Y$ gives a deflation
	$M\oplus P\twoheadrightarrow Y$ which is still a right approximation.
	Writing $Y'$ for its kernel gives an exact sequence
	\[
	0\longrightarrow F(Y')\longrightarrow(M\oplus P)^{\wedge}_{\cm}
	\longrightarrow F(Y)\longrightarrow0.
	\]
	Iteration shows that $F(Y)\in\mathrm{fp}_{\infty}(\cm)$.

	Now choose conflations $\Omega X\rightarrowtail P_0\twoheadrightarrow X$
	and $\Omega^2X\rightarrowtail P_1\twoheadrightarrow\Omega X$ with
	$P_0,P_1\in\cp$. The resulting projective presentation of $H(X)$,
	right exactness of $j_!$, and left exactness of $F$ give an exact sequence
	\begin{equation}\label{eq:induced-Higgs-recognition}
		0\longrightarrow F(\Omega^2X)\longrightarrow(P_1)^{\wedge}_{\cm}
		\longrightarrow(P_0)^{\wedge}_{\cm}
		\longrightarrow j_!H(X)\longrightarrow0 .
	\end{equation}
	Thus $j_!H(X)$ is pseudo-coherent. If instead $\cm$ is right coherent,
	$j_!$ preserves finite presentations, and every finitely presented right
	$\cm$-module is pseudo-coherent.
	
	Under either hypothesis the dimension
	bound makes $j_!H(X)$ perfect. Indeed, the $n$th syzygy in a resolution
	by finitely generated projectives is finitely generated and projective.
	Lemma~\ref{Lemma: P is fp Gro} now gives
	$H(X)\in\mathrm{gpr}_{\infty}(\cp,\cm)$.
	
	Next let $L$ be any $\cm$-pseudo-coherent module. Its induced module is
	perfect, so its finitely generated projective resolution can be chosen
	of length $d\leq n$. Since $\cm$ is additive and idempotent complete,
	each term is representable. Restriction, which is exact, gives
	\[
	0\longrightarrow H(M_d)\longrightarrow\cdots\longrightarrow
	H(M_0)\longrightarrow L\longrightarrow0,\qquad M_i\in\cm .
	\]
	Lemma~\ref{Lemma: P is fp Gro} gives
	$\operatorname{Ext}^{q}_{\cp}(H(M_i),P^{\wedge})=0$ for $q>0$.
	Dimension shifting gives $\operatorname{Ext}^{q}_{\cp}(L,P^{\wedge})=0$
	for $q>n$, as claimed.
	
	To identify the essential image, recall that $H$ is exact and fully
	faithful. Indeed, for a projective presentation arising from two
	conflations as above, both $\ce(X,Y)$ and
	$\operatorname{Hom}_{\cp}(HX,HY)$ identify with the kernel of
	$\ce(P_0,Y)\to\ce(P_1,Y)$. Comparing the cokernels of
	$\ce(P_0,Y)\to\ce(\Omega X,Y)$ also gives the natural isomorphism
	\[
	\operatorname{Ext}^{1}_{\ce}(X,Y)
	\xrightarrow{\sim}\operatorname{Ext}^{1}_{\cp}(HX,HY).
	\]
	Consequently the image of $H$ is extension closed, with its inherited
	exact structure.
	
	Suppose that $L\in\mathrm{gpr}_{\infty}(\cp,\cm)$. In the restricted
	resolution above, set $K_0=L$ and
	$K_{i+1}=\ker(H(M_i)\to K_i)$ for $0\leq i<d$. Thus $K_d\cong H(M_d)$,
	and the short exact sequences
	$0\to K_{i+1}\to H(M_i)\to K_i\to0$ give
	\[
	\operatorname{Ext}^{q}_{\cp}(K_i,P^{\wedge})
	\cong\operatorname{Ext}^{q+i}_{\cp}(L,P^{\wedge})=0
	\qquad(q>0,\ P\in\cp).
	\]
	We show by descending induction that every $K_i$ belongs to the image
	of $H$. Given $K_{i+1}\cong H(Z)$, fullness identifies its inclusion
	in $H(M_i)$ with $H(f)$ for some $f\colon Z\to M_i$. Choose an
	inflation $u\colon Z\rightarrowtail I$ into a projective-injective
	object. Since $\operatorname{Ext}^{1}_{\cp}(K_i,H(I))=0$, the map
	$H(u)$ extends to $H(M_i)\to H(I)$. By full faithfulness, we obtain a
	factorization $u=vf$.

	Since $\mathcal E$ is weakly idempotent complete and $u=vf$ is an
	inflation, applying \cite[Proposition~7.6]{Buhler2010} in
	$\mathcal E^{\mathrm{op}}$ shows that $f$ is an inflation. Applying $H$ to its conflation implies
	$K_i\cong H(\operatorname{coker} f)$. This proves the claimed
	essential surjectivity and hence the exact equivalence.
	
	Finally, if $\cm$ is right coherent, each
	$L\in\mathrm{gpr}_{\infty}(\cp)$ is finitely presented, so $j_!L$ is
	finitely presented and hence pseudo-coherent over $\cm$. The dimension
	bound makes $j_!L$ perfect, this proves the final equality.
\end{proof}

\section{Relationship with Plamondon's category $ \cd(\overline{Q},\overline{W}) $}\label{Section: Relationship with Plamondon's categor}
Let $(Q,F,W)$ be an ice quiver with potential. We make the following
standing assumption.
\begin{assumption}\label{assumption}
	The additive subcategory $\cp=\add(e\relGammabf)$ is functorially
	finite in $\add(\relGammabf)$.
\end{assumption}

\begin{Rem}
	The assumption is equivalent to the following two conditions for
	every unfrozen vertex $i\in Q_0\setminus F_0$.
	\begin{itemize}
		\item $\Hom_{\per\bmgamma}(e\relGammabf,e_i\relGammabf)$ is a
		finitely generated right $\End_{\per\bmgamma}(e\relGammabf)$-module.
		\item $\Hom_{\per\bmgamma}(e_i\relGammabf,e\relGammabf)$ is a
		finitely generated left $\End_{\per\bmgamma}(e\relGammabf)$-module.
	\end{itemize}
\end{Rem}
\begin{Rem}
	Suppose that $\ct=\add(\bmgamma)$ is a cluster-tilting subcategory
	of a stably $2$-Calabi--Yau Frobenius category $\ce$ whose
	projective-injective subcategory is $\cp\subseteq\ct$.
	Then Assumption~\ref{assumption} holds. Indeed, a deflation from
	a projective gives a right $\cp$-approximation, and an inflation
	into an injective gives a left $\cp$-approximation. One of our
	aims is to construct such a Frobenius category $\ce$ when
	$(\overline Q,\overline W)$ is Jacobi-finite.
\end{Rem}
\begin{Ex}
	Let $f$ be the only frozen vertex, let $F_1=\varnothing$, and
	consider the ice quiver with potential
	\[
	\begin{tikzcd}[column sep=large, row sep=large]
		{\color{blue}\boxed{f}} \arrow[r,"a"]
		& 1 \arrow[r,"b"] & 2 \arrow[d,"c"] \\
		&                 & 3 \arrow[ul,"d"']
	\end{tikzcd}
	\qquad W=0.
	\]
	Put $\Gamma=\bmgamma(Q,F,W)$, $\Gamma_i=e_i\Gamma$, and
	$A=H^0(\Gamma)=\widehat{kQ}$. Since no arrow ends at $f$, we have
	\[
	\End_{\per\Gamma}(\Gamma_f)=e_fAe_f=k.
	\]
	On the other hand, let $t=dcb$ be the cycle based at $1$.
	Every path from $f$ to $1$ is uniquely of the form $t^n a$
	with $n\geq 0$. Hence
	\[
	\Hom_{\per\Gamma}(\Gamma_f,\Gamma_1)
	=e_1Ae_f=k[[t]]a
	\]
	is infinite-dimensional over $k$.
	
	If $g\colon P\to\Gamma_1$ were a right
	$\add(\Gamma_f)$-approximation, then composition with $g$ would
	induce a surjection
	\[
	\Hom_{\per\Gamma}(\Gamma_f,P)
	\longrightarrow \Hom_{\per\Gamma}(\Gamma_f,\Gamma_1).
	\]
	The source is finite-dimensional because $P\in\add(\Gamma_f)$
	and $\End_{\per\Gamma}(\Gamma_f)=k$. This is a contradiction.
	Thus $\Gamma_1$ has no right $\add(\Gamma_f)$-approximation,
	and $\cp=\add(\Gamma_f)$ is not functorially finite in $\add(\Gamma)$.
	
	By deleting $f$, one obtains the oriented three-cycle with zero
	potential, whose completed Jacobian algebra is infinite-dimensional.
	For comparison, on the same ice quiver, we take $W=dcb$. Its cyclic
	derivatives give the relations $dc=bd=cb=0$. The Jacobian algebra \(J(\bar Q,\bar W)\) obtained by deleting \(f\) has dimension 6. In this case $Ae_f$ has basis
	$e_f,a,ba$, and $e_fA=ke_f$, where $A$ now denotes the relative
	Jacobian algebra for $W=dcb$. Thus Assumption~\ref{assumption} holds.
\end{Ex}

\begin{Prop}\label{prop:reduced-finite-implies-approximations}
	If $(\overline Q,\overline W)$ is Jacobi-finite, then
	$\cp=\add(e\relGammabf)$ is functorially finite in $\add(\relGammabf)$.
\end{Prop}
\begin{proof}
	Set $A = H^0(\bmgamma)$ and $B = eAe$, and let $J$ be the image of the arrow ideal in $A$. Vertex deletion induces an isomorphism $\bar A = J(\overline Q,\overline W) \cong A/\overline{AeA}$, where $\overline{AeA}$ denotes the closure of $AeA$ in the path topology.
	
	Next, we show that $M = Ae$ is finitely generated as a right $B$-module. Let $\mathfrak b = eJe$ and $N = \overline{M\mathfrak b} \subseteq Je$. Since $e(Je) = \mathfrak b \subseteq N$, the ideal $AeA$ annihilates the left $A$-module $Je/N$, and by continuity so does its closure $\overline{AeA}$. Grouping paths by their first arrow shows that $Je$ is generated as a left $A$-module by the finitely many arrows starting at frozen vertices; hence $Je/N$ is finitely generated over $\bar A$. Since both $\bar A$ and $M/Je = (A/J)e$ are finite-dimensional, it follows that $M/N$ is finite-dimensional as well.
	
	Choose lifts $m_1,\ldots,m_s \in M$ of a $k$-basis of $M/N$, and put $L = \sum_{i=1}^s m_i B$. For each $n \geq 1$, let $L_n$ denote the image of $L$ in the discrete quotient $M_n = M/J^n e$. Because $N$ and $M\mathfrak b$ have the same image in $M_n$, we have $M_n = L_n + M_n\mathfrak b$. Iterating this identity and using that $M_n\mathfrak b^n = 0$ yields $M_n = L_n$ for every $n \geq 1$, which shows that $L$ is dense in $M$.
	
	On the other hand, $L$ is also closed in $M$, as it is the image of the continuous map $B^{\oplus s} \to M$, $(b_1,\ldots,b_s) \mapsto \sum_{i=1}^s m_i b_i$, from the linearly compact space $B^{\oplus s}$. Thus $M = L$, so $Ae$ is finitely generated as a right $B$-module. Applying the same argument to $A^{\mathrm{op}}$ shows that $eA$ is finitely generated as a left $B$-module, and together these yield the right and left $\cp$-approximations, respectively.
\end{proof}

\subsection{Plamondon's category $ \cd(Q,W) $}\label{Subsection: Plamondon's category}
In~\cite{plamondonCategoriesAmasseesAux2011}, Plamondon generalized
Amiot's construction~\cite{amiotClusterCategoriesAlgebras2009} to
arbitrary quivers with potential. Let $(Q,W)$ be a quiver with
potential and let $\Gammabf(Q,W)$ be its complete Ginzburg dg algebra.
The generalized cluster category of $(Q,W)$ is the idempotent completion
of the Verdier quotient
\[
\cc(Q,W)=\operatorname{Kar}\!\left(
\per\bmgamma(Q,W)/\pvd(\bmgamma(Q,W))\right).
\]

\begin{Def}\rm\cite[Definition~3.9]{plamondonCategoriesAmasseesAux2011}\label{Def:Plamondon's cat}
	The category $\cd(Q,W)$ is the full subcategory of
	$\pr_{\cc(Q,W)}\bmgamma(Q,W)\cap\copr_{\cc(Q,W)}\bmgamma(Q,W)$
	whose objects are the $X$ such that
	$\Ext^1_{\cc(Q,W)}(\bmgamma(Q,W),X)$ is finite-dimensional.
\end{Def}

In our situation, $\overline Q$ is obtained from $Q$ by deleting all
frozen vertices and their incident arrows, and $\overline W$ is obtained
from $W$ by deleting all cycles passing through a frozen vertex. Put
$\overc=\cc(\overline Q,\overline W)$, and let $\overGamma$ be the
complete Ginzburg dg algebra of $(\overline Q,\overline W)$. Then
\[
\cd(\overline Q,\overline W)=
\bigl\{X\in\pr_{\overc}\overGamma\cap\copr_{\overc}\overGamma
\bigm| \dim_k\Ext^1_{\overc}(\overGamma,X)<\infty\bigr\}.
\]

\begin{Prop}\label{Prop: plamondon cat is extriangulated}
	The category $\cd(\overline Q,\overline W)$ is extension-closed in
	$\cc(\overline Q,\overline W)$. Thus it inherits an extriangulated
	structure in the sense of
	Nakaoka--Palu~\cite{nakaokaExtriangulatedCategoriesHovey2019}.
\end{Prop}
\begin{proof}
	When the frozen subquiver is empty, the Higgs category of
	$(\overline Q,\varnothing,\overline W)$ is precisely
	$\cd(\overline Q,\overline W)$. The assertion follows from
	Proposition~\ref{Prop: H is extension closed}.
\end{proof}

We write $p^*$ for extension of scalars along the vertex-deletion
morphism $p\colon\bmgamma\to\overGamma$ and for the induced functor
on the cluster categories.
Let $T=\pi^{\mathrm{rel}}\bmgamma$ and $\bar T=\pi\overGamma$.
For an object $U$, write $[U]$ for the ideal of morphisms factoring
through $\add(U)$.

\begin{Prop}\label{prop:higgs-embedding-plamondon}
	Under Assumption~\ref{assumption}, vertex deletion induces a fully
	faithful $k$-linear functor
	\[
	p^*\colon\ch/[\cp]\longrightarrow\cd(\overline Q,\overline W).
	\]
\end{Prop}
\begin{proof}
	Set $A=H^0(\bmgamma)$ and $\bar A=H^0(\overGamma)$, and denote the
	images of objects and morphisms under $p^*$ by a bar. We first examine
	the effect of $p^*$ on $\add(T)$. Choose a right $\cp$-approximation
	$P\to\bmgamma$ in $\add(\bmgamma)$, then the ideal $AeA$ is the image
	of the induced map $\Hom_{\per\bmgamma}(\bmgamma,P)\to A$. Because
	this map is continuous from a linearly compact space into the
	Hausdorff space $A$, its image $AeA$ is closed, and vertex deletion
	yields $\bar A\cong A/AeA$. 
	
	Applying
	Proposition~\ref{Prop:fully faithful of pr} in both the relative and
	the absolute settings, we see that $p^*$ induces an equivalence
	\[
	\add(T)/[\cp]\xrightarrow{\sim}\add(\bar T).
	\]
	
	Let $R=\Hom_{\cc}(T,\Si-)$ and $\bar R=\Hom_{\overc}(\bar T,\Si-)$.
	For any $X\in\ch$, take a copresentation $X\to U_0\to U_1\to\Si X$
	with $U_0,U_1\in\add(T)$. By rigidity, $RX$ is the cokernel of
	$\Hom_{\cc}(T,U_0)\to\Hom_{\cc}(T,U_1)$; applying $p^*$ yields a
	copresentation of $\bar X$ in $\overc$, so $\bar R\bar X$ is computed
	in the same way. 
	
	Right exactness of $-\otimes_A\bar A$, combined with
	the equivalence above, therefore gives a natural isomorphism
	$RX\otimes_A\bar A\xrightarrow{\sim}\bar R\bar X$. Since $X\in\ch$,
	the module $RX$ is finite-dimensional and satisfies
	$(RX)e=\Hom_{\cc}(\pi^{\mathrm{rel}}(e\bmgamma),\Si X)=0$. Thus $RX$
	is already an $\bar A$-module, and $\bar R\bar X\simeq RX$. Because
	$p^*$ preserves presentations and copresentations, we have
	$\bar X\in\cd(\overline Q,\overline W)$. Moreover, $p^*\cp=0$, so
	$p^*$ factors through $\ch/[\cp]$.
	
	Now let $X,Y\in\ch$. By Proposition~\ref{Prop: modules}\textup{(1)}
	applied to $\cc$ and $\overc$, the natural identification
	$\bar R\bar X\simeq RX$ implies that $p^*$ induces an isomorphism
	\[
	\frac{\Hom_{\cc}(X,Y)}{[T](X,Y)}
	\xrightarrow{\sim}
	\frac{\Hom_{\overc}(\bar X,\bar Y)}{[\bar T](\bar X,\bar Y)}.
	\]
	It remains to compare morphisms factoring through $\add(T)$. Choose a
	copresentation of $X$ and a presentation of $Y$ in $\cc$:
	\[
	X\xrightarrow{u} U_0\xrightarrow{b} U_1\to\Si X,
	\qquad
	V_1\xrightarrow{d}V_0\xrightarrow{v}Y\to\Si V_1.
	\]
	By rigidity, $u$ is a left $\add(T)$-approximation and $v$ is a right
	$\add(T)$-approximation, and the same holds for $\bar u$ and $\bar v$
	with respect to $\add(\bar T)$. Consequently, every morphism in
	$[\bar T](\bar X,\bar Y)$ has the form $\bar v h'\bar u$ for some
	$h'\colon\bar U_0\to\bar V_0$. Using the equivalence
	$\add(T)/[\cp]\xrightarrow{\sim}\add(\bar T)$, we can lift $h'$ to a
	morphism $h\colon U_0\to V_0$, so that
	$\bar v h'\bar u=\overline{vhu}$. Together with the displayed
	isomorphism, this shows that $p^*$ is full.
	
	To prove faithfulness, let $f\colon X\to Y$ be a morphism with
	$\bar f=0$. By the displayed isomorphism above, $f$ lies in
	$[T](X,Y)$, and hence $f=vhu$ for some $h\colon U_0\to V_0$. Since
	$\bar v\bar h\bar u=\bar f=0$ in $\overc$, exactness of the
	presentation of $\bar Y$ gives $\bar h\bar u=\bar d t$ for some
	$t\colon\bar X\to\bar V_1$. Because $\bar u$ is a left
	$\add(\bar T)$-approximation, $t$ factors as $t=a'\bar u$ for some
	$a'\colon\bar U_0\to\bar V_1$. Then $(\bar h-\bar d a')\bar u=0$, so
	exactness of the copresentation of $\bar X$ yields a morphism
	$c'\colon\bar U_1\to\bar V_0$ such that
	$\bar h-\bar d a'=c'\bar b$. 
	
	Lifting $a'$ and $c'$ to morphisms
	$a\colon U_0\to V_1$ and $c\colon U_1\to V_0$ in $\add(T)$, we see
	that $p^*(h-da-cb)=0$, so $h-da-cb$ factors through $\cp$. Since
	$vd=0$ and $bu=0$, we obtain
	\[
	f = vhu = v(h-da-cb)u \in [\cp](X,Y),
	\]
	which establishes faithfulness on $\ch/[\cp]$.
\end{proof}

\bigskip

The role of Assumption~\ref{assumption} in this proof is to ensure
that $AeA$ is closed. For any ice quiver with potential the same
copresentation calculation gives
$\bar R(p^*X)\simeq RX\otimes_A\bar A$, a quotient of the
finite-dimensional space $RX$ for $X\in\ch$. Thus $p^*$ always
induces a functor from $\ch/[\cp]$ to $\cd(\overline Q,\overline W)$.
It is fully faithful whenever $AeA$ is closed, and this condition
is also necessary for faithfulness. Indeed, on the endomorphisms of
$T$ the induced map is $A/AeA\to A/\overline{AeA}$.

\subsection{The case where $(\overline Q,\overline W)$ is Jacobi-finite}
\label{ss:reduced-Jacobi-finite-case}

Throughout this subsection assume that $(\overline Q,\overline W)$
is Jacobi-finite. By Proposition~\ref{prop:reduced-finite-implies-approximations},
Assumption~\ref{assumption} then holds.

By~\cite[Theorem~3.5]{amiotClusterCategoriesAlgebras2009}, the category
$\overc$ is Hom-finite and $2$-Calabi--Yau, and $\bar T$ is cluster-tilting.
Thus every object of $\overc$ admits both an $\add(\bar T)$-presentation
and an $\add(\bar T)$-copresentation. Since $\Ext^1_{\overc}(\bar T,-)$
is finite-dimensional on every object, we have
\[
\cd(\overline Q,\overline W)=\overc
\]
as in~\cite[Remark~3.11]{plamondonCategoriesAmasseesAux2011}.
By Definition~\ref{Defn: categoey D(Q,F,W)}, it follows that
$\cd(Q,F,W)=\cc(Q,F,W)$.

\begin{Prop}\label{Prop:pr is an equivalence}
	The functor $p^*\colon\per\bmgamma\to\per\overGamma$ induces an
	equivalence of $k$-linear categories
	\[
	\pr_{\cd}^{F}\bmgamma/[\cp]\xrightarrow{\sim}
	\pr_{\overline{\cd}}\overGamma,
	\]
	where $\cd=\cd(\bmgamma)$ and $\overline{\cd}=\cd(\overGamma)$.
\end{Prop}
\begin{proof}
	Since $\pr_{\cd}^{F}\bmgamma\subseteq\cz$, full faithfulness follows
	from Proposition~\ref{Prop:silting reduction-fully faithful}.
	In particular, $p^*$ is fully faithful on $\add(\bmgamma)/[\cp]$.
	Every indecomposable summand of $\overGamma$ lifts to the corresponding summand of $\bmgamma$, so $\add(\bmgamma)/[\cp]\simeq\add(\overGamma)$.
	
	To prove density, take a presentation
	$P_1\xrightarrow{\alpha}P_0\to M\to\Si P_1$ in $\per\overGamma$.
	Lift $\alpha$ to $b\colon T_1\to T_0$ in $\add(\bmgamma)$ and
	choose a left $\cp$-approximation $h\colon T_1\to P$.
	Every map from $T_1$ to an object of $\cp$ factors through $h$.
	Hence the cone of $(b,h)^t\colon T_1\to T_0\oplus P$ belongs to
	$\pr_{\cd}^{F}\bmgamma$. Its image under $p^*$ is isomorphic to
	$M$, since $p^*P=0$.
	
	We record the dual construction. Given a copresentation
	$M\to P_0\to P_1\to\Si M$, lift its middle map to
	$b\colon T_0\to T_1$ and choose a right $\cp$-approximation
	$h\colon P\to T_1$. Then
	$\Si^{-1}\cone((b,h)\colon T_0\oplus P\to T_1)$ belongs to
	$\copr_{\cd}^{F}\bmgamma$ and maps to $M$ up to isomorphism.
\end{proof}

\begin{Prop}\label{Prop: equivalences between two pr-copr in D}
	The functor $p^*\colon\per\bmgamma\to\per\overGamma$ induces an
	equivalence of $k$-linear categories
	\[
	\frac{\pr_{\cd}^{F}\bmgamma\cap\copr_{\cd}^{F}\bmgamma}{[\cp]}
	\xrightarrow{\sim}
	\pr_{\overline{\cd}}\overGamma\cap\copr_{\overline{\cd}}\overGamma.
	\]
\end{Prop}
\begin{proof}
	For either of the nonpositive dg algebras $A=\bmgamma,\overGamma$,
	we have
	\[
	\pr_{\cd(A)}A\cap\copr_{\cd(A)}A=\add(A).
	\]
	Indeed, a presentation $T_1\to T_0\to X\to\Si T_1$ gives
	$\Hom_{\cd(A)}(P,\Si X)=0$ for every $P\in\add(A)$, since
	$\Hom_{\cd(A)}(P,\Si T_0)=\Hom_{\cd(A)}(P,\Si^2 T_1)=0$.
	If $X$ also admits a copresentation
	$X\to P_0\to P_1\to\Si X$, its last morphism is therefore zero.
	The triangle splits, so $X$ is a summand of $P_0$. Conversely,
	every object of $\add(A)$ admits a split presentation and a split
	copresentation.
	
	The degree-zero comparison in the proof of
	Proposition~\ref{prop:higgs-embedding-plamondon} gives
	$\add(\bmgamma)/[\cp]\simeq\add(\overGamma)$ under
	Assumption~\ref{assumption}.
	Since $\add(\bmgamma)\subseteq\cz$, this proves the assertion.
\end{proof}

We next pass to the cluster categories.
\begin{Prop}\label{Prop: equivalence pr-copr}
	The functor $p^*\colon\cc(Q,F,W)\to\cc(\overline Q,\overline W)$
	induces equivalences of $k$-linear categories
	\[
	\begin{aligned}
		\pr_{\cc}^{F}\bmgamma/[\cp]
		&\xrightarrow{\sim}\pr_{\overc}\overGamma,\\
		\frac{\pr_{\cc}^{F}\bmgamma\cap\copr_{\cc}^{F}\bmgamma}{[\cp]}
		&\xrightarrow{\sim}
		\pr_{\overc}\overGamma\cap\copr_{\overc}\overGamma,
	\end{aligned}
	\]
	where $\cc=\cc(Q,F,W)$ and $\overc=\cc(\overline Q,\overline W)$.
\end{Prop}
\begin{proof}
	Before passing to the quotient, both source subcategories lie in
	$\cy$, so full faithfulness follows immediately from
	Proposition~\ref{Prop:silting reduction-fully faithful in cc}.
	To establish density, observe that by
	Proposition~\ref{Prop:fully faithful of pr} (applied in the case
	with no frozen vertices), any morphism between objects of
	$\add(\overGamma)$ in $\overc$ lifts to a morphism in
	$\per\overGamma$. Applying the two constructions from the proof of
	Proposition~\ref{Prop:pr is an equivalence} to such a lifted
	presentation or copresentation morphism and projecting to $\cc$, we
	obtain, via Proposition~\ref{Prop: equivalence between pr-copr}, a
	lift in $\pr_{\cc}^{F}\bmgamma$ of every presented object and a lift
	in $\copr_{\cc}^{F}\bmgamma$ of every copresented object. In
	particular, this proves density for the first equivalence.
	
	Now suppose that $M\in\overc$ admits both a presentation and a
	copresentation, and choose lifts $X\in\pr_{\cc}^{F}\bmgamma$ and
	$Y\in\copr_{\cc}^{F}\bmgamma$ of $M$ in $\cc$. Since
	$p^*X\simeq M\simeq p^*Y$, the full faithfulness of $p^*$ on
	$\cy/[\cp]$ gives morphisms $f\colon X\to Y$ and $g\colon Y\to X$
	such that $1_X-gf=ba$ for some $P\in\cp$ and morphisms
	$a\colon X\to P$ and $b\colon P\to X$. It follows that
	$(g,b)(f,a)^t=1_X$, so $X$ is a direct summand of $Y\oplus P$.
	As in the proof of Lemma~\ref{Lem: Higgs is idempotent complete},
	both $\copr_{\cc}\bmgamma$ and $\cy$ are closed under direct
	summands, and hence so is their intersection
	$\copr_{\cc}^{F}\bmgamma=\copr_{\cc}\bmgamma\cap\cy$. Since both
	$Y$ and $P$ belong to $\copr_{\cc}^{F}\bmgamma$, we conclude that
	$X\in\copr_{\cc}^{F}\bmgamma$ as well. Therefore, $X$ lies in
	$\pr_{\cc}^{F}\bmgamma\cap\copr_{\cc}^{F}\bmgamma$ and lifts $M$,
	which completes the proof.
\end{proof}

\begin{Cor}\label{Cor: equivalence of k-cat}
	The functor $p^*$ induces an equivalence of $k$-linear categories
	\[
	p^*\colon\ch/[\cp]\xrightarrow{\sim}\overc.
	\]
\end{Cor}
\begin{proof}
	Let $Y\in\pr_{\cc}^{F}\bmgamma\cap\copr_{\cc}^{F}\bmgamma$.
	Apply Lemma~\ref{Lemma: stable extension} below with
	$X=\pi^{\mathrm{rel}}\bmgamma\in\cy$. Since $p^*X\simeq\bar T$,
	it gives an isomorphism
	\[
	\Hom_{\cc}(X,\Si Y)\xrightarrow{\sim}
	\Hom_{\overc}(\bar T,\Si p^*Y).
	\]
	The right-hand side is finite-dimensional, so $Y\in\ch$.
	Every object of $\overc$ admits both a presentation and a
	copresentation by $\add(\bar T)$. The assertion now follows from
	the second equivalence of Proposition~\ref{Prop: equivalence pr-copr}.
\end{proof}

\begin{Lem}\label{Lemma: stable extension}
	Let $Y\in\pr_{\cc}^{F}\bmgamma\cap\copr_{\cc}^{F}\bmgamma$.
	For every $X\in\cy={}^{\perp}(\Si^{>0}\cp)
	\cap(\Si^{<0}\cp)^{\perp}\subseteq\cc$, the functor $p^*$
	induces an isomorphism
	\[
	\Hom_{\cc}(X,\Si Y)\xrightarrow{\sim}
	\Hom_{\overc}(p^*X,\Si p^*Y).
	\]
\end{Lem}
\begin{proof}
	Take a copresentation $Y\xrightarrow{u}T_0\to T_1\to\Si Y$ of $Y$.
	Assumption~\ref{assumption} provides a left $\cp$-approximation of
	$T_0$ in $\add(\bmgamma)$, which by
	Proposition~\ref{Prop:fully faithful of pr} remains an approximation
	in $\cc$. Since $\Hom_{\cc}(T_1,\Si P')=0$ for any $P'\in\cp$, every
	morphism $Y\to P'$ factors through $u$. Composing $u$ with the
	approximation of $T_0$ therefore produces a left $\cp$-approximation
	$f\colon Y\to P_Y$, which we complete to a triangle
	\[
	Y\xrightarrow{f}P_Y\xrightarrow{v}T\xrightarrow{w}\Si Y.
	\]
	Since $Y,P_Y\in\cy$, applying the two long exact Hom sequences to
	this triangle yields $T\in\cy$; here, the degree-one vanishing
	$\Hom_{\cc}(T,\Si P')=0$ for $P'\in\cp$ follows from the surjectivity
	of $\Hom_{\cc}(P_Y,P')\to\Hom_{\cc}(Y,P')$.
	
	Next, since $\Hom_{\cc}(P',\Si Y)=0$ for all $P'\in\cp$, every
	morphism $P'\to T$ factors through $v$, so the image of
	$\Hom_{\cc}(X,v)$ is precisely $[\cp](X,T)$. Applying
	$\Hom_{\cc}(X,-)$ to the triangle above and using that
	$\Hom_{\cc}(X,\Si P_Y)=0$, we see that $w$ induces an isomorphism
	\[
	\Hom_{\cc}(X,T)/[\cp](X,T)\xrightarrow{\sim}\Hom_{\cc}(X,\Si Y).
	\]
	By Proposition~\ref{Prop:silting reduction-fully faithful in cc},
	$p^*$ identifies the left-hand side with $\Hom_{\overc}(p^*X,p^*T)$.
	Since $p^*P_Y=0$, the morphism $p^*w\colon p^*T\to\Si p^*Y$ is also an
	isomorphism. Combining these identifications shows that the map
	$\Hom_{\cc}(X,\Si Y)\to\Hom_{\overc}(p^*X,\Si p^*Y)$ induced by
	$p^*$ is an isomorphism; in particular, it is natural in both
	variables.
\end{proof}

\begin{Prop}\label{Prop: bifunctorial isomorphisms in H}
	We have bifunctorial isomorphisms
	$D\Ext^1_{\ch}(X,Y)\simeq\Ext^1_{\ch}(Y,X)$ for all $X,Y\in\ch$.
\end{Prop}
\begin{proof}
	By Proposition~\ref{Prop: H is extension closed}, the induced
	extriangulated structure on $\ch$ has
	\[
	\Ext^1_{\ch}(X,Y)=\Hom_{\cc}(X,\Si Y).
	\]
	Apply Lemma~\ref{Lemma: stable extension} in both orders and
	use the $2$-Calabi--Yau duality of $\overc$, which is Hom-finite
	by~\cite[Theorem~3.5]{amiotClusterCategoriesAlgebras2009}.
\end{proof}

\bigskip

\begin{Thm}\label{Higgs category is a Silting reduction}
	The Higgs category $\ch$ is the full subcategory
	\[
	\cy={}^{\perp}(\Si^{>0}\cp)\cap(\Si^{<0}\cp)^{\perp}
	\]
	of $\cc(Q,F,W)$.
\end{Thm}

\begin{proof}
	By Definition~\ref{Def: Higgs cat}, we have $\ch\subseteq\cy$.
	Let $X\in\cy$. By Corollary~\ref{Cor: equivalence of k-cat},
	choose $H\in\ch$ with
	$p^*_{\cc}H\simeq p^*_{\cc}X$.
	Proposition~\ref{Prop:silting reduction-fully faithful in cc} gives
	a fully faithful functor $\cy/[\cp]\to\overc$ induced by $p^*_{\cc}$.
	Thus there are morphisms $f\colon X\to H$ and $g\colon H\to X$
	whose classes in $\cy/[\cp]$ are mutually inverse. In particular,
	$1_X-gf=ba$ for some $a\colon X\to P_0$ and $b\colon P_0\to X$
	with $P_0\in\cp$. Since
	\[
	\begin{pmatrix}g&b\end{pmatrix}
	\begin{pmatrix}f\\a\end{pmatrix}=1_X,
	\]
	$X$ is a direct summand of $H\oplus P_0$. Now $\cp\subseteq\ch$,
	and $\ch$ is closed under direct summands by
	Lemma~\ref{Lem: Higgs is idempotent complete}. Hence $X\in\ch$.
\end{proof}

\bigskip
Recall that an extriangulated category $\ce$ is \emph{Frobenius}
if it has enough projectives and enough injectives and its projective
objects coincide with its injective objects
\cite[Definition~3.2]{liuFrobeniusExangulatedCategories2020}.

\begin{Thm}\label{Thm: Jacobi-finite to Frobenius}
	The Higgs category $\ch(Q,F,W)$ is a Frobenius extriangulated
	category with projective-injective objects $\cp=\add(e\bmgamma)$.
	Moreover, for each $X\in\cc(Q,F,W)$, there exist $l\in\mathbb Z$,
	$U\in\ch$ and $P\in\thick_{\cc}(\cp)\simeq\per(e\bmgamma e)$
	such that there is a triangle in $\cc$
	\begin{equation}\label{triangle in C_rel}
		P\longrightarrow X\longrightarrow\Si^{l}U\longrightarrow\Si P.
	\end{equation}
	Dually, there exist $m\in\mathbb Z$, $V\in\ch$ and
	$Q\in\thick_{\cc}(\cp)$ such that there is a triangle
	\begin{equation}\label{triangle in C_rel^op}
		\Si^{m}V\longrightarrow X\longrightarrow Q
		\longrightarrow\Si^{m+1}V.
	\end{equation}
\end{Thm}
\begin{proof}
	\emph{Step 1. The Frobenius structure.}
	By Theorem~\ref{Higgs category is a Silting reduction}, we have
	$\ch=\cy$. In particular,
	$\Hom_{\cc}(P,\Si X)=0=\Hom_{\cc}(X,\Si P)$
	for $X\in\ch$ and $P\in\cp$. Thus every object of $\cp$ is both
	projective and injective for the induced extriangulated structure.
	
	Let $X\in\ch$. As in the proof of Lemma~\ref{Lemma: stable extension},
	a copresentation $X\to T_0\to T_1\to\Si X$ and a left
	$\cp$-approximation of $T_0$ give a left $\cp$-approximation
	$X\to I_0$. In the triangle
	\[
	X\longrightarrow I_0\longrightarrow X'\longrightarrow\Si X,
	\]
	the two long exact Hom sequences give $X'\in\cy=\ch$.
	This gives enough injectives. Dually, a presentation of $X$ and
	a right $\cp$-approximation give a triangle
	$X''\to P_0\to X\to\Si X''$ with $P_0\in\cp$ and $X''\in\ch$.
	Thus $\ch$ has enough projectives. If $X$ is injective, the first
	triangle splits and $X$ is a summand of $I_0$. If $X$ is projective,
	the second triangle splits and $X$ is a summand of $P_0$.
	This identifies both classes with $\cp$.
	
	\emph{Step 2. The two triangles.}
	Let $\mathcal B=\thick_{\cc}(\cp)$ and
	$q\colon\cc\to\cc/\mathcal B$ be the Verdier projection.
	Proposition~\ref{Prop:silting reduction-fully faithful in cc}
	identifies $\cc/\mathcal B$ with $\overc$. By
	Corollary~\ref{Cor: equivalence of k-cat}, for each $X\in\cc$
	there exists $H\in\ch$ with $qX\simeq qH$.
	Represent such an isomorphism by a roof
	$X\xleftarrow{s}L\xrightarrow{a}H$, with
	$K=\cone(s)\in\mathcal B$. Both $q(s)$ and $q(a)$ are isomorphisms.
	
	There is an integer $d$, depending on $K$, such that
	$\Hom_{\cc}(K,\Si^j V)=0$ for all $j>d$ and all $V\in\ch$.
	Indeed, the objects with this property form a thick subcategory
	and contain $\cp$, since $\ch\subseteq(\Si^{<0}\cp)^{\perp}$.
	Choose $l\geq1$ with $l+1>d$. Iterating the syzygy conflations
	of $H$ from Step~1 gives a composite of connecting morphisms
	\[
	b\colon H\longrightarrow\Si\Omega H\longrightarrow\cdots
	\longrightarrow\Si^l\Omega^l H
	\]
	which becomes invertible under $q$. Put $U=\Omega^l H\in\ch$.
	Applying $\Hom_{\cc}(-,\Si^l U)$ to
	$L\xrightarrow{s}X\to K\to\Si L$, and using
	$\Hom_{\cc}(\Si^{-1}K,\Si^l U)=0$, we obtain
	$c\colon X\to\Si^l U$ with $cs=ba$. Hence $q(c)$ is invertible.
	Since $\mathcal B$ is thick, $\Si^{-1}\cone(c)\in\mathcal B$.
	Completing $c$ to a triangle proves~\eqref{triangle in C_rel}.
	Applying the same argument in the opposite category, with
	cosyzygy conflations, gives~\eqref{triangle in C_rel^op}.
\end{proof}

By Proposition~\ref{Prop: H is extension closed}, the extriangulated
structure $(\ch,\mathbb E,\mathfrak s)$ is induced by the triangulated
structure of $\cc$. For $X,Y\in\ch$ it is given by
\[
\mathbb E(X,Y)=\Hom_{\cc}(X,\Si Y).
\]
For $\delta\in\mathbb E(X,Y)$ choose a distinguished triangle
\[
Y\xrightarrow{u}E\xrightarrow{v}X\xrightarrow{\delta}\Si Y .
\]
Then $E\in\ch$, and $\mathfrak s(\delta)=[Y\xrightarrow{u}E
\xrightarrow{v}X]$ is independent of the choice of triangle.

\begin{Prop}\label{Prop: equivalence of k-categories}
	The functor $p^*\colon\cc\to\cc(\overline Q,\overline W)$
	induces an equivalence of triangulated categories
	\[
	p^*\colon\ch/[\cp]\xrightarrow{\sim}\overc.
	\]
\end{Prop}
\begin{proof}
	The functor is an additive equivalence by
	Corollary~\ref{Cor: equivalence of k-cat}. By
	Theorem~\ref{Thm: Jacobi-finite to Frobenius}, $\ch$ is Frobenius
	extriangulated, so its stable category is triangulated
	\cite{liuFrobeniusExangulatedCategories2020}.
	Its suspension $S$ is given by conflations
	$Y\to I_Y\to SY$ with $I_Y\in\cp$. Applying $p^*$ to the
	corresponding triangle gives an isomorphism
	$p^*SY\xrightarrow{\sim}\Si p^*Y$, since $p^*I_Y=0$.
	Morphisms of these conflations show that this isomorphism is
	natural on $\ch/[\cp]$.
	
	Each distinguished triangle in the stable category is isomorphic
	to one induced by a conflation $Y\to E\to X$ in $\ch$.
	Its connecting map $X\to SY$, composed with $SY\to\Si Y$ in
	$\cc$, represents the extension in $\Hom_{\cc}(X,\Si Y)$.
	Under the suspension isomorphism just constructed, its image is
	therefore the triangle obtained by applying $p^*$ to
	$Y\to E\to X\to\Si Y$ in $\cc$. Thus the equivalence is exact.
\end{proof}

\subsection{Summary of results}

Let $ \ce $ be a Frobenius exact category and $ \cm $ a full additive subcategory of $ \ce $ which contains the full subcategory $ \cp $ of $ \ce $ formed by the projective-injective objects. We denote by $ \ck^{b}(\ce)  $ and $ \cd^{b}(\ce) $ respectively the bounded homotopy category and the bounded derived category of $ \ce $.

We say that a complex $ X\colon\cdots\ra X^{i-1}\ra X^{i}\ra X^{i+1}\ra\cdots $ in $ \ck^{b}(\ce) $ is \emph{$ \ce $-acyclic} if there are conflations $
\begin{tikzcd}
	Z^{i}\arrow[r,tail,"l^{i}"]&X^{i}\arrow[r,two heads,"\pi^{i}"]& Z^{i+1}
\end{tikzcd} $ such that $ d_{X}^{i}=l^{i+1}\circ\pi^{i} $ for each $ i\in\mathbb{Z} $.

We denote by $\ck^b(\cp)$ and $\ck^b(\cm)$ the bounded homotopy
categories of $\cp$ and $\cm$, viewed as full subcategories of
$\ck^b(\ce)$. We write $\ck^b_{\ce\text{-ac}}(\ce)$ for the full
subcategory on objects isomorphic to $\ce$-acyclic complexes and put
\[
\ck^b_{\ce\text{-ac}}(\cm)=
\ck^b(\cm)\cap\ck^b_{\ce\text{-ac}}(\ce).
\]

Combining Corollary~\ref{Cor: equivalence of k-cat}, Theorem~\ref{Thm: Jacobi-finite to Frobenius} and Theorem~\ref{Thm: structure thm without Noetherian}, we have the following result.

\begin{Thm}\label{Thm: main results}
	Assume that $(\overline Q,\overline W)$ is Jacobi-finite.
	Then Assumption~\ref{assumption} holds by
	Proposition~\ref{prop:reduced-finite-implies-approximations}.
	\begin{itemize}
		\item[1)] The functor $p^*$ induces an equivalence of $k$-linear categories
		\[
		\ch(Q,F,W)/[\cp]\xrightarrow{\sim}\cc(\overline Q,\overline W).
		\]
		\item[2)] The Higgs category $\ch$ equals the full subcategory
		of $\cc$ formed by the objects $X$ such that
		\[
		\Ext^i(X,P)=0=\Ext^i(P,X)
		\qquad\text{for all }i>0\text{ and }P\in\cp.
		\]
		Here $\Ext^i$ is computed in $\cc$. The category $\ch$ is Frobenius
		extriangulated with projective-injective subcategory $\cp$, and
		the equivalence in $1)$ is a triangle equivalence for the
		triangulated structure on $\ch/[\cp]$.
		Moreover, $T=\pi^{\mathrm{rel}}\bmgamma$ is a canonical cluster-tilting
		object of $\ch$, with $\End_{\ch}(T)\cong H^0(\bmgamma)=J(Q,F,W)$.
		\item[3)] Assume in addition that $H^i(\bmgamma)=0$ for $i\ne0$. Put
		$A=H^0(\bmgamma)=J_{rel}$ and $B=eAe$. In $\ch$ put
		$I=\pi^{\mathrm{rel}}(e\bmgamma)$ and $\cm=\add T$.
		Then $\ch$ is a Frobenius exact category. The boundary algebra $B$ is
		$\mathrm{fp}_{\infty}$-Gorenstein of dimension at most $3$ with
		respect to $A$ in the sense of
		Theorem~\ref{Thm: structure thm without Noetherian}.
		Restricted Yoneda gives an exact equivalence
		\[
		\Phi\colon\ch\xrightarrow{\sim}\gpr_{\infty}(B,A),
		\qquad X\longmapsto\Hom_{\ch}(I,X).
		\]
		Here $\gpr_{\infty}(B,A)$ consists of those
		$L\in\mathrm{fp}_{\infty}(B)$ for which
		$\operatorname{Ext}_B^q(L,B)=0$ for all $q>0$ and
		$L\otimes_B eA$ has a bounded resolution by finitely generated
		projective right $A$-modules. In particular, $\Phi(T)\cong Ae$.
		If $A$ is right coherent, then $\Phi$ induces an exact equivalence
		$\ch\simeq\gpr_{\infty}(B)$.
		
		\item[4)] Under the hypotheses of~$3)$, there is a commutative
		diagram of triangulated categories, up to natural isomorphism,
		whose rows are Verdier localization sequences and whose vertical
		arrows are triangle equivalences.
		\[
		\begin{tikzcd}[column sep=large]
			\ck^b_{\ch\text{-ac}}(\cm) \arrow[r,hook] \arrow[d,"\sim"']
			& \ck^b(\cm) \arrow[r] \arrow[d,"\sim"']
			& \cd^b(\ch) \arrow[d,"\sim"] \\
			\pvd_e(\bmgamma) \arrow[r,hook]
			& \per\bmgamma \arrow[r,"\pi^{\mathrm{rel}}"']
			& \cc(Q,F,W).
		\end{tikzcd}
		\]
		In particular, we have a triangle equivalence $\cc(Q,F,W)\simeq\cd^b(\ch)$, where $\cd^b(\ch)$
		is taken with respect to the exact structure in~$3)$.
	\end{itemize}
	
\end{Thm}

\begin{proof}
	Part~$1)$ is Corollary~\ref{Cor: equivalence of k-cat}, while the
	characterization of $\ch$ and its Frobenius structure in part~$2)$
	follow from Theorems~\ref{Higgs category is a Silting reduction}
	and~\ref{Thm: Jacobi-finite to Frobenius}. Since
	$T=\pi^{\mathrm{rel}}\bmgamma$ lies in $\cy=\ch$, the stable triangle
	equivalence of Proposition~\ref{Prop: equivalence of k-categories}
	sends $T$ to the cluster-tilting object $\overGamma$ of $\overc$
	\cite[Theorem~3.5]{amiotClusterCategoriesAlgebras2009}. We observe
	that the preimage of $\add(\overGamma)$ in $\ch$ is $\add(T)$: if
	the stable image of $X\in\ch$ is a retract of a finite direct sum of
	copies of $T$, then adjoining a suitable summand from
	$\cp\subseteq\add T$ lifts this retraction to $X$ itself.
	
	Consequently, both Ext-orthogonality conditions for $\add(T)$ follow
	from their counterparts in $\underline{\ch}$. Likewise, right and
	left approximations in $\underline{\ch}$ lift to approximations in
	$\ch$ by adding a projective deflation or an injective inflation,
	respectively, showing that $T$ is cluster-tilting in $\ch$. Finally,
	Proposition~\ref{Prop:fully faithful of pr} yields
	$\End_{\ch}(T)\cong\End_{\per\bmgamma}(\bmgamma)=H^0(\bmgamma)$.
	
	For parts~$3)$ and~$4)$, assume in addition that the cohomology of
	$\bmgamma$ is concentrated in degree zero. Then the quasi-isomorphism
	$\bmgamma\to A=H^0(\bmgamma)$ identifies $\per\bmgamma$ with
	$\per A$, matching the subcategories $\pvd_e(\bmgamma)$ and
	$\pvd_e(A)$. Throughout the remainder of the proof, we write
	$\pi\colon\per A\to\cc$, $T=\pi A$, $I=\pi(eA)$, $B=eAe$, and
	$\cm=\add T$.
	
	\smallskip
	\emph{Proof of $3)$.}
	We first show that $\Hom_{\cc}(T,\Si^{-\ell}T)=0$ for every
	$\ell\ge1$. Represent a morphism $T\to\Si^{-\ell}T$ by a roof
	$A\xleftarrow{s}U\to\Si^{-\ell}A$, and complete $s$ to a triangle
	$U\to A\to S\to\Si U$ with $S\in\pvd_e(A)$. Because $A$ is projective
	and concentrated in degree zero, the morphism $A\to S$ factors
	through $\tau_{\le0}S$. Replacing $U$ by the fiber of
	$A\to\tau_{\le0}S$ refines the roof, so we may assume that
	$S\in\cd^{\le0}(A)$. By Proposition~\ref{Prop:relative 3-CY}, we have
	\[
	\Hom_{\cd(A)}(S,\Si^{1-\ell}A)\cong D H^{\ell+2}(S)=0.
	\]
	Since $\Hom_{\cd(A)}(A,\Si^{-\ell}A)=0$ as well, the long exact
	sequence forces $\Hom_{\cd(A)}(U,\Si^{-\ell}A)=0$, so the roof
	represents the zero morphism.
	
	Now let $X,Y\in\ch$. Applying $\Hom_{\cc}(T,-)$ to a copresentation
	$Y\to T^0\to T^1\to\Si Y$ with $T^0,T^1\in\cm$ yields
	$\Hom_{\cc}(T,\Si^{-j}Y)=0$ for all $j\ge1$. Using a presentation
	$T_1\to T_0\to X\to\Si T_1$ of $X$, we then obtain
	$\Hom_{\cc}(X,\Si^{-1}Y)=0$. From the long exact Hom sequences
	associated with any conflation in $\ch$, it follows that every
	inflation is a monomorphism and every deflation is an epimorphism in
	$\ch$. Hence the extriangulated structure on $\ch$ is exact by
	\cite[Corollary~3.18]{nakaokaExtriangulatedCategoriesHovey2019}, and
	by~$2)$ it is Frobenius with projective-injective subcategory
	$\cp=\add I$.
	
	Next, we verify the dimension bound required by
	Theorem~\ref{Thm: structure thm without Noetherian}. Let
	$J=\operatorname{rad}A$. Since $J$ is the image of the arrow ideal
	and the vertex idempotents lift the primitive idempotents of
	$A/J\cong\prod_{i\in Q_0}k$, the algebra $A$ is semiperfect.
	Moreover, each simple left $A$-module $S_i$ ($i\in Q_0$) has
	projective dimension at most three. Indeed, grouping paths by their rightmost graded arrow decomposes the kernel of the augmentation $\bmgamma e_i \to S_i$, as a graded $\bmgamma$-module, into $\bigoplus_{\alpha} \bmgamma\alpha$, where $\alpha$ runs over the finitely many graded arrows starting at $i$, and each summand $\bmgamma\alpha$ is isomorphic to $\bmgamma e_{t(\alpha)}$ up to a grading shift of degree $\vert{}\alpha\vert{}$.
	
	 Since the graded arrows have degrees
	$0,-1,-2$ and the differential of each generator involves generators
	of higher degree, the cone of the inclusion of this kernel into
	$\bmgamma e_i$ provides a finite semi-free resolution of $S_i$ with
	generators in degrees $0,-1,-2,-3$. Extending scalars along
	$\bmgamma\to A$ yields the desired projective resolution of length at
	most three. In particular, the left $A$-module $A/J$ has projective
	dimension at most three.
	
	Now let $N\in\mathrm{fp}_{\infty}(A)$, and choose a minimal
	resolution $Q_\bullet\to N$ by finitely generated projective right
	$A$-modules. Because tensoring with $A/J$ annihilates the
	differentials of $Q_\bullet$, for every $j>3$ we obtain
	\[
	Q_j/Q_jJ\cong\operatorname{Tor}^A_j(N,A/J)=0.
	\]
	By Nakayama's lemma, $Q_j=0$ for all $j>3$, establishing the
	projective dimension bound for every pseudo-coherent right
	$A$-module.
	
	Set $F=\Hom_{\ch}(T,-)$ and $\Phi=\Hom_{\ch}(I,-)$. Evaluating at $T$
	and $I$ identifies induction from $\cp$ to $\cm$ with the functor
	$L\mapsto L\otimes_B eA$. To examine its effect on $\Phi X$ for
	$X\in\ch$, pick conflations $\Omega X\to P_0\to X$ and
	$\Omega^2X\to P_1\to\Omega X$ with $P_0,P_1\in\cp$, which together
	yield a projective presentation of $\Phi X$. Since induction sends
	$\Phi(P_i)$ to $F(P_i)$, the right exactness of $-\otimes_B eA$ and
	the left exactness of $F$ produce an exact sequence
	\[
	0\longrightarrow F(\Omega^2X)\longrightarrow F(P_1)
	\longrightarrow F(P_0)\longrightarrow(\Phi X)\otimes_B eA
	\longrightarrow 0.
	\]
	Using the cluster-tilting property from~$2)$, choose a right
	$\cm$-approximation conflation $M_1\to M_0\to\Omega^2X$ with
	$M_0,M_1\in\cm$. Applying $F$ gives a projective resolution of
	$F(\Omega^2X)$ of length at most one, and splicing this with the
	exact sequence above yields a resolution of $(\Phi X)\otimes_B eA$ of
	length at most three by finitely generated projective right
	$A$-modules.
	
	By Lemma~\ref{Lem: Higgs is idempotent complete} and part~$2)$, the
	category $\ch$ is idempotent complete and $\cm$ is contravariantly
	finite in $\ch$. Theorem~\ref{Thm: structure thm without Noetherian}
	therefore applies with $n=3$, yielding the relative
	$\mathrm{fp}_{\infty}$-Gorenstein bound for $B$ and an exact
	equivalence $\Phi\colon\ch\xrightarrow{\sim}\gpr_{\infty}(B,A)$ with
	$\Phi(T)\cong Ae$. If $A$ is right coherent, the functor $-\otimes_B eA$ sends any finitely presented right $B$-module to a finitely presented, hence pseudo-coherent, right $A$-module. By the projective dimension bound established above, every such module is perfect. It follows that
	$\gpr_{\infty}(B,A)=\gpr_{\infty}(B)$.
	
	\smallskip
	\emph{Proof of $4)$.}
	By~$2)$, the stable category $\underline{\ch}$ is Hom-finite and
	$2$-Calabi--Yau, and $\underline{\cm}=\cm/[\cp]$ is cluster-tilting
	in $\underline{\ch}$. Since $\ch$ is Frobenius exact,
	\cite[Lemma~2]{Palu2009} yields an equivalence
	\[
	\ck^b(\cm)/\ck^b_{\ch\text{-ac}}(\cm)
	\xrightarrow{\sim}\cd^b(\ch).
	\]
	Applying $F|_{\cm}$ termwise identifies $\ck^b(\cm)$ with
	$\ck^b(\operatorname{proj}A)\simeq\per A$, where
	$\operatorname{proj}A$ denotes the category of finitely generated
	projective right $A$-modules. Under this equivalence,
	\cite[Lemma~7]{Palu2009} shows that $\ck^b_{\ch\text{-ac}}(\cm)$
	corresponds to the subcategory of perfect complexes whose cohomology
	lies in $\mathrm{mod}\,\underline{\cm}$. Evaluation at $T$ identifies
	$\mathrm{mod}\,\underline{\cm}$ with the category of finitely
	presented right modules over
	\[
	\End_{\underline{\ch}}(T) = A/AeA \cong H^0(\overGamma),
	\]
	because $AeA$ is precisely the ideal of morphisms factoring through
	$\add I$ and the stable equivalence in~$2)$ maps $T$ to $\overGamma$.
	Since $H^0(\overGamma)$ is finite-dimensional, these are exactly the
	finite-dimensional right $A$-modules annihilated by $e$. Every
	complex in $\per A$ has bounded cohomology; conversely, the simple
	right $A$-modules at unfrozen vertices are perfect, as they admit
	finite semi-free resolutions over $\bmgamma$. By finite-length
	filtrations and truncation triangles, every object of $\pvd_e(A)$
	therefore lies in $\per A$. Consequently, the image of
	$\ck^b_{\ch\text{-ac}}(\cm)$ in $\per A$ is precisely $\pvd_e(A)$.
	
	Passing to Verdier quotients gives an equivalence
	$\per A/\pvd_e(A)\simeq\cd^b(\ch)$. Because $\ch$ is idempotent
	complete, so is $\cd^b(\ch)$ by
	\cite[Theorem~2.8]{BalmerSchlichting2001}. Hence the canonical
	functor from $\per A/\pvd_e(A)$ to its idempotent completion
	$\cc(Q,F,W)$ is an equivalence. Together with the identification of
	the kernels, this completes the proof of~$4)$.
\end{proof}

\section{Cluster characters}\label{Section: Cluster characters}
\label{s:cluster characters}
Suppose $k=\mathbb C$. Let $(Q,F,W)$ be an ice quiver with potential,
and let $\bmgamma$ be its complete relative Ginzburg dg algebra.
Write $Q_0=\{1,\ldots,n\}$ and $F_0=\{r+1,\ldots,n\}$, where
$1\leq r\leq n$. For $i\in Q_0$, let $\bmgamma_i=e_i\bmgamma$
and let $S_i$ be the simple right $J_{rel}$-module at $i$.
Put $e=\sum_{f\in F_0}e_f$ and
$\cp=\add(e\bmgamma)=\add(\bmgamma_{r+1}\oplus\cdots\oplus\bmgamma_n)$.
Throughout this section, we assume that $\cp$ is functorially finite in
$\add(\bmgamma)$ and that $(\overline Q,\overline W)$ is Jacobi-finite.
Let $\ch$ be the Higgs category of $(Q,F,W)$. We write
$\bmgamma_i$ also for its image in $\cc$ under $\pi^{\mathrm{rel}}$.
By Theorem~\ref{Thm: main results}, $\ch$ is Frobenius extriangulated
and its stable category is triangle equivalent to
$\overc=\cc(\overline Q,\overline W)$.

\begin{Def}\rm\label{Def: cc map}
	A \emph{cluster character} on $\ch$ with values in
	$\mathbb Q[x_{r+1},\ldots,x_n][x_1^{\pm1},\ldots,x_r^{\pm1}]$
	is a map
	\[
	X_{?}\colon\obj(\ch)\ra
	\mathbb Q[x_{r+1},\ldots,x_n][x_1^{\pm1},\ldots,x_r^{\pm1}]
	\]
	such that
	\begin{itemize}
		\item[1)] $X_L=X_{L'}$ whenever $L\simeq L'$,
		\item[2)] $X_{L\oplus M}=X_LX_M$ for all $L,M\in\ch$, and
		\item[3)] if $\dim\Ext^1_{\ch}(L,M)=1$ and
		\[
		L\ra E\ra M\ra\Si L,\qquad M\ra E'\ra L\ra\Si M
		\]
		are non-split triangles, then $X_LX_M=X_E+X_{E'}$.
	\end{itemize}
\end{Def}

\subsection{Index}\label{Subsection: index}
Let $ X $ be an object of $ \pr_{\cc}\bmgamma $. We define the \emph{index with respect to $ \bmgamma $} of $ X $ as the element of $ K_{0}(\add\relGammabf) $ given by
$$ \ind_{\bmgamma}X=[T_{0}^{X}]-[T_{1}^{X}] ,$$
where $ T_{1}^{X}\ra T_{0}^{X}\ra X\ra\Si T_{1}^{X} $ is an $ (\add\bmgamma) $-presentation of $ X $. If $ \widetilde{X} $ is the preimage of $ X $ under the $ k $-linear equivalence 
$$ \pr_{\cd}\bmgamma\ra\pr_{\cc}\bmgamma $$
induced by $ \pi^{rel} $ (cf.\ Proposition \ref{Prop:fully faithful of pr}), then $ \ind_{\bmgamma}(X) $ identifies with the class of $ \widetilde{X} $ in 
\[
K_{0}(\add\bmgamma)\iso K_{0}(\per\bmgamma).
\]
Thus, it is independent of the choice of presentation.


For a dimension vector $\mathbf d=(d_1,\ldots,d_r)$ supported on
the unfrozen vertices, put
\[
l(\mathbf d)=\sum_{i=1}^r d_i[S_i]\in K_0(\per\bmgamma)
\simeq K_0(\add\bmgamma).
\]
The simple modules $S_i$ at the unfrozen vertices are perfect. Hence
$l(\underline{\dim}N)=[N]$ for every finite-dimensional $J_{rel}$-module
$N$ supported on those vertices, by a composition series of $N$.

\begin{Lem}
	Let $X\in\pr_{\cc}\bmgamma\cap\copr_{\cc}\bmgamma$ and suppose that
	$RX=\Hom_{\cc}(\bmgamma,\Si X)$ is finite-dimensional. Then
	\[
	\ind_{\bmgamma} X+\ind_{\bmgamma}(\Si X)=l(\underline{\dim}RX).
	\]
\end{Lem}
\begin{proof}
	Write $\pi=\pi^{\mathrm{rel}}$. Choose a copresentation lift
	$B\in\copr_{\cd}\bmgamma$ of $X$. Full faithfulness on
	$\pr_{\cd}\bmgamma$ gives $H^1(B)\simeq RX$, since $\Si B$
	is a presentation lift of $\Si X$. By
	Proposition~\ref{Prop: modules}\,(4), $RX$ is supported on the
	unfrozen vertices. The truncation argument in the proof of
	Proposition~\ref{Prop: H is extension closed} applies to $B$:
	both $B$ and its perfect dual have cohomology in degrees at most $1$.
	Thus $A=\tau_{\le0}B$ is a presentation lift of $X$ and there is a
	triangle
	\[
	A\longrightarrow B\longrightarrow\Si^{-1}RX\longrightarrow\Si A.
	\]
	Consequently we have $\ind_{\bmgamma}X+\ind_{\bmgamma}(\Si X)
	=[A]-[B]=[RX]=l(\underline{\dim}RX)$.
\end{proof}

\begin{Lem}\label{Lemma:index of triangle}
	Let $X\xrightarrow{\alpha}Y\xrightarrow{\beta}Z
	\xrightarrow{\gamma}\Si X$ be a triangle in $\cc$ with
	$X,Z\in\pr_{\cc}\bmgamma$. Suppose that the module
	\[
	M=\coker\!\bigl(\Hom_{\cc}(\bmgamma,\beta)\bigr)
	=\coker\!\bigl(R(\Si^{-1}\beta)\bigr)
	\]
	is finite-dimensional. Then $Y\in\pr_{\cc}\bmgamma$, and
	\[
	\ind_{\bmgamma}X+\ind_{\bmgamma}Z
	=\ind_{\bmgamma}Y+l(\underline{\dim}M).
	\]
\end{Lem}

\begin{proof}
	Write $\Gamma=\bmgamma$, $\mathcal N=\pvd_e(\Gamma)$, and
	$\pi\colon\per\Gamma\to\cc$ for the quotient functor. We identify
	$K_0(\add\Gamma)$ with $K_0(\per\Gamma)$. Under this identification,
	the index of an object in $\pr_{\cc}\Gamma$ is the class of its
	lift in $\pr_{\cd(\Gamma)}\Gamma$.
	
	We first show that $Y$ admits such a lift. Write
	$P^\vee=\operatorname{RHom}_{\Gamma}(P,\Gamma)$ for the perfect
	dual, using the same notation for left modules, and let
	$\mathcal N^\ell$ denote the left-module analogue of $\mathcal N$.
	By relative $3$-Calabi--Yau duality
	(Proposition~\ref{Prop:relative 3-CY}), perfect duality exchanges
	$\mathcal N$ and $\mathcal N^\ell$, with
	\[
	H^j(S^\vee)\simeq D H^{3-j}(S)
	\qquad(S\in\mathcal N),\qquad D=\Hom_k(-,k).
	\]
	The same formula holds on the left by biduality. We also use the
	following criterion. If $V$ is perfect, $V\in\cd^{\le0}$ and
	$V^\vee\in\cd^{\le1}$, then $V\in\pr_{\cd(\Gamma)}\Gamma$.
	Indeed, for $L\in\cd^{\le-2}(\Gamma)$, the complex
	$L\otimes_\Gamma^{\mathbf L}V^\vee\simeq
	\operatorname{RHom}_{\Gamma}(V,L)$ has cohomology in degrees
	at most $-1$. The criterion follows from
	Lemma~\ref{Lemma: pr is extension closed}.
	
	Choose two-term lifts $P_X,P_Z$ of $X,Z$. Represent $\gamma$ by
	a roof $P_Z\xleftarrow{s}Z'\to\Si P_X$ with
	$\operatorname{Cone}(s)\in\mathcal N$, and complete its right
	arrow to a triangle $P_X\to U\to Z'\to\Si P_X$. Then $\pi U\simeq Y$.
	Since $P_X,P_Z\in\cd^{\le0}$ and their perfect duals lie in
	$\cd^{\le1}$, the long exact cohomology sequences for these
	triangles and their duals give
	\[
	\tau_{>0}U\in\mathcal N,\qquad
	\tau_{>1}(U^\vee)\in\mathcal N^\ell.
	\]
	Indeed, in the indicated degrees the only possible cohomology
	comes from subquotients of the cohomology of $\operatorname{Cone}(s)$
	and its dual.
	
	Put $P=\tau_{\le0}U$. It is perfect and $\pi P\simeq Y$, since
	$\tau_{>0}U\in\mathcal N$. Dualizing the truncation triangle shows
	that $L=\tau_{>1}(P^\vee)$ belongs to $\mathcal N^\ell$.
	Set $V=(\tau_{\le1}(P^\vee))^\vee$. The dual truncation triangle is
	\[
	L^\vee\longrightarrow P\longrightarrow V
	\longrightarrow\Si L^\vee .
	\]
	Since $L$ has cohomology in degrees at least $2$, the duality
	formula gives $L^\vee\in\mathcal N\cap\cd^{\le1}$. Hence
	$\pi V\simeq Y$ and $V\in\cd^{\le0}$. Moreover,
	$V^\vee\simeq\tau_{\le1}(P^\vee)\in\cd^{\le1}$.
	The criterion above now gives $Y\in\pr_{\cc}\Gamma$.
	
	To compute the index, choose a two-term lift $P_Y$ of $Y$.
	By Proposition~\ref{Prop:fully faithful of pr}, $\beta$ lifts to
	a morphism $P_Y\to P_Z$. Complete it to a triangle
	\[
	B\longrightarrow P_Y\longrightarrow P_Z
	\longrightarrow\Si B .
	\]
	Then $\pi B\simeq X$. Full faithfulness identifies
	$H^0(P_Y)\to H^0(P_Z)$ with $\Hom_{\cc}(\Gamma,\beta)$, so
	$H^1(B)\simeq M$. Moreover, $M$ is supported on the unfrozen
	vertices. Indeed, the long exact sequence embeds $M$ in $RX$,
	and for each frozen vertex $f$ we have
	\[
	(RX)e_f=\Hom_{\cc}(\pi\Gamma_f,\Si X)
	\simeq\Hom_{\cd(\Gamma)}(\Gamma_f,\Si P_X)=0.
	\]
	Here $\Gamma_f=e_f\Gamma$ is left orthogonal to $\mathcal N$.
	Thus $M\in\mathcal N$.
	
	By the defining triangle for $B$, both $B$ and $B^\vee$ have
	cohomology in degrees at most $1$. Put $B_0=\tau_{\le0}B$.
	The cone of $B_0\to B$ is $\Si^{-1}M$, so $\pi B_0\simeq X$.
	The dual triangle
	\[
	\Si^{-2}DM\longrightarrow B^\vee\longrightarrow B_0^\vee
	\longrightarrow\Si^{-1}DM
	\]
	gives $B_0^\vee\in\cd^{\le1}$. The criterion above shows that
	$B_0$ is a two-term lift of $X$. Comparing the two expressions
	for $[B]$ gives
	\[
	\ind_\Gamma Y-\ind_\Gamma Z=[B]
	=[B_0]-[M]=\ind_\Gamma X-[M].
	\]
	
	The equality $[M]=l(\underline{\dim}M)$ gives the asserted formula.
\end{proof}

\bigskip
For each $M\in\ch$, define
\begin{equation}\label{CC formula for Higgs}
	X_M=x^{\ind_{\bmgamma}M}
	\sum_{\mathbf d}\chi\!\left(\Gr_{\mathbf d}(RM)\right)x^{-l(\mathbf d)}.
\end{equation}
Here $\Gr_{\mathbf d}(RM)$ is the quiver Grassmannian of submodules of
$RM$ of dimension vector $\mathbf d$, and $\chi$ is the topological
Euler characteristic. The sum is finite because $RM$ is finite-dimensional.

We check that each monomial occurring in this formula has nonnegative
exponent at every frozen vertex. As in the first lemma of this subsection,
choose a copresentation lift $B$ of $M$ and put $A=\tau_{\le0}B$, so that
$[A]=\ind_{\bmgamma}M$ and $H^1(B)=RM$. Given $U\subseteq RM$, put
$D_U=RM/U$ and complete the composite
$B\to\Si^{-1}RM\to\Si^{-1}D_U$ to a triangle. The octahedral axiom and
the short exact sequence $0\to U\to RM\to D_U\to0$ give triangles
\[
V\longrightarrow B\longrightarrow\Si^{-1}D_U\longrightarrow\Si V,
\qquad
A\longrightarrow V\longrightarrow\Si^{-1}U\longrightarrow\Si A .
\]
In particular, $V$ is perfect and
$[V]=\ind_{\bmgamma}M-l(\underline{\dim}U)$.

Fix a frozen vertex $f$ and write $S_f$ for its simple module. For
$H=U,D_U$, disjoint support gives $\Hom_{J_{rel}}(H,S_f)=0=
\Hom_{J_{rel}}(S_f,H)$. The standard $t$-structure and relative
$3$-Calabi--Yau duality (Proposition~\ref{Prop:relative 3-CY}) therefore give
\[
\Hom_{\cd(\bmgamma)}(H,\Si^j S_f)=0\qquad(j\notin\{1,2\}).
\]
Write $\Hom^j(P,S_f)=\Hom_{\cd(\bmgamma)}(P,\Si^jS_f)$.
The presentation of $A$ gives $\Hom^j(A,S_f)=0$ for $j<0$, and the
copresentation of $B$ gives $\Hom^j(B,S_f)=0$ for $j>0$. Applying
$\Hom^j(-,S_f)$ to the second triangle for $j<0$ and to the first
for $j>0$ shows that $\Hom^j(V,S_f)=0$ for $j\ne0$.
The Euler pairing with $S_f$ extracts the coefficient of $[\bmgamma_f]$
in $K_0(\per\bmgamma)$, since
$\dim\Hom^j(\bmgamma_i,S_f)=\delta_{if}\delta_{j0}$. It follows that
the exponent of $x_f$ in this term is $\dim\Hom^0(V,S_f)\ge0$. Hence
\[
X_{?}\colon\obj(\ch)\ra
\mathbb Q[x_{r+1},\ldots,x_n][x_1^{\pm1},\ldots,x_r^{\pm1}].
\]

\begin{Thm}\label{Thm: relative cc map}
	The map $ X_{?} $ defined above is a cluster character on $ \ch $.
\end{Thm}
We prove this Theorem in the next subsection.

\subsection{Multiplication formula}\label{Subsection: Multiplication formula}
Let $ L $ and $ M $ be objects of $ \ch $ such that $ \dim\Ext_{\ch}^{1}(L,M)=1 $. By Proposition~\ref{Prop: bifunctorial isomorphisms in H}, we also have $ \dim\Ext_{\ch}^{1}(M,L)=1 $. Let
$$ L\xrightarrow{i}E\xrightarrow{p}M\xrightarrow{\epsilon}\Si L $$
and
$$ M\xrightarrow{i'}E'\xrightarrow{p'}L\xrightarrow{\epsilon'}\Si M $$
be non-split triangles in $ \cc $. Recall that $ R $ is the functor $ R=\Ext^{1}_{\cc}(\bmgamma,?)\colon\cc\ra \Mod J_{rel} $. For any submodules $ U $ of $ R(L) $ and $ V $ of $ R(M) $, define
$$ G_{U,V}=\{W\in\bigcup_{e}\Gr_{e}(R(E))\,|\,(Ri)^{-1}(W)=U,\,(Rp)(W)=V\} $$
and
$$ G'_{U,V}=\{W\in\bigcup_{e}\Gr_{e}(R(E'))\,|\,(Ri')^{-1}(W)=V,\,(Rp')(W)=U\} .$$

\begin{Prop}
	Let $U$ and $V$ be as above. Then exactly one of $ G_{U,V} $ and $ G'_{U,V} $ is non-empty.
\end{Prop}
\begin{proof}
	By Proposition~\ref{Prop: H is extension closed}, $E,E'\in\ch$.
	Lemma~\ref{Lemma: stable extension} identifies the two extension
	spaces, in each order, with those of $p^*L,p^*M$ in $\overc$.
	In particular both triangles remain non-split after applying $p^*$.
	The same lemma with first argument $\bmgamma$ gives, naturally in
	$Y\in\ch$, an identification of $RY$ with
	$\Hom_{\overc}(\overGamma,\Si p^*Y)$, viewed as a $J_{rel}$-module
	via vertex deletion. It identifies the submodules and the maps
	defining $G_{U,V}$ and $G'_{U,V}$ with their counterparts in $\overc$.
	The assertion follows from
	\cite[Proposition~3.13]{plamondonCategoriesAmasseesAux2011}.
\end{proof}

\smallskip

For dimension vectors $e,f,g$, define the following constructible subsets.
$$ G_{e,f}=\bigcup_{\substack{\underline{\dim}U=e\\
		\underline{\dim}V=f}}G_{U,V} $$
$$ G'_{e,f}=\bigcup_{\substack{\underline{\dim}U=e\\
		\underline{\dim}V=f}}G'_{U,V} $$
$$ G_{e,f}^{g}=G_{e,f}\cap\Gr_{g}(R(E)) $$
$$ G'^{g}_{e,f}=G'_{e,f}\cap\Gr_{g}(R(E')). $$

\begin{Lem}\cite[Lemma 3.17]{plamondonCategoriesAmasseesAux2011}\label{Lemma:spli of Gr}
	With the notations above, we have
	$$ \chi(\Gr_{e}(R(L)))\cdot\chi(\Gr_{f}(R(M)))=\sum_{g}(\chi(G^{g}_{e,f})+\chi(G'^{g}_{e,f})). $$
\end{Lem}

\begin{Lem}\cite[Lemma 3.18]{plamondonCategoriesAmasseesAux2011}\label{Lemma: e,f,g with cokernel}
	If $G^g_{e,f}$ is nonempty, then
	\[
	\underline{\dim}\coker R(\Si^{-1}p)=e+f-g.
	\]
	If $G'{}^g_{e,f}$ is nonempty, the corresponding formula holds
	with $p'$ in place of $p$.
\end{Lem}
\begin{proof}
	For $W\in G^g_{e,f}$, put $U=(Ri)^{-1}W$ and $V=(Rp)W$.
	Exactness of $RL\to RE\to RM$ gives
	$0\to (Ri)U\to W\to V\to0$. Since $\ker(Ri)\subseteq U$,
	this gives $e+f-g=\underline{\dim}\ker(Ri)$. The long exact
	sequence identifies this kernel with $\coker R(\Si^{-1}p)$.
	The proof for $p'$ is the same with $L$ and $M$ interchanged.
\end{proof}

\begin{Proof}\textbf{of Theorem~\ref{Thm: relative cc map}.}
	The character is invariant under isomorphisms. The index is additive
	on direct sums, and scaling one summand of $R(L\oplus M)=RL\oplus RM$
	gives a $\mathbb C^{\times}$-action on each quiver Grassmannian. Its
	fixed locus is the disjoint union of the products
	$\Gr_e(RL)\times\Gr_f(RM)$ with $e+f$ equal to the prescribed
	dimension vector. The Euler characteristic equals that of the fixed
	locus, so $X_{L\oplus M}=X_LX_M$.
	
	For the multiplication formula, take $L,M,E,E'$ as above and put
	$h=\ind_{\bmgamma}L+\ind_{\bmgamma}M$. On each nonempty stratum
	$G^g_{e,f}$, Lemmas~\ref{Lemma:index of triangle}
	and~\ref{Lemma: e,f,g with cokernel} and the additivity of $l$ give
	\[
	h-l(e+f)=\ind_{\bmgamma}E-l(g).
	\]
	On $G'{}^g_{e,f}$ the same identity holds with $E'$ in place of $E$.
	By Lemma~\ref{Lemma:spli of Gr}, we obtain
	\[
	\begin{aligned}
		X_LX_M
		&=\sum_{e,f}\chi(\Gr_e(RL))\chi(\Gr_f(RM))\,x^{h-l(e+f)}\\
		&=\sum_g\left(
		x^{\ind_{\bmgamma}E-l(g)}\sum_{e,f}\chi(G^g_{e,f})
		+x^{\ind_{\bmgamma}E'-l(g)}\sum_{e,f}\chi(G'{}^g_{e,f})\right)\\
		&=\sum_g\chi(\Gr_g(RE))\,x^{\ind_{\bmgamma}E-l(g)}
		+\sum_g\chi(\Gr_g(RE'))\,x^{\ind_{\bmgamma}E'-l(g)}
		=X_E+X_{E'} .
	\end{aligned}
	\]
	Indeed, for fixed $g$ the $G^g_{e,f}$ partition $\Gr_g(RE)$ into
	constructible subsets, and the primed strata partition $\Gr_g(RE')$.
\end{Proof}

\begin{Rem}\label{rem:relative-character-frozen-coefficients}
	Write $\Gamma=\bmgamma$ and $\bar\Gamma=\overline{\bmgamma}$.
	Let $\iota:\mathbb Z^r\hookrightarrow\mathbb Z^n$ extend a vector
	by zero on the frozen vertices, let $\pi_{\mathrm{mut}}$ be the
	projection onto the first $r$ coordinates, and put
	\[
	g^F=g-\iota\pi_{\mathrm{mut}}g
	\qquad (g\in\mathbb Z^n).
	\]
	For $N\in\cd(\bar Q,\bar W)$ define
	\[
	\widetilde{CC}(N)
	=
	x^{\iota(\ind_{\bar\Gamma}N)}
	\sum_{\mathbf d\in\mathbb N^r}
	\chi\!\left(\Gr_{\mathbf d}(\bar R N)\right)
	x^{-l(\mathbf d)} .
	\]
	Here $\bar R N=\Hom_{\overline{\cc}}(\bar\Gamma,\Sigma N)$,
	and $l(\mathbf d)\in\mathbb Z^n$ is the relative index map of
	Subsection~\ref{Subsection: index}, evaluated on the dimension
	vector supported on the unfrozen vertices. Write $\bar l$ for the
	corresponding map for $(\bar Q,\bar W)$.
	
	For every $M\in\ch$ we have
	\[
	X_M
	=
	x^{(\ind_\Gamma M)^F}\,\widetilde{CC}(p^*M).
	\]
	Indeed, applying $p^*$ to an $\add\Gamma$-presentation of $M$ gives
	\[
	\pi_{\mathrm{mut}}(\ind_\Gamma M)
	=\ind_{\bar\Gamma}(p^*M).
	\]
	Put $A=J(Q,F,W)$ and $\bar A=J(\bar Q,\bar W)$. Applying
	$\Hom_{\cc}(\Gamma,{-})$ to a presentation of $\Si M$ gives a
	projective presentation of $RM$, since $\Gamma$ is rigid in $\cc$.
	Applying the same construction after $p^*$ tensors that projective
	presentation with $\bar A$. Thus
	\[
	\bar R(p^*M)\cong RM\otimes_A\bar A\cong RM,
	\]
	where the last isomorphism uses $RM\,e=0$ and
	$\bar A=A/AeA$ under the standing Jacobi-finiteness hypothesis.
	In particular the two modules have the same quiver Grassmannians.
	For each $\mathbf d\in\mathbb N^r$, choose
	$Z\in\pr_{\cc}\Gamma\cap\copr_{\cc}\Gamma$ with
	$RZ\simeq\bigoplus_{i=1}^r S_i^{d_i}$, using
	Proposition~\ref{Prop: modules}\,(4). The same computation gives
	$\bar R(p^*Z)\simeq RZ$ as a $\bar A$-module.
	Applying the index comparison to $Z$ and $\Si Z$ gives
	\[
	\pi_{\mathrm{mut}}l(\mathbf d)=\bar l(\mathbf d).
	\]
	Consequently, specialization of the frozen variables gives
	Plamondon's coefficient-free cluster character
	\[
	\left.X_M\right|_{x_{r+1}=\cdots=x_n=1}
	=\overline{CC}(p^*M).
	\]
	For example, take $Q=(1\to2)$ with $2$ frozen, $W=0$ and
	$M=\operatorname{Cone}(\Gamma_1\to\Gamma_2)$. Here $RM=S_1$,
	$\ind_\Gamma M=[\Gamma_2]-[\Gamma_1]$ and $l(1)=[\Gamma_2]$.
	The two submodules of $S_1$ give
	$X_M=(x_2/x_1)(1+x_2^{-1})=(1+x_2)/x_1$, whose specialization
	at $x_2=1$ is $\overline{CC}(p^*M)=2/x_1$.
	
	Assume further that $(Q,W)$, viewed as an ordinary quiver with
	potential on all $n$ vertices, is non-degenerate. Then
	$X_M\in\cu_{Q,F}^{+}$ for every $M\in\ch$, with the frozen
	variables kept non-invertible.
	
	Indeed, let $\Gamma^{\mathrm a}=\bmgamma(Q,W)$ be the ordinary
	Ginzburg algebra. The inclusion of completed graded path algebras
	$u\colon\Gamma\hookrightarrow\Gamma^{\mathrm a}$ preserves the
	differential. Indeed, the extra opposite arrows and degree $-2$
	loops have both endpoints frozen, so the commutator terms in
	$d(t_i)$ are the same for each unfrozen $i$. The differentials
	of the original arrows and of the common opposite arrows also agree.
	For such an $i$, the standard semi-free resolution of $S_i$ is
	indexed by the graded arrows ending at $i$ together with the
	vertex projective at $i$.
	
	Notice that extension of scalars takes this resolution
	to the corresponding absolute resolution and hence takes $S_i$ to
	the absolute simple at $i$. Since the $S_i$ for $i$ unfrozen
	generate $\pvd_e(\Gamma)$ as a thick subcategory, extension of
	scalars descends through the Verdier quotients and their idempotent
	completions to a triangle functor $u^*\colon\cc\to\cc(Q,W)$.
	It sends $T_i$ to $T_i^{\mathrm a}$ for every $1\leq i\leq n$, and hence preserves all $n$ coordinates of the index of any object in $\pr_{\cc}T$.
	
	Put $M^{\mathrm a}=u^*M$ and
	$R^{\mathrm a}=\Hom_{\cc(Q,W)}(\Gamma^{\mathrm a},\Si\,{-})$.
	Computing $R^{\mathrm a}M^{\mathrm a}$ from the image of an
	$\add\Gamma$-copresentation of $M$ gives
	\[
	R^{\mathrm a}M^{\mathrm a}
	\cong RM\otimes_{J(Q,F,W)}J(Q,W)\cong RM.
	\]
	Indeed, $A\to J(Q,W)$ is surjective and each extra relation
	$\partial_aW$, $a\in F_1$, has frozen endpoints. These relations
	act trivially on $RM$ since $RM\,e=0$. Their closed ideal also
	acts trivially since $RM$ is finite-dimensional and nilpotent.
	The same comparison applies to the lift $Z$ chosen above for
	$\mathbf d$. Since both $Z$ and $\Si Z$ admit
	presentations, their indices, and hence their sum, are preserved.
	Thus $M^{\mathrm a}$ admits both a presentation and a copresentation
	and $R^{\mathrm a}M^{\mathrm a}$ is finite-dimensional. It belongs
	to Plamondon's subcategory. Since the indices, the exponents
	$l(\mathbf d)$ on the support of $RM$, and the quiver
	Grassmannians agree, its ordinary cluster character in all $n$
	variables is $X_M$.
	
	By \cite[Propositions~4.1 and~4.8 and Corollary~4.14]
	{PlamondonClusterAlgebrasVia2011}, in each seed $t$ obtained by
	mutations at unfrozen vertices, $X_M$ is the character of the
	corresponding decorated representation $(L,V)$ in the seed variables.
	Its module part is initially $RM$, and mutation changes the vertex
	space only at the mutated vertex. Thus $L_f=0$ for every frozen
	vertex $f$. Let
	$b_{ij}(t)=\#(i\to j)-\#(j\to i)$ in the quiver at $t$.
	For a submodule $U\subseteq L$, put
	\[
	U_{\mathrm{in}}=\bigoplus_{a:f\to j}U_j,\qquad
	U_{\mathrm{out}}=\bigoplus_{b:j\to f}U_j
	\]
	using the arrows of the quiver at $t$, and define
	$L_{\mathrm{in}}$ and $L_{\mathrm{out}}$ similarly.
	With our right-module convention, the DWZ map (\cite{DerksenWeymanZelevinsky08})
	$\gamma_f\colon L_{\mathrm{out}}\to L_{\mathrm{in}}$ sends
	$U_{\mathrm{out}}$ into $U_{\mathrm{in}}$, since its entries
	act by linear combinations of paths. Since $L_f=0$, the DWZ
	formula for the $g$-vector gives
	$g_f=\dim\ker\gamma_f+\dim V_f$.
	Hence the exponent of $x_f$
	in the term indexed by $\underline{\dim}\,U$ satisfies
	\[
	g_f+\sum_j b_{fj}(t)\dim U_j
	=\dim V_f+\dim\ker\gamma_f
	+\dim U_{\mathrm{in}}-\dim U_{\mathrm{out}}
	\geq \dim V_f\geq 0,
	\]
	where the first inequality follows from
	$\dim U_{\mathrm{out}}\leq \dim\ker\gamma_f+\dim U_{\mathrm{in}}$
	by rank-nullity on $\gamma_f|_{U_{\mathrm{out}}}$.
	Thus $X_M$ is polynomial in the frozen variables and Laurent in
	the unfrozen variables of every seed. Writing $x_{j;t}$ for the
	$j$-th cluster variable at $t$, we obtain
	\[
	X_M\in\bigcap_t
	\mathbb{Q}[x_{r+1},\ldots,x_n]
	[x_{1;t}^{\pm1},\ldots,x_{r;t}^{\pm1}]
	=\cu_{Q,F}^{+}.
	\]
\end{Rem}

Recall that a basic cluster-tilting object of $\ch$ is \emph{reachable
	from $T=\pi^{\mathrm{rel}}\bmgamma$} if it is obtained from $T$ by a
finite sequence of mutations at non-projective summands. An
indecomposable rigid object is \emph{reachable} if it is a summand of
a cluster-tilting object reachable from $T$.

\begin{Thm}\label{Thm: cluster bijection}
	Let $ (Q,F) $ be an ice quiver and $ W $ a non-degenerate potential on $ Q $. Assume that \((\bar Q,\bar W)\) is Jacobi-finite. Let $\ch$ be the associated Higgs category and
	put $T=\pi^{\mathrm{rel}}\bmgamma=\bigoplus_{i=1}^{n}T_i$.
	We include the frozen variables among the cluster variables and keep
	them non-invertible. Then the cluster character
	$X_{?}\colon\obj(\ch)\ra\cu_{Q,F}^{+}$ induces a bijection
	\[
	\left.
	\left\{\begin{array}{@{}c@{}}
		\text{reachable rigid}\\[-1pt]
		\text{indecomposable objects of }\ch
	\end{array}\right\}
	\right/ {\cong}
	\ \xrightarrow{\ \sim\ }\
	\left\{\begin{array}{@{}c@{}}
		\text{cluster variables}\\[-1pt]
		\text{of }\ca_{Q,F}
	\end{array}\right\}
	\ \subseteq\ \cu_{Q,F}^{+}.
	\]
	For each frozen vertex $f$, the summand $T_f$ corresponds to $x_f$.
	Under this bijection, the isomorphism classes of basic cluster-tilting
	objects reachable from $T$ correspond to the extended clusters of
	$\ca_{Q,F}$.
\end{Thm}

\begin{proof}
	Write $\cc^{\mathrm a}=\cc(Q,W)$ for the ordinary cluster category on
	all vertices, and put $M^{\mathrm a}=u^*M$ for $M\in\ch$, where
	$u^*\colon\cc(Q,F,W)\to\cc^{\mathrm a}$ is the triangle functor of
	Remark~\ref{rem:relative-character-frozen-coefficients}. By that
	remark, $M^{\mathrm a}$ lies in Plamondon's subcategory and satisfies
	\[
	R^{\mathrm a}M^{\mathrm a}\cong RM,
	\qquad
	CC^{\mathrm a}(M^{\mathrm a})=X_M\in\cu_{Q,F}^{+},
	\]
	where $R^{\mathrm a}=\Hom_{\cc^{\mathrm a}}(T^{\mathrm a},\Si-)$.
	
	We first show that $u^*$ is full on $\ch$. Let $M,N\in\ch$, and
	denote by $[T]$ and $[T^{\mathrm a}]$ the ideals of morphisms
	factoring through $\add T$ and $\add T^{\mathrm a}$, respectively.
	Applying Proposition~\ref{Prop: modules}\textup{(1)} in both the
	relative and the ordinary cluster categories shows that $u^*$ induces
	an isomorphism
	\[
	\frac{\Hom_{\ch}(M,N)}{[T](M,N)}
	\xrightarrow{\sim}
	\frac{\Hom_{\cc^{\mathrm a}}(M^{\mathrm a},N^{\mathrm a})}
	{[T^{\mathrm a}](M^{\mathrm a},N^{\mathrm a})},
	\]
	since both sides identify with module homomorphisms $RM\to RN$, and
	the $J(Q,F,W)$-actions on $RM$ and $RN$ factor through the canonical
	surjection $J(Q,F,W)\twoheadrightarrow J(Q,W)$. It remains to lift
	morphisms in $[T^{\mathrm a}](M^{\mathrm a},N^{\mathrm a})$. Choose
	an $\add T$-copresentation $M\to U_0\to U_1\to\Si M$ and an
	$\add T$-presentation $V_1\to V_0\to N\to\Si V_1$. By rigidity of
	$T^{\mathrm a}$, every morphism in
	$[T^{\mathrm a}](M^{\mathrm a},N^{\mathrm a})$ factors as
	\[
	M^{\mathrm a}\longrightarrow U_0^{\mathrm a}
	\xrightarrow{h} V_0^{\mathrm a}\longrightarrow N^{\mathrm a}.
	\]
	Because $u^*$ induces the surjection
	$J(Q,F,W)\twoheadrightarrow J(Q,W)$ on the endomorphism algebras of
	$T$ and $T^{\mathrm a}$, the morphism $h$ lifts to a morphism
	$U_0\to V_0$. This shows that $u^*$ is full on $\ch$.
	
	Since $\per\bmgamma$ is Krull--Schmidt,
	Proposition~\ref{Prop:fully faithful of pr} and
	Lemma~\ref{Lem: Higgs is idempotent complete} imply that $\ch$ is
	Krull--Schmidt as well. For any indecomposable object $M\in\ch$, we
	have $M^{\mathrm a}\ne0$: when $RM\ne0$, this follows from
	$R^{\mathrm a}M^{\mathrm a}\cong RM$, whereas when $RM=0$,
	Proposition~\ref{Prop: modules}\textup{(1)} implies that $1_M$
	factors through $\add T$, so $M\simeq T_i$ for some $i$ and hence
	$M^{\mathrm a}\simeq T_i^{\mathrm a}\ne0$. Consequently,
	$\operatorname{End}(M^{\mathrm a})$ is a nonzero quotient of the
	local ring $\operatorname{End}(M)$ and is therefore local. In
	particular, $u^*$ preserves indecomposability and reflects
	isomorphisms between indecomposable objects. Indeed, given an
	isomorphism $M^{\mathrm a}\xrightarrow{\sim}N^{\mathrm a}$ with
	$M,N\in\ch$ indecomposable, fullness lets us lift it and its inverse
	to morphisms $f\colon M\to N$ and $g\colon N\to M$; then $1_M-gf$ and
	$1_N-fg$ lie in the kernels of the induced homomorphisms of local
	endomorphism rings---hence in their maximal ideals---so both $gf$ and
	$fg$ are invertible.
	
	We next compare mutation at a non-projective summand in $\ch$ with
	ordinary mutation at the corresponding vertex. By
	Theorem~\ref{Thm: main results}, $\ch$ is Frobenius and its stable
	category $\underline{\ch}$ is triangle equivalent to
	$\cc(\overline Q,\overline W)$, with $T$ mapping to $\bar\Gamma$.
	Since restriction to the full subquiver on the unfrozen vertices
	commutes with QP mutation at those vertices,
	$(\overline Q,\overline W)$ is non-degenerate, and every sequence of
	mutations at unfrozen vertices can be carried out in
	$\underline{\ch}$. 
	
	Any basic cluster-tilting object in
	$\underline{\ch}$ lifts to $\ch$ by taking the indecomposable
	non-projective representatives of its summands and adjoining all
	frozen summands $T_f$. And extensions in $\ch$ are computed in
	$\underline{\ch}$, while approximations in $\underline{\ch}$ lift to
	$\ch$ by adding a projective deflation or an injective inflation. By
	the Krull--Schmidt property, this lift is unique up to isomorphism.
	
	Suppose inductively that
	$V=V_k\oplus V_0$ is a reachable basic cluster-tilting object
	in $\mathcal H$ and that $V^a=u^*V$ is the reachable rigid
	object in Plamondon's subcategory $\mathcal D(Q,W)$ obtained
	from $T^a$ by the same sequence of mutations. Lifting the stable exchange triangle at $k$ gives an
	exchange conflation
	\[
	V_k\xrightarrow{a}E\longrightarrow V_k^*
	\]
	in $\ch$, where $V_k^*$ is the indecomposable non-projective
	representative of the new summand. Because the middle term in
	$\underline{\ch}$ lies in $\add\underline{V_0}$ and $V_0$ contains
	all indecomposable projective-injective objects, we have
	$E\in\add V_0$. By rigidity of $V_k^*\oplus V_0$, the morphism $a$ is
	a left $\add V_0$-approximation, and by fullness of $u^*$, its image
	$u^*a$ is a left $\add V_0^{\mathrm a}$-approximation. Let
	\[
	V_k^{\mathrm a}\xrightarrow{b}D\longrightarrow W_k
	\longrightarrow\Si V_k^{\mathrm a}
	\]
	be the ordinary exchange triangle at $k$
	\cite[Section~2.6]{plamondonCategoriesAmasseesAux2011}. By rigidity
	of $W_k\oplus V_0^{\mathrm a}$, the morphism $b$ is also a left
	$\add V_0^{\mathrm a}$-approximation, so $u^*a$ and $b$ factor
	through each other. Consequently, the morphism
	$(u^*a,b)^t\colon V_k^{\mathrm a}\to E^{\mathrm a}\oplus D$ is
	isomorphic to both $(u^*a,0)^t$ and $(0,b)^t$. By comparing their cones, we obtain
	\[
	(V_k^*)^{\mathrm a}\oplus D \simeq W_k\oplus E^{\mathrm a}.
	\]

	All objects in this comparison belong to
	$\operatorname{pr}_{\mathcal C^a}\Gamma^a$,
	which is Krull--Schmidt by the ordinary version of \cite[Corollary~ 2.12]{plamondonCategoriesAmasseesAux2011}.
	Since $D,E^a\in\operatorname{add}V_0^a$ and $W_k$ is not
	a summand of $V_0^a$, the indecomposability of $(V_k^*)^a$
	and Krull--Schmidt uniqueness imply
	$(V_k^*)^a\simeq W_k$. Starting
	from $T$, induction therefore identifies the reachable summands on
	both sides along every sequence of mutations at unfrozen vertices.
	
	By~\cite[Theorem~2.6 and Corollary~3.5]
	{CerulliKellerLabardiniPlamondon2013}, the ordinary cluster character
	$CC^{\mathrm a}$ sends each reachable indecomposable summand to the
	cluster variable in the same position and distinguishes
	non-isomorphic reachable summands. Restricting to mutations at
	unfrozen vertices produces the mutable variables of $\ca_{Q,F}$,
	while $CC^{\mathrm a}(T_f^{\mathrm a})=x_f$ at each frozen vertex.
	Since $X_M=CC^{\mathrm a}(M^{\mathrm a})$ and $u^*$ reflects
	isomorphisms between indecomposable objects, this establishes the
	displayed bijection; applying it to all summands at each seed gives
	the assertion about extended clusters.
\end{proof}

\subsection{Commutative diagram}\label{subsection: Comm}
Recall that $\overGamma$ is the Ginzburg algebra associated with
$(\overline Q,\overline W)$ and that $\overc=\cc(\overline Q,\overline W)$
is the corresponding cluster category. Since $(\overline Q,\overline W)$
is Jacobi-finite, Theorem~\ref{Thm: main results} gives
$\cd(\overline Q,\overline W)=\overc$. The restriction of $p^*$ to
$\ch$ induces a triangle equivalence $\underline{\ch}\xrightarrow{\sim}\overc$.

For $X\in\cd(\overline Q,\overline W)\subseteq\pr_{\overc}\overGamma$,
the index with respect to $\overGamma$ is given by
\[
\ind_{\overGamma}X=[U_0^X]-[U_1^X]\in K_0(\add\overGamma),
\]
where $U_1^X\ra U_0^X\ra X\ra\Si U_1^X$ is an $\add\overGamma$-presentation.
As in Subsection~\ref{Subsection: index}, the index is independent of
the choice of presentation.

Write $\overline R=\Hom_{\overc}(\overGamma,\Si-)$.
If $\mathbf d=\underline{\dim}\,\overline R X$, put
\[
\overline l(\mathbf d)=\ind_{\overGamma}X+\ind_{\overGamma}\Si X.
\]
By~\cite[Lemma~3.6]{plamondonCategoriesAmasseesAux2011}, this sum depends only
on $\mathbf d$. Plamondon's cluster character~\cite[Theorem~3.12]{plamondonCategoriesAmasseesAux2011}
takes each $\overGamma_i$ to $x_i$ and is given by
\[
\overline{CC}(X)
=x^{\ind_{\overGamma}X}
\sum_{\mathbf d\in\mathbb N^r}
\chi\!\left(\Gr_{\mathbf d}(\overline R X)\right)x^{-\overline l(\mathbf d)}.
\]

Let $\cs$ be the collection of simple modules $S_i$ associated with
the unfrozen vertices of $Q$. As in Subsection~\ref{Section: SMC}, put
\[
\cw=(\Si^{\geqslant0}\cs)^\perp\cap{}^\perp(\Si^{\leqslant0}\cs)
\subseteq\per\relGammabf .
\]
By Proposition~\ref{Prop: stable finite to dense}, the composition
$\cw\hookrightarrow\per\relGammabf\xrightarrow{\pi^{rel}}\cc(Q,F,W)$
induces a $k$-linear equivalence
$\pi^{rel}_{\cw}\colon\cw\xrightarrow{\sim}\cc(Q,F,W)$.

\begin{Rem}
	When the frozen part $ F $ is empty, the category $ \cw $ is equal to 
	$$ \Si\cf=\Si\add\,\bmgamma(Q,W)*\Si^{2}\add\,\bmgamma(Q,W) ,$$
	i.e.\,the shift of the fundamental domain, see~\cite[Theorem 2.12]{iyama-yang 2020}.
\end{Rem}

\begin{Def}\rm\label{Def: index [P_X]}
	For $X\in\cc(Q,F,W)$, choose $X'\in\cw$ with
	$\pi^{rel}_{\cw}(X')\cong\Si X$. Under the identification
	$K_0(\per\bmgamma)\cong K_0(\add\bmgamma)$, let
	$\operatorname{pr}_F$ be the projection onto $K_0(\add(e\bmgamma))$
	along $K_0(\add((1-e)\bmgamma))$. Put
	\[
	[X]^F=-\operatorname{pr}_F[X'].
	\]
	This is independent of the choice of $X'$ since $\pi^{rel}_{\cw}$
	is an equivalence. For $X\in\pr_{\cc}\bmgamma$, write
	$\ind_{\bmgamma}^{F}(X)=\operatorname{pr}_F(\ind_{\bmgamma}X)$.
\end{Def}

\begin{Rem}\label{Rem:frozen index with per index}
	Let $X\in\ch$. By Proposition~\ref{Prop:fully faithful of pr},
	choose a lift $\widetilde X\in\pr_{\cd}\bmgamma$ of $X$ and a
	presentation triangle in $\per\bmgamma$
	\[
	T_1\longrightarrow T_0\longrightarrow\widetilde X
	\longrightarrow\Si T_1,
	\qquad T_0,T_1\in\add\bmgamma.
	\]
	We claim that $\Si\widetilde X\in\cw$. For each unfrozen vertex $i$,
	the simple module $S_i$ is concentrated in degree zero. Thus
	$\Hom_{\cd}(T_a,\Si^q S_i)=0$ for $a=0,1$ and $q\ne0$.
	The presentation triangle then gives
	\[
	\Hom_{\cd}(\widetilde X,\Si^q S_i)=0
	\qquad\text{for }q\notin\{0,1\}.
	\]
	Hence $\Hom_{\cd}(\Si\widetilde X,\Si^m S_i)=0$ for $m\le0$.
	For $m\ge0$, relative $3$-Calabi--Yau duality
	(Proposition~\ref{Prop:relative 3-CY}) gives
	\[
	\Hom_{\cd}(\Si^m S_i,\Si\widetilde X)
	\cong D\Hom_{\cd}(\widetilde X,\Si^{m+2}S_i)=0.
	\]
	These are the two orthogonality conditions defining $\cw$.
	
	Since $\pi^{\mathrm{rel}}(\Si\widetilde X)\cong\Si X$, we may use
	$\Si\widetilde X$ in Definition~\ref{Def: index [P_X]}.
	Under the identification
	$K_0(\per\bmgamma)\cong K_0(\add\bmgamma)$, we obtain
	\[
	[X]^F
	=\operatorname{pr}_F(-[\Si\widetilde X])
	=\operatorname{pr}_F([T_0]-[T_1])
	=\ind_{\bmgamma}^{F}(X).
	\]
	
	Applying $p^*$ to the image of the presentation triangle in $\cc$
	gives an $\add\overline{\bmgamma}$-presentation of $p^*X$.
	The functor $p^*$ kills each frozen summand $\bmgamma_f$ and sends
	each unfrozen summand $\bmgamma_i$ to $\overline{\bmgamma}_i$.
	Thus, writing $\iota$ for the inclusion of the unfrozen summand
	given by $[\overline{\bmgamma}_i]\mapsto[\bmgamma_i]$, we have
	\[
	\ind_{\bmgamma}(X)
	=[X]^F+\iota\!\left(\ind_{\overline{\bmgamma}}(p^*X)\right)
	\qquad\text{in }K_0(\add\bmgamma).
	\]
\end{Rem}

\begin{Thm}\label{Thm: commutative diagram}
	Suppose that $\cp=\add(e\bmgamma)$ is functorially finite in $\add(\bmgamma)$
	and that $(\overline Q,\overline W)$ is Jacobi-finite. Consider the following diagram on isomorphism
	classes of objects, where $q$ is
	the stable quotient functor, $\underline{p^*}$ is the equivalence
	induced by $p^*$, and $\overline{CC}$ is Plamondon's cluster
	character~\cite{plamondonCategoriesAmasseesAux2011}.
	\[
	\begin{tikzcd}
		\ch \arrow[r,hook] \arrow[d,"q"']
		& \cc(Q,F,W) \arrow[r,dashed,"CC_{loc}"] \arrow[d,"p^*"']
		& \mathbb Q[x_1^{\pm1},\ldots,x_n^{\pm1}]
		\arrow[d,"{x_f\mapsto1\ (f>r)}"] \\
		\underline{\ch} \arrow[r,"{\underline{p^*}}"',"\sim"]
		& \cc(\overline Q,\overline W) \arrow[r,"\overline{CC}"']
		& \mathbb Q[x_1^{\pm1},\ldots,x_r^{\pm1}].
	\end{tikzcd}
	\]
	There is a unique map
	\[
	CC_{loc}\colon\cc(Q,F,W)\ra\mathbb Q[x_1^{\pm1},\ldots,x_n^{\pm1}]
	\]
	such that the diagram commutes and the following conditions hold.
	\begin{itemize}
		\item[1)] For each triangle in $ \cc(Q,F,W) $
		$$ P\ra X\ra M\ra\Si P $$ with $ P\in\thick_{\cc}(\cp) $,
		we have 
		$$ CC_{loc}(X)=CC_{loc}(P)\cdot CC_{loc}(M). $$
		\item[2)] The restriction $ CC_{loc}|_{\ch} $ is the cluster character $ X_{?} $ defined 
		in~\eqref{CC formula for Higgs}.
		\item[3)] For each $P\in\thick_{\cc}(\cp)$, let $\widetilde P$ be
		its lift to $\thick_{\per\bmgamma}(e\bmgamma)$ under
		$\pi^{\mathrm{rel}}$. Then $CC_{loc}(P)=x^{[\widetilde P]}$,
		where $[\widetilde P]\in K_0(\per\bmgamma)\simeq\mathbb Z^n$.
	\end{itemize}
\end{Thm}
\begin{proof}
	Let us first show that $CC_{loc}$ is unique if it exists. Let $n>0$
	and $M\in\ch$. Since $\ch$ is Frobenius extriangulated, we may choose
	triangles in $\cc(Q,F,W)$
	\[
	\Omega^i M\ra P_i\ra\Omega^{i-1}M\ra\Si\Omega^i M,
	\qquad 1\leq i\leq n,
	\]
	where $\Omega^0M=M$, $P_i\in\cp$ and $\Omega^iM\in\ch$. Rotating
	and shifting these triangles, and applying~$1)$ successively, gives
	\[
	CC_{loc}(\Si^{-n}M)
	=\left(\prod_{i=1}^n CC_{loc}(\Si^{i-1-n}P_i)\right)
	CC_{loc}(\Omega^n M).
	\]
	Dually, choose triangles
	\[
	\Theta^{i-1}M\ra I_i\ra\Theta^iM\ra\Si\Theta^{i-1}M,
	\qquad 1\leq i\leq n,
	\]
	where $\Theta^0M=M$, $I_i\in\cp$ and $\Theta^iM\in\ch$.
	By~$3)$, $CC_{loc}(\Si^q I_i)$ is a unit with inverse
	$CC_{loc}(\Si^{q+1}I_i)$. Thus~$1)$ applied to the rotated and
	shifted triangles gives
	\[
	CC_{loc}(\Si^nM)
	=\left(\prod_{i=1}^n CC_{loc}(\Si^{n+1-i}I_i)\right)
	CC_{loc}(\Theta^n M).
	\]
	The right-hand sides are determined by~$2)$ and~$3)$, so these
	formulas determine $CC_{loc}(\Si^mM)$ for every $m\in\mathbb Z$.
	For an arbitrary $X\in\cc(Q,F,W)$,
	Theorem~\ref{Thm: Jacobi-finite to Frobenius} gives a triangle
	\[
	P\ra X\ra\Si^mM\ra\Si P
	\]
	with $P\in\thick_{\cc}(\cp)$ and $M\in\ch$. Condition~$1)$ now
	determines $CC_{loc}(X)=CC_{loc}(P)\,CC_{loc}(\Si^mM)$.
	This proves uniqueness.
	
	For existence, define
	\[
	CC_{loc}(X)=x^{[X]^F}\,\widetilde{CC}(p^*X),
	\qquad X\in\cc(Q,F,W),
	\]
	where $[X]^F$ is defined in Definition~\ref{Def: index [P_X]}.
	For $M\in\ch$, Remark~\ref{Rem:frozen index with per index} gives
	$[M]^F=(\ind_{\bmgamma}M)^F$. Hence
	Remark~\ref{rem:relative-character-frozen-coefficients} gives
	\[
	CC_{loc}(M)
	=x^{(\ind_{\bmgamma}M)^F}\,\widetilde{CC}(p^*M)=X_M.
	\]
	This proves~$2)$. Moreover, setting the frozen variables equal to
	one sends $x^{[X]^F}$ to $1$ and $\widetilde{CC}(p^*X)$ to
	$\overline{CC}(p^*X)$, since
	$\pi_{\mathrm{mut}}l(\mathbf d)=\bar l(\mathbf d)$.
	This proves commutativity of the diagram.
	
	Write $\pi=\pi^{\mathrm{rel}}$. For $P\in\thick_{\cc}(\cp)$,
	choose its lift $\widetilde P\in\thick_{\per\bmgamma}(e\bmgamma)$
	under $\pi$. The functor $\pi$ identifies these two thick subcategories
	because the frozen projectives are left orthogonal to $\pvd_e(\bmgamma)$,
	and their thick closure in $\per\bmgamma$ is idempotent complete.
	Every shift of $\widetilde P$ is orthogonal on both sides to the
	unfrozen simples. For a frozen projective, one direction follows
	from $S_i e=0$ and the other from relative $3$-Calabi--Yau duality
	(Proposition~\ref{Prop:relative 3-CY}). These vanishings extend to
	the thick closure. In particular, $\Si\widetilde P\in\cw$.
	The frozen projectives form a silting generator of
	$\thick_{\per\bmgamma}(e\bmgamma)$, so its Grothendieck group
	is generated by the $[e_f\bmgamma]$, $f\in F_0$.
	Thus $[\widetilde P]$ has only frozen coordinates, and Definition~\ref{Def: index [P_X]} gives
	\[
	[P]^F=-\operatorname{pr}_F[\Si\widetilde P]=[\widetilde P].
	\]
	As $p^*P=0$ and $\widetilde{CC}(0)=1$, this proves~$3)$.
	
	Finally, let $P\xrightarrow{a}X\ra M\ra\Si P$ be a triangle
	in $\cc(Q,F,W)$ with $P\in\thick_{\cc}(\cp)$. Choose $X'\in\cw$
	with $\pi X'\cong\Si X$. By full faithfulness of $\pi|_{\cw}$,
	the morphism $\Si a$ lifts, under the chosen identifications, to
	$a'\colon\Si\widetilde P\to X'$. Complete it to a triangle in
	$\per\bmgamma$
	\[
	\Si\widetilde P\xrightarrow{a'}X'\ra M'\ra\Si^2\widetilde P.
	\]
	Its image under $\pi$ gives $\pi M'\cong\Si M$. Moreover, $M'\in\cw$.
	Applying the two Hom functors defining $\cw$ to this triangle gives
	the required vanishings, since $X'\in\cw$ and all shifts of
	$\widetilde P$ are orthogonal on both sides to the unfrozen simples.
	Thus Definition~\ref{Def: index [P_X]} and the triangle relation give
	\[
	[X]^F=-\operatorname{pr}_F[X']
	=[\widetilde P]-\operatorname{pr}_F[M']
	=[P]^F+[M]^F.
	\]
	Since $p^*P=0$, we also have $p^*X\cong p^*M$. Consequently,
	\[
	CC_{loc}(X)
	=x^{[P]^F}x^{[M]^F}\,\widetilde{CC}(p^*M)
	=CC_{loc}(P)\,CC_{loc}(M).
	\]
	This proves~$1)$ and completes the proof.
\end{proof}

\section{Applications to quasi-cluster homomorphisms}\label{Section: Applications to quasi-cluster homomorphisms}

In this section, we first recall the definition of quasi-cluster homomorphism defined by Fraser in~\cite{fraserQuasihomomorphismsClusterAlgebras2016a}. Then our aim is to show that the 
decategorification of the equivalence associated with the mutation at a frozen source (or sink) is a quasi-cluster isomorphism.

\begin{Def}\rm\cite[Definition 1.2]{fraserQuasihomomorphismsClusterAlgebras2016a}
	A \emph{labeled r-regular tree}, $ \mathbb{T}_{r} $, is an $ r $-regular tree with edges labeled by integers so that the set of labels emanating from each vertex is $ [1,r]=\{1,2,\ldots,r\} $. We write $ t\xrightarrow{k}t' $ to indicate that vertices $ t $, $ t' $ are joined by an edge with label $ k $. An isomorphism $ \mathbb{T}_{r}\to\overline{\mathbb{T}}_{r} $ of labeled trees sends vertices to vertices and edges to edges, preserving incidences of edges and the edge labels. Such an isomorphism is uniquely determined by its value at a single vertex $ t\in\mathbb{T}_{r} $.
\end{Def}
Let $(Q,F)$ be a finite ice quiver with
$Q_0=\{1,\ldots,n\}$ and $F_0=\{r+1,\ldots,n\}$. Its signed
extended exchange matrix is the $n\times r$ matrix
$\widetilde B(Q,F)=(b_{ij})$ given by
\[
b_{ij}=\#\{i\to j\text{ in }Q\}-\#\{j\to i\text{ in }Q\},
\qquad 1\leq i\leq n,\quad 1\leq j\leq r .
\]
Its principal part is denoted by $B$. Opposite arrows cancel in
each matrix entry, and arrows between frozen vertices do not
contribute to the matrix. We denote by $\mathbb P$ the tropical
semifield $\mathrm{Trop}(x_{r+1},\ldots,x_n)$ and by $\cf$ the
field of fractions of $\mathbb Q\mathbb P[x_1,\ldots,x_r]$.

\begin{Def}\rm
	A \emph{seed} of geometric type is a pair
	$(\widetilde B,\mathbf{x})$, where $\widetilde B$ is an extended
	exchange matrix with skew-symmetric principal part and
	$\mathbf{x}=(x_1,\ldots,x_n)$ is the extended cluster. Its first
	$r$ entries freely generate $\cf$ over the field of fractions of
	$\mathbb Q\mathbb P$, and its last $n-r$ entries are the fixed
	frozen variables. The seed associated with $(Q,F)$ is
	$(\widetilde B(Q,F),\mathbf{x})$.
\end{Def}

Given $1\leq k\leq r$, the \emph{mutation} of the seed
$(\widetilde B,\mathbf{x})$ in direction $k$ is the pair
$(\mu_k(\widetilde B),\mathbf{x}')$, where $\mu_k$ denotes matrix
mutation, $x'_i=x_i$ for $i\ne k$, and $x'_k$ is given by the
\emph{exchange relation}
\[
x_k x'_k
=\prod_{i=1}^{n}x_i^{[b_{ik}]_+}
+\prod_{i=1}^{n}x_i^{[-b_{ik}]_+},
\qquad [a]_+=\max(a,0).
\]
See~\cite[(2.15)]{FominZelevinskyIV}. If no oriented $2$-cycle or
loop is incident with $k$, the two products are the products over
arrows entering and leaving $k$, respectively. Mutation in a fixed
direction is an involution.

Let $\mathbb T_r$ be the $r$-regular labeled tree with root $t_0$.
\begin{Def}\rm\cite[Definition 1.3]{fraserQuasihomomorphismsClusterAlgebras2016a}
	A collection of seeds in $\cf$, with one seed
	$\Si(t)=(\widetilde B(t),\mathbf{x}(t))$ for each $t\in\mathbb T_r$,
	is called a \emph{seed pattern} if, for each edge
	$t\xrightarrow{k}t'$, the seeds $\Si(t)$ and $\Si(t')$ are related
	by mutation in direction $k$. The seed $\Si(t_0)$ is called the
	\emph{initial seed} and the seed pattern is denoted by $\ce$.
\end{Def}

Fix the initial seed $(\widetilde B(Q,F),\mathbf{x})$, with the
unfrozen vertices indexed by $1,\ldots,r$ and the frozen vertices
by $r+1,\ldots,n$. The sets $\{x_1(t),\ldots,x_n(t)\}$ are the
\emph{clusters} and
their elements are the \emph{cluster variables}. The variables $x_1(t),\ldots,x_r(t)$ are mutable,
and mutation is allowed only in these $r$ directions. The remaining
variables are the frozen variables (coefficients), so $x_j(t)=x_j$
for every seed $t$ and every $r<j\leq n$. The corresponding
\emph{cluster algebra with invertible coefficients} $\ca_Q$ is the
$\mathbb Z\mathbb P$-subalgebra of $\cf$ generated by all cluster
variables. As in the previous section, let $\overline Q$ denote the
full subquiver on the non-frozen vertices. We denote by
$\ca_{\overline Q}$ the coefficient-free cluster algebra with
initial exchange matrix $B$. Specializing the frozen variables to
$1$ sends each cluster variable to its coefficient-free counterpart.

For any $t\in\mathbb T_r$, write $\widetilde B(t)=(b_{ij}(t))$.
The associated $r$-tuple of hatted variables is given by
\[
\widehat y_j(t)=\prod_{i=1}^{n}x_i(t)^{b_{ij}(t)},
\qquad 1\leq j\leq r,
\]
where $x_i(t)=x_i$ for $i>r$.

Now let $(Q',F')$ be another ice quiver with the same sets of
vertices and frozen vertices. Its initial seed is
$(\widetilde B(Q',F'),\mathbf{x}')$. The corresponding seed
pattern $\ce'$ is built on a second copy $\mathbb T'_r$ of the
$r$-regular tree. In particular, when $(Q',F',W')=\mu_v(Q,F,W)$
is the reduced ice mutation at a frozen source or sink $v$, we
use $\widetilde B(Q',F')$ to define this second initial seed.
Only the directions $1,\ldots,r$ are used for seed mutation
within either pattern.


\begin{Def}\rm\cite[Definition 3.1]{fraserQuasihomomorphismsClusterAlgebras2016a}
	\label{def:quasi-cluster homomorphism}
	Let $\ca_Q$ and $\ca_{Q'}$ be taken over $\mathbb Q$, with
	invertible frozen variables. A \emph{quasi-cluster homomorphism} is a
	$\mathbb Q$-algebra morphism $\varphi:\ca_Q\to\ca_{Q'}$ satisfying
	the following conditions.
	\begin{itemize}
		\item[a)] We have $\varphi(\mathbb P)\subseteq\mathbb P'$.
		For each mutable cluster variable $x$, there is a mutable
		cluster variable $x'$ of $\ca_{Q'}$ with
		$\varphi(x)\in\mathbb P'\cdot x'$.
		\item[b)] There is a cluster isomorphism
		$\overline\varphi:\ca_{\overline Q}\to\ca_{\overline{Q'}}$
		satisfying $\overline\varphi\circ\mathrm{sp}
		=\mathrm{sp}'\circ\varphi$. It takes the initial coefficient-free
		seed to a seed with the same principal exchange matrix after
		relabelling.
		\item[c)] There is a target seed $t'$ specializing to the image
		seed in b), with a labelling $j\mapsto j'$ for which
		$\overline\varphi(\bar x_j(t_0))=\bar x'_{j'}(t')$ and
		$\varphi(\widehat y_j(t_0))=\widehat y'_{j'}(t')$ for all
		mutable $j$.
	\end{itemize}
	Here $\mathrm{sp}$ and $\mathrm{sp}'$ denote specialization of the
	frozen variables to one. In c), we evaluate the initial Laurent
	monomial by substitution in $\operatorname{Frac}(\ca_{Q'})$.
	This is well defined by a).
\end{Def}

\begin{Rem}
	In~c), we use the fact that each cluster determines its seed up
	to a simultaneous relabeling, cf.~\cite{CerulliKellerLabardiniPlamondon2013}.
\end{Rem}

Let $(Q,F,W)$ be an ice quiver with irredundant potential such that
$\cp=\add(e\relGammabf)$ is functorially finite in $\add(\relGammabf)$.
We assume that $(Q,W)$, regarded as an ordinary quiver with potential
on all vertices of $Q$, is non-degenerate in the sense of
Derksen--Weyman--Zelevinsky, and that $(\overline Q,\overline W)$
is Jacobi-finite.

For the frozen mutations below, write
$\mathcal L_Q=\mathbb Q[x_1^{\pm1},\ldots,x_n^{\pm1}]$
and use the analogous notation $\mathcal L_{Q'}$ on the primed side.
We write $\Psi_\pm$ also for the functors induced on the relative
cluster categories. By the comparison under vertex deletion in
Proposition~\ref{Rem: mutation at frozen} and its dual, the induced
equivalence $\per(\overGamma')\xrightarrow{\sim}\per(\overGamma)$
sends each indecomposable projective summand $e_i\overGamma'$ to
$e_i\overGamma$, and hence sends $\overGamma'$ to $\overGamma$.
Their endomorphism algebras $H^0(\overGamma')$ and $H^0(\overGamma)$
are therefore isomorphic, so $(\overline{Q'},\overline{W'})$ is
Jacobi-finite as well.

Since $(\overline Q',\overline W')$ is Jacobi-finite,  by
Proposition~\ref{prop:reduced-finite-implies-approximations}, then
$\operatorname{add}(e'\Gamma')$ is functorially finite
in $\operatorname{add}(\Gamma')$.
Thus Theorem~\ref{Thm: commutative diagram} applies on the primed side as well.

\begin{Prop}\label{Prop: decategorification of frozen source}\label{Def: qusi-cluster functor}
	Let $v\in F_0$ be a frozen source and $(Q',F',W')=\mu_v(Q,F,W)$.
	Then there exists a unique $\mathbb Q$-algebra isomorphism
	$\psi_+\colon\mathcal L_{Q'}\ra\mathcal L_Q$ such that the
	following diagram commutes on isomorphism classes
	\begin{equation}\label{diagram: frozen--cc maps}
		\begin{tikzcd}
			\cc(Q',F',W')\arrow[d,"\Psi_+"]\arrow[r,"CC'_{\mathrm{loc}}"]
			&\mathcal L_{Q'}\arrow[d,"\psi_+"]\\
			\cc(Q,F,W)\arrow[r,"CC_{\mathrm{loc}}"']&\mathcal L_Q.
		\end{tikzcd}
	\end{equation}
\end{Prop}
\begin{proof}
	Write $\Gamma=\bmgamma$, $\Gamma'=\bmgamma'$ and $E=\Psi_+$. Put
	\[
	A=K_0(E):K_0(\per\Gamma')\longrightarrow K_0(\per\Gamma).
	\]
	The frozen-source mutation equivalence sends $\Gamma'_j$ to
	$\Gamma_j$ for $j\ne v$ and sends $\Gamma'_v$ to
	\[
	\operatorname{Cone}\!\left(
	\Gamma_v\longrightarrow
	\bigoplus_{\substack{\alpha\in F_1\\s(\alpha)=v}}
	\Gamma_{t(\alpha)}\right)
	\]
	by \cite[Theorem~6.8]{wuCategorificationIceQuiver2021} for
	premutation and \cite[Lemma~4.28]{wuCategorificationIceQuiver2021}
	after reduction. Consequently $A$ fixes the unfrozen basis vectors
	and preserves the frozen sublattice. The monomial isomorphism
	$\psi_+(x'^{\,a})=x^{Aa}$ is therefore given by
	\[
	\psi_+(x'_j)=
	\begin{cases}
		x_j, & j\ne v,\\[2pt]
		x_v^{-1}\displaystyle\prod_{\substack{\alpha\in F_1\\s(\alpha)=v}}
		x_{t(\alpha)}, & j=v .
	\end{cases}
	\]
	
	Let $i$ be an unfrozen vertex. For every vertex $j$ and every
	$m\in\mathbb Z$, we have
	\[
	H^m(E^{-1}S_i)e'_j
	\cong \Hom_{\mathcal D(\Gamma)}(E\Gamma'_j,\Sigma^mS_i).
	\]
	For $j=v$ the right-hand side vanishes since both terms of the
	displayed cone are frozen projectives. For $j\ne v$ it is $k$ if
	$(j,m)=(i,0)$ and zero otherwise. Thus $E^{-1}S_i\simeq S'_i$.
	The orthogonality conditions defining $\mathcal W'$ and $\mathcal W$
	now give $E(\mathcal W')=\mathcal W$. For $X\in\cc(Q',F',W')$, choose
	$Y'\in\mathcal W'$ lifting $\Sigma X$. Then $EY'\in\mathcal W$ lifts
	$\Sigma EX$. Since $A$ commutes with the frozen projections, by
	Definition~\ref{Def: index [P_X]}, we have
	\[
	[EX]^F=-\operatorname{pr}_F[EY']=A[X]^{F'}.
	\]
	
	We next compare the exponents inside the Grassmannian sums.
	For $\mathbf d=(d_1,\ldots,d_r)\in\mathbb N^r$, the simple-class
	formula for $l$ gives
	\[
	l(\mathbf d)=\sum_{i=1}^r d_i[S_i],
	\qquad l'(\mathbf d)=\sum_{i=1}^r d_i[S'_i].
	\]
	Since $ES'_i\simeq S_i$ for every unfrozen $i$, we conclude that
	\[
	A\,l'(\mathbf d)=l(\mathbf d).
	\]
	
	By Proposition~\ref{Rem: mutation at frozen}, deletion gives
	an equivalence $\bar E$ of the absolute cluster categories
	satisfying $\bar E p'^*\simeq p^*E$.
	For each unfrozen $i$, this equivalence sends $\bar\Gamma'_i$
	to $\bar\Gamma_i$. Thus, for
	\[
	N'=p'^*X,\qquad N=p^*EX,
	\]
	the absolute indices agree in the vertex bases. Moreover, the
	induced isomorphism of endomorphism algebras identifies
	$\bar R'N'$ with $\bar R N$ and preserves the vertex idempotents.
	The corresponding quiver Grassmannians are therefore isomorphic
	for every dimension vector $\mathbf d$.
	
	Let $\iota:\mathbb Z^r\to\mathbb Z^n$ extend by zero on the
	frozen vertices. Since $A\iota=\iota$, application of $\psi_+$
	term by term gives
	\[
	\begin{aligned}
		\psi_+\!\left(\widetilde{CC}'(N')\right)
		&=x^{\iota(\ind_{\bar\Gamma}N)}
		\sum_{\mathbf d}
		\chi\!\left(\Gr_{\mathbf d}(\bar R N)\right)
		x^{-A l'(\mathbf d)}\\
		&=\widetilde{CC}(N).
	\end{aligned}
	\]
	Combining this identity with $A[X]^{F'}=[EX]^F$ and the formula
	of Theorem~\ref{Thm: commutative diagram}, we obtain
	\[
	\psi_+\!\left(CC'_{\mathrm{loc}}(X)\right)
	=x^{[EX]^F}\widetilde{CC}(p^*EX)
	=CC_{\mathrm{loc}}(EX).
	\]
	This proves commutativity of~\eqref{diagram: frozen--cc maps}.
	
	Finally, $CC'_{\mathrm{loc}}(\Gamma'_j)=x'_j$ for every vertex $j$.
	Evaluating the diagram at the $\Gamma'_j$ therefore determines the
	image of every Laurent generator, proving uniqueness of $\psi_+$.
\end{proof}

\begin{Rem}\label{rem:frozen-source-quasi-cluster}
	Keep the assumptions of
	Proposition~\ref{Prop: decategorification of frozen source},
	and take the frozen variables to be invertible. Then $\psi_+$
	restricts to a quasi-cluster isomorphism
	$\ca_{Q'}\rightarrow\ca_Q$ in the sense of
	Definition~\ref{def:quasi-cluster homomorphism}.
	
	Indeed, let $A$ be the matrix of $K_0(\Psi_+)$ in the bases
	$([\Gamma'_i])$ and $([\Gamma_i])$. For each unfrozen vertex $j$,
	the standard resolution of the simple module $S_j$
	\cite[Section~4.6]{wuCategorificationIceQuiver2021} gives
	\[
	[S_j]=\sum_{a\colon j\to i}[\Gamma_i]-\sum_{a\colon i\to j}[\Gamma_i]
	=-\widetilde B_{\bullet j}.
	\]
	
	The same formula holds on the primed side, since opposite arrows
	contribute with opposite signs. As $\Psi_+S'_j\simeq S_j$, we have
	$A\widetilde B'=\widetilde B$. Since $A$ fixes the unfrozen basis
	vectors and preserves the frozen sublattice, projection to the
	mutable rows gives $B'=B$. Moreover,
	$\psi_+(\widehat y'_j)=x^{A\widetilde B'_{\bullet j}}=\widehat y_j$.
	
	Fix a mutable vertex $j$ and a seed $t$ in the common labelled
	mutation tree. Write $g_{j;t}$ and $F_{j;t}$ for the $g$-vector
	and $F$-polynomial determined by the common initial principal part.
	Write $y_j=\prod_{i=r+1}^n x_i^{b_{ij}}$ for the initial coefficients
	and $x_{\mathrm{mut}}=(x_1,\ldots,x_r)$. Let
	$m_{j;t}=F_{j;t}|_{\mathbb P}(y_1,\ldots,y_r)$ be the tropical
	evaluation in these coefficients, and define $m'_{j;t}$ similarly. These are frozen Laurent monomials. The separation
	formula \cite[Corollary~6.3]{FominZelevinskyIV} gives
	\[
	m_{j;t}x_{j;t}=x_{\mathrm{mut}}^{g_{j;t}}F_{j;t}(\widehat y),
	\qquad
	\psi_+(x'_{j;t})
	=\frac{m_{j;t}}{\psi_+(m'_{j;t})}\,x_{j;t}.
	\]
	The second equality follows by applying $\psi_+$ to the primed
	separation formula. Thus every cluster variable maps to the
	corresponding cluster variable times a coefficient unit. The inverse
	substitution has the same property, and specialization of all frozen
	variables to one gives the identity on the coefficient-free seed
	pattern. This proves the quasi-cluster assertion.
	
	For the categorical statement, assume in addition that the
	quivers with potential obtained by deleting the frozen vertices
	are non-degenerate. Let $\overline{\mathcal R}$ be the class of
	indecomposable summands of cluster-tilting objects reachable from
	$\bar\Gamma$ in $\cc(\bar Q,\bar W)$. Define $\mathrm{rch}(Q,F,W)$
	to be the full subcategory on the objects
	\[
	\bigl\{X\in\cc(Q,F,W)\mid
	p^*X\in\add(\overline{\mathcal R})\bigr\},
	\]
	and use the same definition on the primed side. Then there is a
	commutative diagram of maps on isomorphism classes
	\[
	\begin{tikzcd}
		\mathrm{rch}(Q',F',W')
		\arrow[r,"CC'_{\mathrm{loc}}"] \arrow[d,"\Psi_+"']
		& \ca_{Q'} \arrow[d,"\psi_+"] \\
		\mathrm{rch}(Q,F,W) \arrow[r,"CC_{\mathrm{loc}}"']
		& \ca_Q .
	\end{tikzcd}
	\]
	
	To check the codomains, let $N_{j;t}$ be the reachable indecomposable
	in the reduced category corresponding to $(j,t)$. Its index is
	$g_{j;t}$ and the quiver Grassmannians of $\bar R N_{j;t}$ give
	$F_{j;t}$ by
	\cite[Propositions~4.1, 4.6 and~4.8]{PlamondonClusterAlgebrasVia2011}
	and \cite[Theorem~5.1]{DerksenWeymanZelevinsky10}.
	Since $-l(\mathbf d)=\widetilde B\mathbf d$, the corrected character
	and the separation formula give
	\[
	\widetilde{CC}(N_{j;t})
	=x_{\mathrm{mut}}^{g_{j;t}}F_{j;t}(\widehat y)
	=m_{j;t}x_{j;t}.
	\]
	If $p^*X\simeq\bigoplus_a N_{j_a;t_a}$, additivity of the index
	and multiplicativity of the Grassmannian $F$-polynomial on direct
	sums therefore give
	\[
	CC_{\mathrm{loc}}(X)
	=x^{[X]^F}\prod_a\bigl(m_{j_a;t_a}\,x_{j_a;t_a}\bigr)
	\in\ca_Q .
	\]
	The same argument applies on the primed side. The equivalence
	induced by $\Psi_+$ after deleting the frozen vertices preserves
	the labelled initial cluster-tilting summands and their exchange
	triangles. Hence it identifies the two reduced reachable classes,
	so $\Psi_+$ restricts to the indicated subcategories.
	Commutativity follows from diagram~\eqref{diagram: frozen--cc maps}.
\end{Rem}

We have the dual statement of Proposition~\ref{Prop: decategorification of frozen source} for mutation at a frozen sink.

\begin{Prop}\label{Prop: decategorification of frozen sink}
	Let $v\in F_0$ be a frozen sink and $(Q',F',W')=\mu_v(Q,F,W)$.
	Then the unique $\mathbb Q$-algebra isomorphism
	$\psi_-\colon\mathcal L_{Q'}\ra\mathcal L_Q$ given by
	\[
	\psi_-(x'_i)=x_i \quad(i\ne v),\qquad
	\psi_-(x'_v)=
	\frac{\prod_{\alpha\in F_1,\,t(\alpha)=v} x_{s(\alpha)}}{x_v}
	\]
	makes the following diagram commute on isomorphism classes
	\begin{equation}\label{diagram: frozen sink--cc maps}
		\begin{tikzcd}
			\cc(Q',F',W')\arrow[d,"\Psi_-"]\arrow[r,"CC'_{\mathrm{loc}}"]
			&\mathcal L_{Q'}\arrow[d,"\psi_-"]\\
			\cc(Q,F,W)\arrow[r,"CC_{\mathrm{loc}}"']&\mathcal L_Q.
		\end{tikzcd}
	\end{equation}
	
\end{Prop}
\begin{proof}
	By \cite[Theorem~6.9]{wuCategorificationIceQuiver2021} and the
	reduction equivalence, the image of $\Gamma'_v$ is the shifted
	cone displayed below, and $\Psi_-\Gamma'_i\simeq\Gamma_i$ for
	$i\ne v$. As in the proof for the frozen source, the shifted
	cone has only frozen terms, so $\Psi_-S'_j\simeq S_j$ for every
	unfrozen $j$. Thus $K_0(\Psi_-)$ carries $l'(\mathbf d)$ to
	$l(\mathbf d)$ and the frozen class of $X$ to that of $\Psi_-X$.
	The comparison after deletion identifies the absolute indices and
	the quiver Grassmannians. The formula of
	Theorem~\ref{Thm: commutative diagram} now gives the asserted
	commutativity.
\end{proof}
\begin{Rem}\label{rem:frozen-sink-quasi-cluster}
	Keep the assumptions of Proposition~\ref{Prop: decategorification of frozen sink}
	and take the frozen variables to be invertible. Then $\psi_-$ restricts
	to a quasi-cluster isomorphism $\ca_{Q'}\rightarrow\ca_Q$ in
	the sense of Definition~\ref{def:quasi-cluster homomorphism}.
	
	Indeed, let $A_-$ be the matrix of $K_0(\Psi_-)$ with respect to the
	bases formed by the classes of the indecomposable projective modules. At the
	frozen sink $v$ we have
	\[
	\Psi_-(\Gamma'_v)\simeq
	\Sigma^{-1}\operatorname{Cone}\!\left(
	\bigoplus_{\substack{\alpha\in F_1\\t(\alpha)=v}}
	\Gamma_{s(\alpha)}\longrightarrow\Gamma_v\right),
	\qquad
	\psi_-(x'_v)=
	\frac{\prod_{\alpha\in F_1,\,t(\alpha)=v}x_{s(\alpha)}}{x_v}.
	\]
	For $i\ne v$, $\Psi_-(\Gamma'_i)\simeq\Gamma_i$. Since both terms
	of the cone are frozen projectives, applying
	$\operatorname{RHom}(\Psi_-\Gamma'_i,S_j)$ for all $i$ gives
	$\Psi_-^{-1}S_j\simeq S'_j$ for each unfrozen $j$. Thus
	$[S_j]=-\widetilde B_{\bullet j}$ gives
	$A_-\widetilde B'=\widetilde B$ and
	$\psi_-(\widehat y'_j)=\widehat y_j$. The separation formula used
	in Remark~\ref{rem:frozen-source-quasi-cluster} now shows that
	$\psi_-$ and its inverse send every cluster variable to the
	corresponding cluster variable times a frozen Laurent monomial.
	After specialization of the frozen variables to one, the induced
	map is the identity on the coefficient-free seed pattern.
	
	For the categorical statement, assume in addition that the quivers
	with potential obtained by deleting the frozen vertices are
	non-degenerate. Use $\mathrm{rch}$ as defined in
	Remark~\ref{rem:frozen-source-quasi-cluster}. The equivalence induced
	by $\Psi_-$ after deleting the frozen vertices preserves the labelled
	initial cluster-tilting summands and their exchange triangles. 
	Thus it
	identifies the reduced reachable summands, and $\Psi_-$ restricts to
	an equivalence of the indicated full subcategories. By the same remark,
	the localized characters on these subcategories take values in the
	respective cluster algebras. Restricting
	diagram~\eqref{diagram: frozen sink--cc maps} therefore gives the
	following commutative diagram on isomorphism classes
	\[
	\begin{tikzcd}
		\mathrm{rch}(Q',F',W')
		\arrow[r,"CC'_{\mathrm{loc}}"] \arrow[d,"\Psi_-"']
		& \ca_{Q'} \arrow[d,"\psi_-"] \\
		\mathrm{rch}(Q,F,W) \arrow[r,"CC_{\mathrm{loc}}"']
		& \ca_Q .
	\end{tikzcd}
	\]
\end{Rem}

\section{Examples from Postnikov diagrams}
\label{s:Postnikov diagrams}

Throughout this section we work over $\mathbb C$.
With each connected Postnikov diagram $D$ in the disc, one can associate
a canonical ice quiver with potential $(Q_D,F_D,W_D)$. The corresponding
complete relative Jacobian algebra $J_D=J(Q_D,F_D,W_D)$ is known as
the \emph{dimer algebra} of the diagram. Pressland has shown
in~\cite{presslandCalabiYauPropertiesPostnikov2019a} that it is
bimodule internally $3$-Calabi--Yau with respect to the boundary
idempotent, in the sense of~\cite{presslandinternal2017}. As a
consequence, the category of finitely generated Gorenstein projective
modules over the boundary algebra yields an additive categorification
of the cluster algebra associated with $(Q_D,F_D)$, cf.~also
\cite{BaurKingMarsh16,CanakciKingPressland21}. By the result of
Galashin--Lam~\cite{GalashinLam19}, after inverting the coefficients
this cluster algebra is isomorphic to the homogeneous coordinate ring
of the closed positroid variety localized at the necklace Pl\"ucker
coordinates. This is the coordinate ring of the open subset of its
affine cone lying over the open positroid variety
\cite[Theorem~1.2 and Remark~1.3]{presslandCalabiYauPropertiesPostnikov2019a}. In this
section, we explain how these ice quivers with potential fit into the
theory developed in this article. In particular, the categories of
Gorenstein projective modules studied by Pressland are examples of
the $\Hom$-infinite Higgs categories constructed in
Section~\ref{Section: Relative cluster categories and Higgs Categories}.

Let $D$ be a connected Postnikov diagram in the disc, and let
$(Q_D,F_D,W_D)$ be its ice quiver with potential
\cite[Definitions~2.1 and~2.4]{presslandCalabiYauPropertiesPostnikov2019a}.
By~\cite[Theorem~3.7]{presslandCalabiYauPropertiesPostnikov2019a}
and~\cite[Lemma~5.11 and Example~5.12]{wuCategorificationIceQuiver2021}, the complete
relative Ginzburg dg algebra $\bmgamma_D=\relGammabf(Q_D,F_D,W_D)$
has cohomology concentrated in degree zero, with $H^0(\bmgamma_D)=J_D$.
By~\cite[Proposition~4.4]{presslandCalabiYauPropertiesPostnikov2019a},
$J_D$ is Noetherian and $J_D/J_DeJ_D$ is finite-dimensional, where
$e=\sum_{v\in(F_D)_0}e_v$ is the boundary idempotent. 

Let $B_D=eJ_De$ be the boundary algebra.
By~\cite[Proposition~4.4]{presslandCalabiYauPropertiesPostnikov2019a},
$J_D$ is Noetherian and $J_D/J_DeJ_D$ is finite-dimensional.
Deleting the vertices corresponding to $e$ induces a surjection
\[
J_D/J_DeJ_D \twoheadrightarrow J(\overline Q_D,\overline W_D),
\]
so $(\overline Q_D,\overline W_D)$ is Jacobi-finite.

For each vertex $v$ of $Q_D$, choose a path $t_v\colon v\to v$
representing a fundamental cycle, and let $t=\sum_{v\in(Q_D)_0}t_v$.
The element $t$ is central in $J_D$, and $T\mapsto t$ makes $J_D$
an algebra over $Z=\mathbb C[[T]]$. By
\cite[Proposition~2.11]{presslandCalabiYauPropertiesPostnikov2019a},
the $Z$-module $e_wJ_De_v$ is free of rank one for all vertices
$v,w$ of $Q_D$. Thus $J_D$ and $B_D$ are finite free $Z$-modules
of ranks $|(Q_D)_0|^2$ and $|(F_D)_0|^2$, respectively.
In particular, $B_D$ is Noetherian. It follows
that for $P\in\add(e\bmgamma_D)$ and $M\in\add(\bmgamma_D)$,
$\Hom(P,M)$ and $\Hom(M,P)$ are finitely generated over $\End(P)$
on the respective sides, since they are already finitely generated
over $Z$. Taking $P=e\bmgamma_D$, generators of $\Hom(P,M)$ give a
right $\add(P)$-approximation of $M$, and generators of $\Hom(M,P)$
give a left one. Hence $\add(e\bmgamma_D)$ is functorially finite in
$\add(\bmgamma_D)$.

Thus $(Q_D,F_D,W_D)$ satisfies Assumption~\ref{assumption} of
Section~\ref{Section: Relationship with Plamondon's categor}.
Since $J_D$ and $B_D$ are Noetherian, Theorem~\ref{Thm: main results}
gives an exact equivalence $\ch\simeq\mathrm{gpr}(B_D)$, the
right-module version of Pressland's category in
\cite[Theorem~4.5]{presslandCalabiYauPropertiesPostnikov2019a}.
Under this equivalence, the canonical cluster-tilting object
$\bmgamma_D$ of $\ch$ corresponds to $T=J_De$. Moreover,
\[
\cc(Q_D,F_D,W_D)\simeq\cd^b(\mathrm{gpr}(B_D)).
\]

In particular, suppose that $D$ is a Grassmannian $(k,n)$-Postnikov
diagram. Let $A_{\mathrm{pol}}$ be the uncompleted Jensen--King--Su
algebra for $\operatorname{Gr}(k,n)$, and put
\[
\widehat A=A_{\mathrm{pol}}\otimes_{\mathbb C[t]}\mathbb C[[t]],
\qquad t=xy,
\]
as in~\cite[Section~3, Remarks~3.2--3.3]{jensenCategorificationGrassmannianCluster2016}.
In the conventions of~\cite[Theorem~11.2]{BaurKingMarsh16}, the
completed boundary algebra is isomorphic to $\widehat A^{\mathrm{op}}$.
Thus right $B_D$-modules correspond to left $\widehat A$-modules.
We write $\operatorname{mod}_{\ell}(\widehat A)$ for finitely generated
left modules and $\operatorname{CM}_{\ell}(\widehat A)$ for those
which are free over $\mathbb C[[t]]$. With this convention,
the above equivalences give
\[
\ch\simeq\operatorname{CM}_{\ell}(\widehat A),\qquad
\cc(Q_D,F_D,W_D)\simeq\cd^b(\operatorname{CM}_{\ell}(\widehat A))
\simeq\cd^b(\operatorname{mod}_{\ell}(\widehat A)).
\]
For the last equivalence, $\operatorname{CM}_{\ell}(\widehat A)$ is a
resolving subcategory, and the kernel of a surjection from a finite
free $\widehat A$-module onto a finitely generated module is finite
free over $\mathbb C[[t]]$, hence Cohen--Macaulay.

Returning to an arbitrary connected Postnikov diagram $D$, recall
from~\cite[Theorem~6.11]{presslandCalabiYauPropertiesPostnikov2019a}
that the Frobenius category $\ch\simeq\mathrm{gpr}(B_D)$ with its
cluster-tilting object $T$ gives an additive categorification of the
cluster algebra associated with $D$, using the right-module convention
above. Left multiplication identifies $J_D$ with $\End_{B_D}(T)$,
so $\Hom_{B_D}(T,M)$ and $\Ext^1_{\mathrm{gpr}(B_D)}(T,M)$ are
right $J_D$-modules by precomposition. The corresponding cluster
character, evaluated at $M\in\mathrm{gpr}(B_D)$, is given by the
Fu--Keller formula~\cite{Fu-Keller2010} in the form of
\cite[Definition~6.6 and Theorem~6.8]{presslandCalabiYauPropertiesPostnikov2019a}
\begin{equation}\label{eq:Fu-Keller formula}
	CC(M)=x^{[FM]}\sum_d
	\chi\!\left(\mathrm{Gr}_d\!\left(
	\Ext^1_{\mathrm{gpr}(B_D)}(T,M)\right)\right)x^{-[N_d]}.
\end{equation}
Here we use the following notation.
\begin{itemize}
	\item[a)] $F$ is the functor
	\[
	\Hom_{\mathrm{gpr}(B_D)}(T,-)\colon
	\mathrm{gpr}(B_D)\longrightarrow\mathrm{mod}(J_D).
	\]
	\item[b)] The algebra $J_D$ is Noetherian and has global dimension at
	most three by~\cite[Theorem~3.7]{presslandCalabiYauPropertiesPostnikov2019a}.
	Hence every finitely generated right $J_D$-module $L$ is perfect.
	We denote its class in $K_0(\per J_D)\cong K_0(\mathrm{proj}\,J_D)$ by $[L]$.
	We identify this group with $\mathbb Z^{(Q_D)_0}$ using the classes
	of the indecomposable projectives $P_i$, $i\in(Q_D)_0$, as a basis.
	\item[c)] For each dimension vector $d$ occurring in the sum, $N_d$
	denotes a submodule of $\Ext^1_{\mathrm{gpr}(B_D)}(T,M)$ of dimension
	vector $d$. Its class $[N_d]$ depends only on $d$, as in the proof
	of~\cite[Proposition~3.2]{Fu-Keller2010}.
\end{itemize}

Write $Ne$ for the image of $N\in\ch$ under the exact equivalence
$\ch\simeq\mathrm{gpr}(B_D)$. Since $T$ is cluster-tilting, we
may choose a conflation $T_1\rightarrowtail T_0\twoheadrightarrow Ne$
with $T_0,T_1\in\add T$. The vanishing of
$\Ext^1_{\mathrm{gpr}(B_D)}(T,T_1)$ gives a projective resolution
\[
0\longrightarrow FT_1\longrightarrow FT_0
\longrightarrow F(Ne)\longrightarrow 0.
\]
The equivalence takes the vertex summands of $\bmgamma_D$ to the
corresponding summands of $T$. In the associated projective bases,
we therefore have
\[
[F(Ne)]=[FT_0]-[FT_1]=\ind_{\bmgamma_D}N.
\]
It also identifies $RN$ with $\Ext^1_{\mathrm{gpr}(B_D)}(T,Ne)$
as right $J_D$-modules and hence identifies their quiver
Grassmannians. For each frozen vertex $i$, the summand $T_i$ is
projective, so $\Ext^1_{\mathrm{gpr}(B_D)}(T_i,Ne)=0$. Thus
each submodule $N_d$ is supported on the unfrozen vertices. Under
the identification $K_0(\per\bmgamma_D)\simeq K_0(\per J_D)$,
a composition series gives
\[
[N_d]=\sum_{i\in(Q_D)_0\setminus(F_D)_0} d_i[S_i]=l(d).
\]
Hence $CC(Ne)=X_N$, where $CC$ is the character defined
by~\eqref{eq:Fu-Keller formula} and $X_N$ is the Higgs cluster
character defined by~\eqref{CC formula for Higgs}.

\appendix
\section{Frobenius categorification of quasi-cluster morphisms, \\
	by C.~Fraser and B.~Keller}
\label{appendix}

\subsection{Quasi-cluster morphisms} \label{ss:quasi-cluster morphisms}
We recall the notion of quasi-cluster morphism from
\cite{fraserQuasihomomorphismsClusterAlgebras2016a}. Let $Q$ be an ice quiver with
frozen subquiver $F\subseteq Q$. Let us assume that the set of non frozen vertices is
formed by the integers $1$, \ldots, $r$ and the set of frozen vertices by $r+1$, \ldots, $n$.
Let $\T_r$ be the $r$-regular tree with root $t_0$. Let $\ca_Q$ be the cluster algebra
{\em with invertible coefficients} associated with $Q$. Let us write $x_1$, \ldots, $x_n$ for
its initial cluster variables. Let $\P$ be its {\em coefficient group}, \ie the group of Laurent monomials 
\[
x_{r+1}^{e_{r+1}} \cdots x_n^{e_n}, e_i \in \Z \ko
\]
in the frozen variables. Let $\ol{Q}$ be the full subquiver on the mutable vertices of $Q$ and
$\ca_{\ol{Q}}$ the associated cluster algebra without coefficients. For a vertex $t$ of $\T_r$,
let $\widetilde B(t)=(b_{ij}(t))_{1\leq i\leq n,\,1\leq j\leq r}$ be the extended exchange matrix
of the corresponding labelled seed, obtained by matrix mutation from
\[
b_{ij}(t_0)=\#\{i\to j\}-\#\{j\to i\}.
\]
Write $x_i(t)$ for the associated cluster variables, with $x_i(t)=x_i$ for $i>r$.
For each mutable vertex $j$, put
\[
\hat{y}_j(t)=\prod_{i=1}^n x_i(t)^{b_{ij}(t)}.
\]

\begin{lemma}\label{lemma:specialization}
	Let both cluster algebras be taken over $\Q$, with the frozen
	variables of $\ca_Q$ invertible. The specialization map
	\[
	\Q[x_1^{\pm1},\ldots,x_n^{\pm1}]
	\longrightarrow
	\Q[x_1^{\pm1},\ldots,x_r^{\pm1}]
	\]
	taking $x_j$ to $1$ for $j>r$ restricts to a surjective
	homomorphism $\mathrm{sp}\colon\ca_Q\to\ca_{\ol Q}$. It takes each
	mutable cluster variable to the variable at the same position
	in the corresponding labelled seed.
	
	Put $I=(m-1\mid m\in\P)\ca_Q$ and $s=x_1\cdots x_r$. Then
	\[
	\ker(\mathrm{sp})=(I:s^\infty)
	=\{a\in\ca_Q\mid s^Na\in I\text{ for some }N\geq0\}.
	\]
	If $\ol Q$ is mutation-equivalent to an acyclic quiver, then
	$\mathrm{sp}$ induces an isomorphism
	\[
	\ca_Q/I\xrightarrow{\sim}\ca_{\ol Q}.
	\]
\end{lemma}

\begin{proof}
	Write $A=\ca_Q$ and $L=\Q[x_1^{\pm1},\ldots,x_n^{\pm1}]$.
	In each exchange relation, setting the frozen variables to $1$
	removes the frozen factors from its two monomials. The principal
	part of the exchange matrix mutates by the same rule on both sides.
	
	By induction along a labelled mutation path, it implies
	$\mathrm{sp}(x_{i;t})=\bar x_{i;t}$ for every mutable $i$ and
	every seed $t$. These variables generate $\ca_{\ol Q}$ over $\Q$,
	so the restriction of $\mathrm{sp}$ is surjective.
	
	Since $\P$ is generated by the frozen variables and their inverses,
	$I=(x_{r+1}-1,\ldots,x_n-1)A$. On $L$ the kernel of specialization
	is $(x_{r+1}-1,\ldots,x_n-1)L=IL$. By the Laurent phenomenon
	$A\subseteq L$. Conversely, $A[s^{-1}]$ contains the inverses of
	all initial mutable variables, and the frozen variables are already
	units in $A$. Thus $A[s^{-1}]=L$, and contraction gives
	\[
	\ker(\mathrm{sp})=A\cap IL=A\cap I A[s^{-1}]=(I:s^\infty).
	\]
	
	Finally, suppose that $\ol Q$ is mutation-equivalent to an acyclic
	quiver. Choose a seed whose principal part is acyclic. Write
	$u_1,\ldots,u_r$ for its mutable variables, $u_1',\ldots,u_r'$ for
	their one-step mutations, and $u_i u_i'=P_i(\mathbf u)$ for the
	exchange relations over $R=\Q[x_{r+1}^{\pm1},\ldots,x_n^{\pm1}]$.
	By \cite[Corollary~1.21]{Berenstein-Fomin-Zelevinsky2005},
	\[
	A\cong
	R[U_1,\ldots,U_r,V_1,\ldots,V_r]/
	\bigl(U_iV_i-P_i(\mathbf U)\mid 1\leq i\leq r\bigr).
	\]
	Taking the quotient by $I$ replaces $R$ by $\Q$ and each $P_i$
	by its coefficient-free specialization. By the same corollary,
	the resulting presentation is a presentation of $\ca_{\ol Q}$.
	The isomorphism sends the classes of $u_i,u_i'$ to their
	specializations, so it is induced by $\mathrm{sp}$.
\end{proof}

\bigskip

Let $Q'$ be another ice quiver and $\P'$ the associated coefficient group.
Following \cite[Definition~3.1]{fraserQuasihomomorphismsClusterAlgebras2016a},
we call a $\Q$-algebra morphism $\phi\colon\ca_Q\to\ca_{Q'}$ a
{\em quasi-cluster morphism} if the following conditions hold.
\begin{itemize}
	\item[a)] We have $\phi(\P)\subseteq\P'$. For each mutable cluster variable
	$x$ of $\ca_Q$, there is a mutable cluster variable $x'$ of $\ca_{Q'}$
	such that $\phi(x)\in\P'x'$.
	\item[b)] The map $\ol{\phi}\colon\ca_{\ol Q}\to\ca_{\ol{Q'}}$ determined by
	$\ol{\phi}\circ\mathrm{sp}=\mathrm{sp}'\circ\phi$ is a cluster isomorphism
	taking the initial coefficient-free seed to a seed with the same principal
	exchange matrix after relabelling.
	\item[c)] There is a seed $t'$ of $\ca_{Q'}$ whose specialization is the image
	seed in b), with a labelling $j\mapsto j'$ such that
	$\ol{\phi}(\bar x_j(t_0))=\bar x'_{j'}(t')$ and
	\[
	\phi(\hat y_j(t_0))=\hat y'_{j'}(t')\qquad(1\leq j\leq r).
	\]
\end{itemize}
Here $\mathrm{sp}'$ denotes coefficient specialization for $Q'$.
In c) we evaluate $\phi$ on the initial Laurent ring by substitution
in $\operatorname{Frac}(\ca_{Q'})$, since a) ensures that all initial
variables have nonzero images.
Let us denote by $\cu_Q$ the {\em upper cluster algebra with invertible coefficients} associated
with $Q$. We define a {\em quasi-cluster morphism} $f: \cu_Q \to \cu_{Q'}$ to be a ring 
homomorphism inducing a quasi-cluster morphism $\ca_Q \to \ca_{Q'}$.

As a simple example, consider the ice quiver
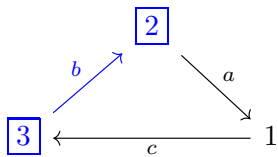
\[
\begin{tikzcd}
	&\color{blue}\boxed{2}\arrow[dr,"a"]\\
	\color{blue}\boxed{3}\arrow[ur,blue,"b"]&&1\arrow[ll,"c"]
\end{tikzcd}
\]
Let us denote the initial cluster variables by $x_1$, $p_1=x_2$ and $p_2=x_3$. Then the only other
cluster variable is
\[
x_1' = \frac{p_1+p_2}{x_1}.
\]
The associated cluster algebra with invertible coefficients is
\[
\ca_Q = \Q[x_1, x_1', p_1^\pm, p_2^\pm]/(x_1 x_1' - p_1-p_2)
\]
and we have $\hat{y}_1 = p_1/p_2$. The associated cluster algebra without coefficients is
\[
\ca_{\ol{Q}}= \Q[x_1, 2/x_1].
\]
Define the algebra automorphism $\sigma: \ca_Q \to \ca_Q$ to send $p_i$ to $1/p_i$, $x_1$
to $x_1'/p_2$ and $x_1'$ to $x_1/p_1$. If $t_1$ is the seed obtained by
mutation at $1$, matrix mutation gives
\[
\hat y_1(t_1)=\frac{p_2}{p_1}
=\frac{\sigma(p_1)}{\sigma(p_2)}
=\sigma(\hat y_1(t_0)).
\]
After specialization, $\sigma$ interchanges the two coefficient-free cluster
variables. Thus $\sigma$ is a quasi-cluster automorphism.

\subsection{Frobenius categorification of cluster algebras with coefficients} \label{ss:Frobenus categorification} 
As before, let $(Q,F)$ be an ice quiver.
Denote by $\ca^+_Q$ the cluster algebra with {\em non invertible} coefficients associated with $Q$.
For example, for the quiver 
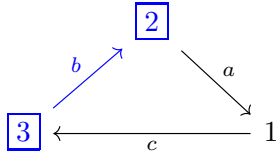
\[
\begin{tikzcd}
	&\color{blue}\boxed{2}\arrow[dr,"a"]\\
	\color{blue}\boxed{3}\arrow[ur,blue,"b"]&&1\arrow[ll,"c"]
\end{tikzcd}
\]
we have
\[
\ca^+_Q = \Q[x_1, x_1', p_1, p_2]/(x_1 x_1' - p_1-p_2).
\]

Now let $k$ be an algebraically closed field and assume that $(\ce,T)$ is a {\em Frobenius
	categorification} of the quiver $Q$. By definition, this means that 
\begin{itemize}
	\item[a)] $\ce$ is a $k$-linear Krull--Schmidt Frobenius exact category which is enriched 
	over the monoidal category of pseudo-compact vector spaces, cf.~section~4 of
	\cite{Van den Bergh2015};
	\item[b)] the stable category $\ul{\ce}$ is $\Hom$-finite and $2$-Calabi--Yau, \ie
	we have bifunctorial isomorphisms
	\[
	D\Hom_{\ul{\ce}}(X,Y) \iso \Hom_{\ul{\ce}}(Y, \Si^2 X)
	\]
	for $X,Y\in \ul{\ce}$, where $D$ is the duality over the ground field $k$;
	\item[c)] $T$ is a basic cluster-tilting object of $\ce$ and we are given an isomorphism
	between $Q$ and the quiver of the endomorphism algebra $A$ of $T$ such that the
	frozen vertices correspond to the projective-injective indecomposable direct factors of $T$,
	the number of frozen arrows from $i$ to $j$ is
	\[
	\dim \Ext^1_A(S_j, S_i) - \dim \Ext^2_A(S_i, S_j)
	\]
	and the number of non frozen arrows from $i$ to $j$ is
	\[
	\dim \Ext^2_A(S_i, S_j)
	\]
	for all vertices $i$ and $j$, where $S_i$ denotes the simple module at $i$.
	\item[d)] the basic cluster-tilting objects of $\ce$ determine a cluster structure on $\ce$ in the 
	sense of \cite[Section~II.1]{BIRS2007}. 
\end{itemize}
By \cite[Theorem~II.1.6]{BIRS2007}, if conditions a)--c) hold, then condition d) holds if no cluster-tilting
object of $\ce$ has loops or $2$-cycles in the quiver of its endomorphism algebra. 
By Prop.~2.19 (v) of \cite{GLS2011},
this holds for many stably $2$-CY categories occurring in Lie theory.

Number the indecomposable direct factors $T_i$ of $T$, $1\leq i\leq n$,
so that $T_i$ is projective for $i>r$ and non-projective for $i\leq r$.
By iterated mutation, each vertex $t$ of the regular tree $\T_r$
determines a basic cluster-tilting object $T(t)$. The signed extended
exchange matrix of the quiver of $\End_{\ce}(T(t))$, with columns
indexed by the non-projective summands, is $\widetilde B(t)$.

From now on, let $k=\C$. We assume that $(\ce,T)$ is equipped with
a cluster character $CC=CC_T\colon\ce\to\cu_Q$ such that
$CC_T(T_i)=x_i$ for $1\leq i\leq n$.

\begin{Thm} The cluster character $CC$ induces the following bijections.
	\begin{itemize}
		\item[-] from the set of isoclasses of reachable rigid indecomposables of $\ce$ to the set of
		cluster variables of $\ca^+_Q$, by $L\mapsto CC(L)$;
		\item[-] from the set of isoclasses of basic reachable cluster-tilting objects of $\ce$ to the set of clusters of $\ca^+_Q$,
		by sending $\bigoplus_i L_i$ to $\{CC(L_i)\}_i$, where the $L_i$ are indecomposable.
	\end{itemize}
\end{Thm}

\begin{proof}
	For each indecomposable $X$ in $\ce$, the algebra $A_X=\End_{\ce}(X)$
	is local and pseudo-compact. It has a proper open ideal of finite
	codimension, which is contained in $\operatorname{rad} A_X$. Thus
	$A_X/\operatorname{rad} A_X$ is a finite-dimensional division algebra
	over $\C$ and is therefore equal to $\C$. By a)--d), $\ce$ is a
	skew-symmetric cluster category with a cluster structure in the sense of
	\cite[Definitions~2.16 and~2.24]{GrabowskiPressland2024}.
	The $\mathcal A$-cluster character of Grabowski--Pressland and $CC_T$
	both take $T_i$ to $x_i$. By multiplicativity on direct sums and the
	exchange formula, they agree on the indecomposable summands of each
	reachable cluster-tilting object, by induction on the length of a
	mutation path. At each step the common value at the summand being mutated
	is the current cluster variable, so it may be cancelled.
	The two bijections now follow from
	\cite[Theorem~5.26]{GrabowskiPressland2024}, by taking the set of
	character values of the indecomposable summands in the second case.
\end{proof}

As an example, let us consider the category $\ce=\mod\La$ of finite-dimensional modules over the
preprojective algebra $\La$ of type $A_2$ given by the quiver
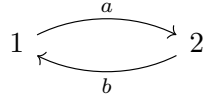
\[
\begin{tikzcd}
	1 && 2
	\arrow["a", curve={height=-12pt}, from=1-1, to=1-3]
	\arrow["b", curve={height=-12pt}, from=1-3, to=1-1]
\end{tikzcd}
\]
with the relations $ab=0$ and $ba=0$.
Then $\ce$ contains four isoclasses of indecomposable objects represented by
the two simple modules $S_1$ and $S_2$ and the two indecomposable projective
modules $P_1$ and $P_2$, where $P_i$ is the projective cover of $S_i$, $i=1,2$. 
The object $T=S_1\oplus P_1 \oplus P_2$ is cluster-tilting. Its endomorphism algebra
is given by the quiver
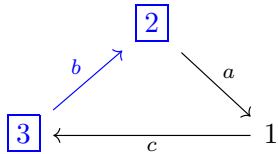
\[
\begin{tikzcd}
	&\color{blue}\boxed{2}\arrow[dr,"a"]\\
	\color{blue}\boxed{3}\arrow[ur,blue,"b"]&&1\arrow[ll,"c"]
\end{tikzcd}
\]
with the relations $ab=0$, $bc=0$. Here the vertex $1$ corresponds to $S_1$, the vertex
$2$ to $P_1$ and the vertex $3$ to $P_2$. 
Notice that the endomorphism algebra is also the relative Jacobian algebra of the same
quiver endowed with the potential $W=abc$. The cluster character associated with $T$
takes $S_1$, $S_2$, $P_1$ and $P_2$ to $x_1$, $x'_1$, $p_1$
and $p_2$, respectively. The space $\Ext^1_\ce(S_1,S_2)$ is one-dimensional and we have
the associated exchange conflations
\[
S_1 \rightarrowtail P_2 \twoheadrightarrow S_2 \mbox{ and }
S_2 \rightarrowtail P_1 \twoheadrightarrow S_1.
\]
They decategorify to the exchange relation
\[
CC(S_1)CC(S_2) = CC(P_1)+CC(P_2),\qquad x_1 x'_1 = p_1+p_2.
\]
Our first aim in this appendix is to propose a categorical interpretation of
expressions like $1/p_1$, $x'_1/p_1$,~\ldots which have monomials
in frozen variables as denominators. We will use objects of the derived category
$\cd^b(\ce)$ for this purpose, cf.~Theorem~\ref{thm:extension} below.

\subsection{On split Grothendieck groups of Frobenius categories} 
Let $\ce$ be a Krull--Schmidt Frobenius category and $\cp\subseteq \ce$ its full subcategory
of projective-injectives. We denote by $K_0^{sp}(\ce)$ the {\em split Grothendieck group}
of $\ce$ which has a basis given by the isomorphism classes of the indecomposables in $\ce$.
Exceptionally, in this section, we will denote by $X \mapsto X[1]$ the suspension
functor of the derived category $\cd^b(\ce)$. We denote by $K_0^{sp}(\cd^b(\ce))$ the
split Grothendieck group of $\cd^b(\ce)$ and by $K_{\ce,\cp}$ its quotient by the subgroup
generated by all elements $[P]-[X]+[Y]$, where $P$, $X$ and $Y$ appear in a triangle
\[
P \to X \to Y \to P[1]
\]
of $\cd^b(\ce)$ and $P$ is isomorphic to a bounded complex of projectives of $\ce$. 

\begin{Prop} \label{prop:extension} The morphism $\phi: K_0^{sp}(\ce) \to K_{\ce,\cp} $ induced
	by the canonical functor  $\ce \to \cd^b(\ce)$ is bijective.
\end{Prop}

\begin{proof} We will construct the inverse $\psi$ of $\phi$. Let $\cc^{-,b}(\cp)$ denote the category
	of right bounded complexes $X$ with components in $\cp$ which are acyclic in all degrees $n\ll 0$,
	which means that we have conflations
	\[
	Z^n(X) \rightarrowtail X^n \twoheadrightarrow Z^{n+1}(X).
	\]
	The corresponding homotopy category $\ch^{-,b}(\cp)$ is canonically equivalent to 
	$\cd^b(\ce)$. Thus it suffices to construct an element $\psi(X)$ in $K_0^{sp}(\ce)$ for
	each complex $X$ in $\cc^{-,b}(\cp)$ such that 
	\begin{itemize}
		\item[a)] we have $\psi(X)=\psi(X')$ if $X$ and $X'$ are homotopy equivalent;
		\item[b)] we have $\psi(P)+\psi(Y)=\psi(X)$ whenever there is a triangle
		\[
		P \to X \to Y \to P[1]
		\]
		of $\ch^{-,b}(\cp)$, where $P$ is bounded;
		\item[c)] the induced map $\psi: K_{\ce,\cp} \to K_0^{sp}(\ce)$ is inverse to $\phi$.
	\end{itemize}

	\noindent
	{\em 1st step: For each $n\leq 0$, we define $\psi(X)$ for complexes
		$X$ in $\cc^{-,b}(\cp)$ which have nonzero components only in
		degrees $\leq n$ and admit a conflation
		\[
		Z^n(X) \rightarrowtail X^n \twoheadrightarrow Y
		\]
		as well as conflations
		\[
		Z^p(X) \rightarrowtail X^p \twoheadrightarrow Z^{p+1}(X)
		\]
		for all $p\leq n-1$.} Here $Z^n(X)$ is the kernel of the augmentation
	$X^n\twoheadrightarrow Y$, and the conflations induce the differentials
	of $X$. In other words, such a complex is a projective resolution of
	$Y[-n]$. Put $m=-n$ and choose an injective resolution
	\[
	0 \to Y \to I^0 \to I^1 \to \cdots .
	\]
	Write $\Si^p Y$ for the $p$th cosyzygy in this resolution, with
	$\Si^0 Y=Y$. From the conflation
	\[
	Y \rightarrowtail I^0 \twoheadrightarrow \Si Y
	\]
	we deduce that in $K_{\ce,\cp}$ we have
	\[
	[Y[1]] = [\Si Y] - [I^0],
	\]
	and by induction that
	\[
	[X]=[Y[m]] = [\Si^m Y] + (-1)^m \sum_{p=0}^{m-1} (-1)^p [I^p].
	\]
	Therefore, we define the element $\psi(X)$ of $K_0^{sp}(\ce)$ by
	\[
	\psi(X)=[\Si^m Y] + (-1)^m \sum_{p=0}^{m-1} (-1)^p [I^p],
	\]
	with the sum understood to be empty when $m=0$.
	By Schanuel's lemma for injective resolutions
	\cite[Proposition~3.4]{MathieuRosbotham2022},
	this expression is independent of the chosen injective resolution. A homotopy equivalence
	with another such complex induces an isomorphism on the cokernels of
	the differentials in degree $n-1$, so it preserves the isomorphism
	class of $Y$ and hence $\psi(X)$. By construction, $\phi(\psi(X))=[X]$.
	
	{\em 2nd step: We define $\psi(X)$ for bounded complexes $X$ with projective components by
		\[
		\psi(X) = \sum_{p\in\Z} (-1)^p [X^p].
		\]
	}
	
	{\em 3rd step: We define $\psi(X)$ for general $X$ in $\cc^{-,b}(\cp)$.}
	For this, we choose $N\ll 0$ such that $X$ is acyclic in degrees $n\leq N$. We then have a triangle
	\[
	\si_{> N}(X) \to X \to \si_{\leq N}(X) \to (\si_{>N}(X))[1]
	\]
	and we define
	\[
	\psi(X) = \psi (\si_{> N}(X)) + \psi(\si_{\leq N}(X))
	\]
	using the first and second steps. Here $\si$ denotes brutal truncation.
	Write $C_N=\operatorname{coker}(d^{N-1})$ and denote the expression
	in Step~1 for $A[m]$ by $q_m(A)$. At two consecutive cutoffs the conflation
	$C_{N-1}\rightarrowtail X^N\twoheadrightarrow C_N$ gives
	\[
	q_{1-N}(C_{N-1})=q_{-N}(C_N)+(-1)^{1-N}[X^N].
	\]
	Indeed, we may take $X^N$ as the first term of an injective
	resolution of $C_{N-1}$. The last summand cancels the change
	in the Euler class of the upper truncation. Thus $\psi(X)$
	is independent of $N$.
	
	The formula is additive on direct sums. For a contractible
	complex of projectives the cycle conflations split. In particular,
	$C_N$ is projective, so $q_{-N}(C_N)=(-1)^N[C_N]$ and the
	alternating sum telescopes to zero. If $f\colon X\to X'$ is a
	homotopy equivalence, its mapping cylinder is isomorphic, as a
	complex, both to $X'$ plus a contractible complex and to $X$
	plus a contractible complex. Indeed, the corresponding degreewise
	split sequences have contractible quotients
	$\operatorname{Cone}(1_X)$ and $\operatorname{Cone}(f)$,
	respectively, and hence split as sequences of complexes.
	This proves a).
	
	For b), represent $Y$ by $\operatorname{Cone}(P\to X)$ and choose
	$N$ below the supports of both $P$ and $P[1]$. The lower truncations
	of $Y$ and $X$ coincide. Since $Y^i=X^i\oplus P^{i+1}$, the
	upper Euler classes differ by $-\psi(P)$. Hence
	$\psi(Y)=\psi(X)-\psi(P)$. Finally, the truncation triangle gives
	$\phi\psi([X])=[X]$. Taking $n=0$ in Step~1 for a projective
	resolution of an object of $\ce$ gives $\psi\phi=\mathrm{id}$,
	which proves c).
\end{proof}

\subsection{Extension of cluster characters to the derived category} Let $\ce$ be a Krull--Schmidt
Frobenius category which is stably $2$-Calabi--Yau. Suppose that 
$CC: \ce \to R$ is a cluster character with values in a commutative domain $R$.
Suppose that the isomorphism classes of the indecomposable projectives of $\ce$ are
represented by objects $P_{r+1}$, \ldots, $P_n$ and put $x_i=CC(P_i)$ for $r+1\leq i\leq n$.
Let $\iota: R \to R_{loc}$ be a ring homomorphism to a domain which makes $x_{r+1}$, \ldots, $x_n$ 
invertible. Let us also denote the canonical functor $\ce \to \cd^b(\ce)$ by $\iota$. 

Let $\pi: R_{loc} \to \ul{R}$ be a ring homomorphism which sends the $\iota(x_i)$ to $1$, $r+1\leq i\leq n$.
Let us also denote by  $\pi: \cd^b(\ce) \to \ul{\ce}$ the canonical triangle functor. 
On $\ul{\ce}$, put $\ul{CC}(\ul X)=\pi(\iota(CC(X)))$. This is
well defined since stably isomorphic objects become isomorphic after
adding projective summands, whose character values specialize to $1$.

\begin{Thm}\label{thm:extension}
	\begin{itemize}
		\item[a)] There is a unique map $CC_{loc}\colon \cd^b(\ce)\to R_{loc}$
		on isomorphism classes of objects such that
		\begin{itemize}
			\item[1)] the following diagram commutes
			\[
			\begin{tikzcd}
				\ce \arrow[d,"\iota"'] \arrow[r,"CC"]
				& R \arrow[d,"\iota"] \\
				\cd^b(\ce) \arrow[d,"\pi"'] \arrow[r,"CC_{loc}"]
				& R_{loc} \arrow[d,"\pi"] \\
				\ul{\ce} \arrow[r,"\ul{CC}"]
				& \ul{R}
			\end{tikzcd}
			\]
			\item[2)] whenever we have a triangle
			\[
			P \to X \to Y \to P[1]
			\]
			of $\cd^b(\ce)$, where $P$ is isomorphic to a bounded complex of
			projectives, we have
			\[
			CC_{loc}(X)=CC_{loc}(P)\,CC_{loc}(Y).
			\]
		\end{itemize}
		\item[b)] Let $X$ and $Y$ be objects of $\cd^b(\ce)$ such that the space
		$\Ext^1_{\ul{\ce}}(\pi X,\pi Y)$ is one-dimensional.
		Let $E$ and $E'$ be the middle terms of triangles
		\[
		Y \to E \to X \to Y[1]
		\qquad\text{and}\qquad
		X \to E' \to Y \to X[1]
		\]
		of $\cd^b(\ce)$ whose images in $\ul{\ce}$ are non-split.
		Then we have
		\[
		CC_{loc}(X)\,CC_{loc}(Y)=CC_{loc}(E)+CC_{loc}(E').
		\]
	\end{itemize}
\end{Thm}

\begin{proof}
	a) Let $M$ be the submonoid of $K_0^{sp}(\ce)$ consisting of the
	elements $[A]-[P]$ with $A\in\ce$ and $P\in\cp$. Since $\ce$ is
	Krull--Schmidt, $M$ is obtained from the monoid of isomorphism classes
	in $\ce$ under direct sum by inverting the classes of projectives.
	Thus $\iota\circ CC$ extends to a monoid homomorphism
	$CC_1\colon M\to(R_{loc},\cdot)$ given by
	\[
	CC_1([A]-[P])=\iota(CC(A))\,\iota(CC(P))^{-1}.
	\]
	
	Let $\phi$ be the isomorphism of Proposition~\ref{prop:extension}.
	Its proof shows that for each $X\in\cd^b(\ce)$ we have
	\[
	\phi^{-1}([X])=[A]+\eta,\qquad
	A\in\ce,\quad \eta\in K_0^{sp}(\cp),\quad \pi X\cong\pi A.
	\]
	Indeed, the truncation triangle in Step~3 reduces $X$ to $Y[m]$
	up to a bounded complex of projectives, and the cosyzygy triangles
	in Step~1 identify $Y[m]$ with $A=\Si^mY$ in the stable quotient,
	with the difference of their classes given by shifts of projectives.
	Thus $\phi^{-1}([X])\in M$, and we define
	\[
	CC_{loc}(X)=CC_1\bigl(\phi^{-1}([X])\bigr).
	\]
	The map is multiplicative on direct sums, and the top square commutes
	by construction. Condition~2) follows from
	the defining relations of $K_{\ce,\cp}$. The bottom square commutes
	since $\pi X\cong\pi A$ and $\pi$ sends the projective character
	values to $1$.
	
	For uniqueness, conditions~1) and~2) determine the value on a bounded
	complex $P^\bullet$ of projectives as
	$\prod_j\iota(CC(P^j))^{(-1)^j}$, by shifts and truncation.
	Applying condition~2) to the same truncation and cosyzygy triangles
	then determines the value on $X$ from the value on $A$.
	
	b) Suppose first that $X$ and $Y$ belong to $\ce\subset\cd^b(\ce)$.
	The two triangles are represented by non-split conflations
	\[
	Y\rightarrowtail E\twoheadrightarrow X
	\qquad\text{and}\qquad
	X\rightarrowtail E'\twoheadrightarrow Y.
	\]
	Thus the claim follows from the exchange formula for $CC$.
	Let us choose conflations
	\[
	X\rightarrowtail IX\twoheadrightarrow\Si X
	\qquad\text{and}\qquad
	Y\rightarrowtail IY\twoheadrightarrow\Si Y
	\]
	with $IX$ and $IY$ injective. By the injective horseshoe lemma,
	we can choose conflations
	\[
	E\rightarrowtail IY\oplus IX\twoheadrightarrow\Si E
	\qquad\text{and}\qquad
	E'\rightarrowtail IX\oplus IY\twoheadrightarrow\Si E'
	\]
	so that the cosyzygies form conflations
	$\Si Y\rightarrowtail\Si E\twoheadrightarrow\Si X$ and
	$\Si X\rightarrowtail\Si E'\twoheadrightarrow\Si Y$.
	Their stable triangles are the suspensions of the original ones.
	
	Write $IE=IY\oplus IX$ and $IE'=IX\oplus IY$.
	For $U\in\{X,Y,E,E'\}$ and $m\geq0$, the rotated cosyzygy triangle
	and part~a) give
	\[
	CC_{loc}(U[m+1])
	= CC_{loc}((\Si U)[m])\,CC_{loc}((IU)[m])^{-1}.
	\]
	The two middle terms have the same denominator,
	$CC_{loc}((IX)[m])\,CC_{loc}((IY)[m])$, which is a unit.
	Applying the formula at shift $m$ to the two cosyzygy conflations
	and dividing by this common unit proves the formula at shift $m+1$.
	Thus it holds for $X[n]$ and $Y[n]$ for all $n\geq0$.
	
	Now let $X$ and $Y$ be arbitrary objects of $\cd^b(\ce)$. Choose
	right bounded projective resolutions of $X$ and $Y$. Representing the
	two connecting morphisms by chain maps and taking shifted cones gives
	models for $E$ and $E'$ with componentwise split conflations in the
	two indicated directions. Choose $N\ll0$ so that all four complexes
	are acyclic in degrees $\leq N$. In each complex replace the term in
	degree $N-1$ by $\operatorname{coker}d^{N-2}$ and discard the terms in
	lower degrees. We keep the same notation for the resulting complexes.
	By exactness of the discarded tails and the $3\times3$ lemma, we have
	conflations
	\[
	Y^{N-1}\rightarrowtail E^{N-1}\twoheadrightarrow X^{N-1}
	\qquad\text{and}\qquad
	X^{N-1}\rightarrowtail E'^{N-1}\twoheadrightarrow Y^{N-1}.
	\]
	
	For each $U\in\{X,Y,E,E'\}$ we have a componentwise split conflation
	of complexes
	\[
	\si_{\geq N}(U)\rightarrowtail U\twoheadrightarrow U^{N-1}[1-N].
	\]
	Since $\si_{\geq N}(U)$ is a bounded complex of projectives, the
	projections identify the two given triangles in the stable quotient
	with those of the bottom conflations shifted by $1-N$. Thus these
	conflations satisfy the hypotheses of the shifted case above.
	Moreover, part~a) gives
	\[
	CC_{loc}(U)
	=CC_{loc}(\si_{\geq N}(U))\,CC_{loc}(U^{N-1}[1-N]).
	\]
	The componentwise split conflations of the upper truncations give
	\[
	\begin{aligned}
		CC_{loc}(\si_{\geq N}(E))
		&=CC_{loc}(\si_{\geq N}(X))\,CC_{loc}(\si_{\geq N}(Y))\\
		&=CC_{loc}(\si_{\geq N}(E')).
	\end{aligned}
	\]
	Multiplying the exchange formula for $X^{N-1}[1-N]$ and
	$Y^{N-1}[1-N]$ by this common factor yields the formula for $X$ and $Y$.
\end{proof}

\begin{Rem} \label{rmk:fractions}
	Let $X$ and $X'$ be objects of $\cd^b(\ce)$ whose images in the
	stable category $\cd^b(\ce)/\ch^b(\cp) \iso \ul{\ce}$ are isomorphic. By the calculus
	of fractions for this Verdier quotient, there are triangles of $\cd^b(\ce)$
	\[
	P' \to X'' \to X \to P'[1] \mbox{ and } P'' \to X'' \to X' \to P''[1]
	\]
	such that $P'$ and $P''$ belong to $\ch^b(\cp)$. It follows that we have
	\[
	CC_{loc}(X') = M \cdot CC_{loc} (X)
	\]
	for a Laurent monomial $M$ in the $CC_{loc}(P)$, $P\in \cp$.
\end{Rem}

\begin{Ex1} \label{ex:suspensionA2} Let us continue the example from the end of section~\ref{ss:Frobenus categorification}.
	With the notation used there, we have $CC_{loc}(P_i[1])=p_i^{-1}$
	for $i=1,2$. Rotating the two exchange conflations gives triangles
	\[
	P_2\to S_2\to S_1[1]\to P_2[1]
	\qquad\text{and}\qquad
	P_1\to S_1\to S_2[1]\to P_1[1].
	\]
	It follows that
	\[
	CC_{loc}(S_1[1])=x'_1/p_2,
	\qquad
	CC_{loc}(S_2[1])=x_1/p_1.
	\]
	Thus we have a commutative square
	\[
	\begin{tikzcd}
		\cd^b(\ce) \ar[r, "CC_{loc}"] \ar[d, "{[1]}"'] & \ca_Q \ar[d,"\sigma"]\\
		\cd^b(\ce) \ar[r, "CC_{loc}"] & \ca_Q.
	\end{tikzcd}
	\]
	Indeed, the two composites agree on the four indecomposable objects
	of $\ce$ and are multiplicative on direct sums. Both satisfy the
	projective triangle formula in Theorem~\ref{thm:extension}\,a),
	since the shift preserves bounded complexes of projectives.
	The uniqueness argument in that theorem gives equality on
	$\cd^b(\ce)$. Thus the suspension functor $[1]$ of $\cd^b(\ce)$
	categorifies the quasi-cluster automorphism $\sigma$ of $\ca_Q$.
\end{Ex1}

\subsection{Categorification of quasi-cluster morphisms}
Let $(\ce,T)$ be a Frobenius categorification (cf. section~\ref{ss:Frobenus categorification})
of the ice quiver $Q$ of a cluster algebra $\ca$ with invertible coefficients.
Let $\cu$ be the corresponding upper cluster algebra and let
$\Phi\colon\ce\to\cu$ be the normalized cluster character associated with $T$.
Let $\cp\subseteq\ce$ be the full subcategory of projective-injective objects.
Since the values of $\Phi$ on the indecomposable projectives are the frozen
variables, they are units in $\cu$. By part a) of Theorem~\ref{thm:extension},
$\Phi$ extends to a map $CC_{\mathrm{loc}}\colon\cd^b(\ce)\to\cu$,
which we will still denote by $\Phi$. Let $(\ce',T')$ be a Frobenius
categorification of the ice quiver $Q'$ of another cluster algebra $\ca'$
with invertible coefficients. We use the analogous notation
$\cu'$, $\Phi'$ and $\cp'$ on the target side.

\begin{Thm}\label{thm:categorification}
	Let $(\ce,T)$ and $(\ce',T')$ be Frobenius categorifications as above,
	and let $f\colon\cu\to\cu'$ be a ring homomorphism.
	Suppose that there is a triangle functor
	$F\colon\cd^b(\ce)\to\cd^b(\ce')$ with the following properties.
	\begin{itemize}
		\item[1)] $F$ takes $\cp$ into $\ch^b(\cp')$.
		\item[2)] The induced functor
		$\ul F\colon\ul{\ce}\to\ul{\ce'}$ is a triangle equivalence.
		There is a basic cluster-tilting object $U$ of $\ce'$, reachable
		from $T'$, such that $\ul F(\ul T)\cong\ul U$.
		\item[3)] The square on isomorphism classes
		\[
		\begin{tikzcd}
			\add(T) \arrow[d,"F"'] \arrow[r,"\Phi"] & \cu \arrow[d,"f"] \\
			\cd^b(\ce') \arrow[r, "\Phi'"'] & \cu'
		\end{tikzcd}
		\]
		commutes.
	\end{itemize}
	Then $f$ is a quasi-cluster morphism.
\end{Thm}

\begin{Rem}
	We thank Matthew Pressland for suggesting that compatibility with the
	cluster characters be required only on $\add(T)$ in condition~3).
	Pressland applied this criterion to prove the conjecture of
	Muller and Speyer~\cite[Remark~4.7]{muller-speyer2017} that the
	source-labelled and target-labelled cluster structures on an open
	positroid variety quasi-coincide
	\cite[Theorem~6.16]{presslandMuller-Speyer conjecture}.
\end{Rem}

\begin{proof}
	Write $T=\bigoplus_{i=1}^{n}T_i$ and $U=\bigoplus_{a=1}^{n'}U_a$,
	ordered so that the non-projective summands appear first. Label these
	summands so that $\ul F(\ul T_j)\cong\ul U_j$ for $1\leq j\leq r$,
	and set $x_i=\Phi(T_i)$ and $u_a=\Phi'(U_a)$. Let $B$ and $B_U$
	denote the extended exchange matrices at $T$ and $U$, respectively.
	
	We first determine the exponents of $f(x_i)$ in the cluster at $U$.
	By Proposition~\ref{prop:extension}, we may identify $K_{\ce',\cp'}$
	with $K_0^{\mathrm{sp}}(\ce')$, and we write $\kappa(X)$ for the
	class of $X\in\cd^b(\ce')$ under this identification. The classes
	$[U_a]$ form a basis of the free abelian group
	$L_U=K_0^{\mathrm{sp}}(\add(U))\subseteq K_0^{\mathrm{sp}}(\ce')$;
	for $v=\sum_a v_a[U_a]\in L_U$, we put $u^v=\prod_a u_a^{v_a}$. In
	particular, for any bounded complex $P\in\ch^b(\cp')$ of projectives,
	we have $\kappa(P)=\sum_q(-1)^q[P^q]$ and
	$\Phi'(P)=u^{\kappa(P)}\in\P'$ by Theorem~\ref{thm:extension}\,a).
	
	Suppose that $X\in\cd^b(\ce')$ and $V\in\add(U)$ have isomorphic
	images in $\ul{\ce'}$. As in Remark~\ref{rmk:fractions}, choosing a
	common roof in the Verdier quotient yields triangles
	\[
	P_1\longrightarrow Z\longrightarrow X\longrightarrow P_1[1],
	\qquad
	P_2\longrightarrow Z\longrightarrow V\longrightarrow P_2[1]
	\]
	with $P_1,P_2\in\ch^b(\cp')$. The defining relations of
	$K_{\ce',\cp'}$ together with the character formula applied to these
	triangles give
	\[
	\kappa(X)=[V]+\kappa(P_2)-\kappa(P_1)\in L_U,
	\qquad
	\Phi'(X)=\Phi'(V)\frac{\Phi'(P_2)}{\Phi'(P_1)}=u^{\kappa(X)}.
	\]
	Combining this with conditions~1)--3), we obtain
	\[
	\kappa(FT_i)=\sum_{a=1}^{n'}G_{ai}[U_a],
	\qquad
	f(x_i)=\prod_{a=1}^{n'}u_a^{G_{ai}}
	\]
	for an integral $n'\times n$ matrix $G$ whose first $r$ rows form
	$(I_r\ \ 0)$. In particular, $f(\P)\subseteq\P'$, and
	$f(x_j)\in\P' u_j$ for all $1\leq j\leq r$.
	
	We next claim that $GB=B_U$. Fix a mutable index
	$j\in\{1,\ldots,r\}$, and let $\Delta_j^+$ and $\Delta_j^-$ be the
	two exchange conflations at $T_j$, viewed as triangles in
	$\cd^b(\ce)$. Order their middle terms $E_j^+,E_j^-$ according to the
	exchange matrix convention so that
	\[
	[E_j^+]-[E_j^-]=\sum_{i=1}^{n}b_{ij}[T_i]
	\quad\text{in }K_0^{\mathrm{sp}}(\add(T)),
	\]
	and adopt the same convention for the exchange conflations at $U_j$,
	with middle terms $E_{U,j}^+,E_{U,j}^-$. By condition~2) and the
	uniqueness of exchange triangles, the stable image of $F\Delta_j^\pm$
	is isomorphic to that of the corresponding exchange triangle at $U_j$
	in each direction.
	
	For a triangle $\Delta=(A\to D\to C\to A[1])$ in $\cd^b(\ce')$, put
	$\partial\Delta=\kappa(A)-\kappa(D)+\kappa(C)$. Applying
	\cite[Lemma~3.13]{WuRelabeledPositroid2026} to the algebraic
	triangulated category $\cd^b(\ce')$ and its thick subcategory
	$\ch^b(\cp')$ shows that triangles with isomorphic images in the
	stable quotient have the same class $\partial\Delta$. Applying this
	to both exchange triangles $\Delta_j^+$ and $\Delta_j^-$ and
	subtracting the resulting equalities, the contributions of the end
	terms cancel because they appear in opposite order in the two
	triangles, leaving
	\[
	\sum_{i=1}^{n}b_{ij}\kappa(FT_i)
	=\kappa(FE_j^+)-\kappa(FE_j^-)
	=[E_{U,j}^+]-[E_{U,j}^-]
	=\sum_{a=1}^{n'}(B_U)_{aj}[U_a].
	\]
	Comparing coordinates in $L_U$ proves that $GB=B_U$, including the
	frozen rows. Consequently, for every $1\leq j\leq r$,
	\[
	f(\widehat y_j)
	=\prod_{i=1}^{n}f(x_i)^{b_{ij}}
	=\prod_{a=1}^{n'}u_a^{(GB)_{aj}}
	=\widehat y_{U,j},
	\]
	where the first equality is obtained by substituting into
	$\operatorname{Frac}(\cu')$ from the initial Laurent ring, which is
	valid because each $f(x_i)$ is nonzero.
	
	It remains to verify the behaviour of $f$ on cluster variables away
	from the initial seed. Along any labelled mutation path starting at
	$T$, the stable equivalence identifies the successive exchange
	triangles with those along the corresponding path starting at $U$. We
	prove by induction along the path that $f(\Phi(V))=\Phi'(FV)$ for
	every summand $V$ of each cluster on the source path. The base case
	is given by condition~3). For an exchange $V_k\leftrightarrow V_k^*$,
	the middle terms of both exchange triangles lie in the additive
	closure of the current cluster, so the induction hypothesis applies
	to them by multiplicativity on direct sums. Applying the exchange
	formula to the source pair and Theorem~\ref{thm:extension}\,b) to the
	images of its exchange triangles yields
	\[
	f(\Phi(V_k))\,f(\Phi(V_k^*))=\Phi'(FV_k)\,\Phi'(FV_k^*).
	\]
	By Remark~\ref{rmk:fractions}, $\Phi'(FV_k)$ is a target cluster
	variable times a coefficient unit, and is therefore nonzero;
	cancelling it completes the induction step. The same remark now shows
	that $f(\Phi(V))\in\P'\Phi'(V')$ for every reachable mutable summand
	$V$, where $V'$ is the corresponding target summand. In particular,
	$f$ restricts to a homomorphism $\ca\to\ca'$ of cluster algebras with
	invertible coefficients.
	
	Finally, let $\mathrm{sp}$ and $\mathrm{sp}'$ be the source and
	target specializations, and set $h=\mathrm{sp}'\circ f|_{\ca}$. With
	$I$ and $s$ as in Lemma~\ref{lemma:specialization}, we have
	$h(I)=0$ and $h(s)=\prod_{j=1}^{r}\mathrm{sp}'(u_j)\ne0$. For any
	$a\in\ker(\mathrm{sp})$, the lemma gives $s^Na\in I$ for some
	$N\geq0$, and since the target is an integral domain, it follows that
	$h(a)=0$; hence $h$ descends to the coefficient-free cluster algebra.
	Moreover, the first $r$ rows of $GB=B_U$ show that the principal
	exchange matrices at $T$ and $U$ coincide, and the descended map
	sends the initial mutable variables at $T$ to those at $U$. By the
	coefficient-free exchange relations, it therefore agrees along every
	mutation path with the cluster isomorphism determined by these
	labelled seeds. Together with the equalities
	$f(\widehat y_j)=\widehat y_{U,j}$, this verifies all three
	conditions of Section~\ref{ss:quasi-cluster morphisms}.
\end{proof}
\bigskip

We keep the notations introduced at the beginning of this section. In the following variant
of the above theorem, we make a slightly stronger assumption in condition 1) but
do not suppose that we are given a ring homomorphism $\ca \to \ca'$ from
the outset.

\begin{Thm}\label{thm:isocategorification}
	Let $(\ce,T)$ and $(\ce',T')$ be Frobenius categorifications as above,
	with invertible frozen variables. Let
	$F\colon\cd^b(\ce)\to\cd^b(\ce')$ be a triangle functor satisfying
	the following conditions.
	\begin{itemize}
		\item[1)] $F$ takes $\cp$ into $\ch^b(\cp')$, and its restriction to
		$\ch^b(\cp)$ induces an isomorphism
		\[
		K_0(\ch^b(\cp))\xrightarrow{\sim}K_0(\ch^b(\cp')).
		\]
		\item[2)] The induced functor
		$\ul F\colon\ul{\ce}\to\ul{\ce'}$ is a triangle equivalence.
		There is a basic cluster-tilting object $U$ of $\ce'$, reachable
		from $T'$, such that $\ul F(\ul T)\cong\ul U$.
	\end{itemize}
	Then there is a unique quasi-cluster isomorphism
	$f\colon\ca\xrightarrow{\sim}\ca'$ such that
	\[
	f(x_i)=\Phi'(FT_i)\qquad(1\leq i\leq n).
	\]
	Its extension to the ambient fields restricts to an isomorphism
	$\cu\xrightarrow{\sim}\cu'$ of the upper cluster algebras.
\end{Thm}
\begin{proof}
	Write $T=\bigoplus_{i=1}^{n}T_i$, with $T_1,\ldots,T_r$ non-projective
	and $T_{r+1},\ldots,T_n$ projective. Taking the Euler class of a
	bounded complex identifies $K_0(\ch^b(\cp))$ with
	$K_0^{\mathrm{sp}}(\cp)$, and similarly on the target side. Condition~1)
	therefore implies that $\ce$ and $\ce'$ have the same number $s=n-r$
	of indecomposable projectives. Using condition~2), we may write
	$U=\bigoplus_{a=1}^{n}U_a$ with $U_{r+1},\ldots,U_n$ projective and
	with the non-projective summands labelled so that
	\[
	\ul F(\ul T_j)\cong\ul U_j
	\qquad (1\leq j\leq r).
	\]
	Let $D\in\mathrm{GL}_s(\mathbb Z)$ be the matrix representing the
	isomorphism $K_0^{\mathrm{sp}}(\cp)\xrightarrow{\sim}K_0^{\mathrm{sp}}(\cp')$
	induced by condition~1) with respect to the bases
	$[T_{r+1}],\ldots,[T_n]$ and $[U_{r+1}],\ldots,[U_n]$.
	
	Set $u_a=\Phi'(U_a)$ and $v_i=\Phi'(FT_i)$. For each $1\leq j\leq r$,
	Remark~\ref{rmk:fractions} yields $v_j=c_j u_j$ for some $c_j\in\P'$.
	For each $1\leq b\leq s$, the object $FT_{r+b}$ is represented by a
	bounded complex of projectives, so by Theorem~\ref{thm:extension}\,a)
	its character is the frozen Laurent monomial determined by its Euler
	class:
	\[
	v_{r+b}=\prod_{a=1}^{s}u_{r+a}^{D_{ab}}.
	\]
	Consequently, in the coordinates of the extended cluster at $U$, the
	exponent matrix of the substitution $x_i\mapsto v_i$ is
	\[
	G=\begin{pmatrix}I_r&0\\ C&D\end{pmatrix},
	\]
	where the $j$th column of $C$ is the exponent vector of $c_j$. Since
	$D$ is unimodular, so is $G$; hence this substitution extends uniquely
	to a $\mathbb Q$-algebra isomorphism $f$ between the ambient rational
	function fields satisfying $f(\P)=\P'$.
	
	We next compare the extended exchange matrices $\widetilde B_T$ and
	$\widetilde B_U$. Write $\kappa(X)$ for the class of
	$X\in\cd^b(\ce')$ under the identification
	$K_{\ce',\cp'}\cong K_0^{\mathrm{sp}}(\ce')$ of
	Proposition~\ref{prop:extension}. As in the proof of
	Theorem~\ref{thm:categorification}, the coordinates of $\kappa(FT_i)$
	in the basis $[U_1],\ldots,[U_n]$ are given by the exponents of $v_i$,
	so that
	\[
	\kappa(FT_i)=\sum_{a=1}^{n}G_{ai}[U_a].
	\]
	For each mutable index $j\in\{1,\ldots,r\}$, condition~2) identifies
	the stable images of the two exchange triangles at $T_j$ with those
	of the corresponding exchange triangles at $U_j$. Comparing their
	triangle defects and subtracting, just as in the proof of
	Theorem~\ref{thm:categorification}, we obtain
	\[
	\sum_{i=1}^{n}(\widetilde B_T)_{ij}\kappa(FT_i)
	=\sum_{a=1}^{n}(\widetilde B_U)_{aj}[U_a].
	\]
	Equating coordinates in the classes $[U_a]$ yields
	$G\widetilde B_T=\widetilde B_U$. In particular, the principal
	exchange matrices coincide, and
	\[
	f(\widehat y_{j;T})
	=\prod_{i=1}^{n}v_i^{(\widetilde B_T)_{ij}}
	=\prod_{a=1}^{n}u_a^{(\widetilde B_U)_{aj}}
	=\widehat y_{j;U}
	\qquad (1\leq j\leq r).
	\]
	
	Now consider the labelled mutation patterns rooted at $T$ and $U$.
	Because the mutation rule for $\widehat y$-variables depends only on
	the principal exchange matrix and the current $\widehat y$-variables,
	$f$ identifies the $\widehat y$-variables at every pair of
	corresponding seeds. We show by induction along any mutation path
	that $f(x_i(t))=c_i(t)u_i(t)$ for every mutable index $i$, with
	$c_i(t)\in\P'$. At the initial seeds, this holds by construction.
	Assuming it holds at a vertex $t$, let $t'$ be obtained from $t$ by
	mutation in direction $k$, and write the exchange relations as
	\[
	x_k(t)x_k(t')=M^++M^-,
	\qquad
	u_k(t)u_k(t')=N^++N^-,
	\]
	where $M^+/M^-$ and $N^+/N^-$ are the respective $\widehat y$-variables
	in direction $k$. Since the principal exchange matrices agree and
	$f(\P)=\P'$, the induction hypothesis implies that both
	$f(M^+)/N^+$ and $f(M^-)/N^-$ lie in $\P'$. From the equality
	$f(M^+/M^-)=N^+/N^-$, these two monomials coincide; denoting their
	common value by $d_k\in\P'$ and applying $f$ to the source exchange
	relation, we obtain
	\[
	f(x_k(t'))=\frac{d_k}{c_k(t)}\,u_k(t'),
	\]
	completing the induction step.
	
	Since $U$ is reachable from $T'$, mutation paths starting at $U$
	exhaust the target mutation pattern, so every target cluster variable
	is the image under $f$ of a source cluster variable up to a
	coefficient unit. Together with $f(\P)=\P'$, this implies
	$f(\ca)=\ca'$. By the same reason, $f$ identifies the Laurent
	rings of corresponding seeds, and taking their intersections yields
	$f(\cu)=\cu'$.
	
	Finally, specializing the frozen variables to $1$ yields an induced
	map identifying the labelled coefficient-free seeds. Moreover, the
	agreement of the $\widehat y$-variables at every pair of corresponding
	seeds gives the same property for $f^{-1}$, so both $f$ and $f^{-1}$
	are quasi-cluster morphisms. Because the extended initial cluster
	generates the ambient field, the prescribed values
	$f(x_i)=\Phi'(FT_i)$ determine $f$ uniquely.
\end{proof}

\begin{Ex1}
	We generalize Example~\ref{ex:suspensionA2}. Assume that the frozen
	variables are invertible and that $\Si\ul T$ is reachable from
	$\ul T$ in $\ul{\ce}$.
	
	Consider the shift functor $[1]$ on $\cd^b(\ce)$. It preserves
	$\ch^b(\cp)$ and induces $-\mathrm{id}$ on $K_0(\ch^b(\cp))$,
	since $[P[1]]=-[P]$ for every $P\in\cp$. On $\ul{\ce}$ it induces
	the suspension $\Si$. Performing at $T$ a sequence of mutations
	taking $\ul T$ to $\Si\ul T$ gives a reachable cluster-tilting
	object $U$ of $\ce$ with $\ul U\cong\Si\ul T$.
	Thus Theorem~\ref{thm:isocategorification} gives a unique
	quasi-cluster automorphism $\mathrm{DT}$ of $\ca$ satisfying
	\[
	\mathrm{DT}(x_i)=CC_{\mathrm{loc}}(T_i[1])
	\qquad (1\leq i\leq n).
	\]
	
	To compute these images, choose for each $i$ a conflation in $\ce$
	\[
	T_i\rightarrowtail I_i\twoheadrightarrow V_i
	\qquad\text{with } I_i\in\cp.
	\]
	Then $\ul V_i\cong\Si\ul T_i$. Rotating the associated triangle in
	$\cd^b(\ce)$ gives
	\[
	I_i\longrightarrow V_i\longrightarrow T_i[1]\longrightarrow I_i[1].
	\]
	By Theorem~\ref{thm:extension}\,a),
	$CC(V_i)=CC(I_i)\,CC_{\mathrm{loc}}(T_i[1])$. Hence
	\[
	\mathrm{DT}(x_i)=\frac{CC(V_i)}{CC(I_i)}.
	\]
	The left-hand side shows that the ratio is independent of the choice
	of conflation. If $T_i$ is frozen, take $I_i=T_i$ and $V_i=0$.
	This gives $\mathrm{DT}(x_i)=x_i^{-1}$.
	We call this automorphism the \emph{twist}, or the
	\emph{Donaldson--Thomas transformation}, of $\ca$.
\end{Ex1}

\subsection*{Acknowledgments} We thank Matthew Pressland for
inspiring discussions and encouragement. We are grateful to
Yilin Wu for comments on a preliminary version of
this appendix.

\appendix

\bibliographystyle{plain}

\end{document}